\documentclass[letter,11pt]{article}
 
\usepackage{amsmath,amssymb,amsthm}
\usepackage{xspace}
\usepackage[english]{babel}
\usepackage{hyperref}
\usepackage{color}
\usepackage[T1]{fontenc}
\usepackage[utf8]{inputenc}
\usepackage[margin=1in]{geometry}
\usepackage{bm}
\usepackage{float}
\floatstyle{boxed}
\newfloat{pseudocode}{thb}{pseudo}
\usepackage{enumitem}
\usepackage{graphicx}
\usepackage{caption}
\def\showcomments{True}
\usepackage{amsmath,amsfonts,amsthm}
\usepackage{mathtools}

\usepackage{imakeidx}
\usepackage{multirow}
\usepackage{booktabs}

\usepackage{algorithm}
\usepackage[noend]{algpseudocode}

\usepackage{thmtools,thm-restate}

\usepackage{todonotes}

\newtheorem{definition}{Definition}[section]
\newtheorem{lemma}[definition]{Lemma}

\newtheorem{fact}[definition]{Fact}
\newtheorem{theorem}[definition]{Theorem}

\newtheorem{cor}[definition]{Corollary}

\newtheorem*{theorem*}{Theorem}
\newtheorem*{lemma*}{Lemma}

\renewcommand{\le}{\leqslant}
\renewcommand{\leq}{\leqslant}
\renewcommand{\ge}{\geqslant}
\renewcommand{\geq}{\geqslant}

\ifx\showcomments\undefined

\else

\fi

\newcommand{\bigO}{\mathcal{O}}

\newcommand{\RN}{\mathrm{RN}}
\newcommand{\LN}{\mathrm{LN}}
\newcommand{\AN}{\mathrm{LC}}
\newcommand{\LC}{\mathrm{LC}}
\newcommand{\Prc}[1]{\mathbf{Pr} \left( #1 \right)}
\newcommand{\Prob}[2]{\mathbf{Pr}_{#1}\left(#2\right)}
\newcommand{\Expcc}[1]{\mathbf{E} \left[ #1 \right]}
\newcommand{\Varcc}[1]{\mathbf{Var} \left[ #1 \right]}

\newcommand{\Erdos}{Erd\H{o}s-R\'enyi}

\newcommand{\IC}{$\mathrm{IC}$\xspace}

\newcommand{\SIR}{$\mathrm{SIR}$\xspace}
\newcommand{\SIRimmtrunc}{$L$-\textsc{visit}\xspace}

\newcommand{\SIRimmtruncpar}{\textsc{parallel}\,$L$-\textsc{visit}\xspace}
\newcommand{\percgraph}{{G}_p}

\newcommand{\SWG}{$\mathcal{SWG}(n,q)$\xspace}
\newcommand{\SWGm}{\mathcal{SWG}}

\newcommand{\SWGreg}{3-$\mathcal{SWG}(n)$\xspace}

\newcommand{\dist}{\textrm{dist}}
\newcommand{\good}{\emph{free}\xspace}

\newcommand{\enqueue}{\texttt{enqueue}}
\newcommand{\dequeue}{\texttt{dequeue}}

\newcommand{\condpar}{|Q_{t-1}|=i,|R_{t-1}|=r}

\newcommand{\Bin}{\mathrm{Bin}}

\newcommand{\True}{\texttt{True}}
\newcommand{\False}{\texttt{False}}

\algnewcommand\algorithmiccase{\textbf{case}}
\algdef{SE}[CASE]{Case}{EndCase}[1]{\algorithmiccase\ #1}{\algorithmicend\ \algorithmiccase}%
\algtext*{EndCase}%

\title{Percolation and Epidemic Processes \\ in \\ 
One-Dimensional Small-World Networks\thanks{LT's work  on this project has received funding from the
European Research Council (ERC) under the European Union's Horizon 2020
research and innovation programme (grant agreement No. 834861).  LB's work on
this project was partially supported by the ERC Advanced Grant 788893 AMDROMA,
the EC H2020RIA project ``SoBigData++'' (871042), the MIUR PRIN project
ALGADIMAR. AC's and FP's work on this project was partially supported by the
University of Rome ``Tor Vergata'' under research program ``Beyond Borders''
project ALBLOTECH (grant no. E89C20000620005)}}

\author{Luca Becchetti\\
    {\footnotesize{}Sapienza Università di Roma}\\
    {\footnotesize{} Rome, Italy}\\
    {\footnotesize{}\texttt{becchetti@dis.uniroma1.it}} 
    \and Andrea Clementi \\ 
    {\footnotesize{}Università di Roma Tor Vergata}\\
    {\footnotesize{} Rome, Italy}\\
    {\footnotesize{}\texttt{clementi@mat.uniroma2.it}} 
    \and 
    Riccardo Denni\\
    {\footnotesize{}Sapienza Università di Roma}\\
    {\footnotesize{} Rome, Italy}\\
    {\footnotesize{}\texttt{denni@diag.uniroma1.it}}
    \and Francesco Pasquale \\
    {\footnotesize{}Università di Roma Tor Vergata}\\
    {\footnotesize{} Rome, Italy}\\
    {\footnotesize{}\texttt{pasquale@mat.uniroma2.it}} 
    \and Luca Trevisan \\
    {\footnotesize{}Università Bocconi}\\
    {\footnotesize{} Milan, Italy}\\
    {\footnotesize{}\texttt{l.trevisan@unibocconi.it}} 
    \and Isabella Ziccardi\\
    {\footnotesize{}Università dell'Aquila}\\
    {\footnotesize{}L'Aquila, Italy}\\
    {\footnotesize{}\texttt{isabella.ziccardi@graduate.univaq.it}} 
}
\date{}

\graphicspath{{Images/}}

\begin{document}

\maketitle
%\newpage
%\setcounter{page}{1}
\begin{abstract}

We obtain tight thresholds for bond percolation on one-dimensional small-world graphs, and apply such results to obtain tight thresholds for the \emph{Independent Cascade} process and the \emph{Reed-Frost} process in such graphs.

These are the first fully rigorous results establishing a  phase transition for bond percolation and SIR epidemic processes  in small-world graphs. Although one-dimensional small-world graphs are an idealized and unrealistic network model, a number of realistic qualitative epidemiological phenomena emerge from our analysis, including the epidemic spread     through a sequence of local outbreaks, the danger posed by random connections, and the effect of super-spreader events. 
\end{abstract}


%\medskip \noindent \textbf{Keywords:} Epidemic models, Independent Cascade,
%Small-World Graphs, Percolation, Branching Processes.

\setcounter{page}{1}

\section{Introduction} \label{sec:newintro}

Given a graph $G=(V,E)$ and a bond percolation probability $p$, the bond percolation process is to subsample a random graph $G_p= (V,E_p)$ by independently choosing each edge of $G$ to be included in $E_p$ with probability $p$ and to be omitted with probability $1-p$. We will call $G_p$ the {\em percolation graph} of $G$. The main questions that are studied about this process are whether $G_p$ is likely to contain a large connected component, and what are the typical distances of reachable nodes in $G_p$. 

The study of percolation originates in mathematical physics, where it has often been studied in the setting of infinite graphs, for example infinite lattices and infinite trees \cite{KetAl80,NewmanS86,VK71}. The study of percolation on finite graphs is of interest in computer science,  because of its relation, or even equivalence, to a number of fundamental problems in network analysis \cite{KKT15,EK10,LRSV14,ABS04} and in  distributed and parallel computing \cite{CPGE19,KNT94}.

For example, the percolation process arises in the study of {\em network reliability} in the presence of independent link failures \cite{KNT94,KL19}; in this case one is typically interested in {\em inverse problems}, such as designing networks that have a high probability of having a large connected component for a given edge failure probability $1-p$.

This paper is motivated by the equivalence of the percolation process with the {\em Independent Cascade} process, which models the spread of information in networks \cite{KKT15,EK10}, and with the {\em Reed-Frost} process of \emph{Susceptible-Infectious-Recovered} (\emph{SIR}) epidemic spreading \cite{CWLC13,VespietAl15}.

In a SIR epidemiological process, every person, at any given time, is in one of three possible states: either susceptible (S) to the infection, or  actively infectious and able to spread the infection (I), or recovered (R) from the illness, and immune to it.

In a network SIR model, we represent people as nodes of a graph, and contacts between people as edges, and we have a probability $p$ that each contact between an infectious person and a susceptible one transmits the infection. The Reed-Frost process, which is the simplest SIR network model,   proceeds through synchronous time steps, the infectious state lasts for only one time step, and the graph does not change with time.

The Information Cascade process is meant to model information spreading in a social network, but it is essentially equivalent to the Reed-Frost process.\footnote{The main difference is that Information Cascade allows the probability of ``transmission'' along an edge $(u,v)$ to be a quantity $p_{(u,v)}$, but this generalization would also make sense and be well defined in the Reed-Frost model and in the percolation process. The case in which all the probabilities are equal is called the {\em homogenous} case.}

If we run the Reed-Frost process on a graph $G=(V,E)$ with an initial set $I_0$ and with a probability $p$ that each contact between an infectious and a susceptible person leads to transmission, then the resulting is equivalent to percolation on the graph $G$ with parameter $p$ in the following sense: the set of vertices reachable from $I_0$ in the percolation graph $G_p$ has the same distribution as the set of nodes that are recovered at the end of the Reed-Frost process in $G$ with $I_0$ as the initial set of infected nodes. Furthermore, the set of  nodes infected in the first $t$ steps (that is, the union of infectious and recovered nodes at time $t$) has the same distribution as the set of nodes reachable in the percolation graph $G_p$ from $I_0$ in at most $t$ steps\footnote{A detailed description of   this equivalence is given in Appendix \ref{ssec:equi}.}.

Information Cascade and Reed-Frost processes on networks are able to capture a number of features of real-world epidemics, such as the fact that people typically have a small set of close contacts with whom they interact frequently, and more rare interactions with people outside this group, that different groups of people have different social habits that lead to different patterns of transmissions, that outbreaks start in a localized way and then spread out, and so on. Complex models that capture all these features typically have a large number of tunable parameters, that have to be carefully estimated, and have a behavior that defies rigorous analysis and that can be studied only via simulations.

In this work we are interested in finding the simplest model, having few parameters and defining a simple process, in which we could see the emergence of complex phenomena.

\subsubsection*{One-dimensional small-world graphs}
% We choose to analyze the Reed-Frost process, which is the simplest network epidemic process, on {\em one-dimensional small-world} graphs, which are the simplest generative model of networks in which there is a meaningful distinction between local connection (corresponding to close contacts such as family and coworkers) and long-range connections (corresponding to occasional contacts such as being seated next to each other in a restaurant or a train).
% =======
We choose to analyze the Reed-Frost process on {\em one-dimensional small-world} graphs, which is a fundamental generative model of networks in which there is a  distinction between local connection (corresponding to close contacts such as family and coworkers) and long-range connections (corresponding to occasional contacts such as being seated next to each other in a restaurant or a train).

Small-world graphs are a class of probabilistic generative models for graphs introduced by Watts and Strogatz \cite{watts1998collective}, which are obtained by overlaying a low-dimensional lattice with additional random edges. A one-dimensional small-world graph is a cycle overlayed with additional random edges. In the original works of Watts and Strogatz a one-dimensional small-world network is obtained by starting from a cycle, adding edges between any pair of nodes at distance at most $k$ (one of the parameters of the construction) along the cycle, then selecting a random subset of edges (the density of this subset is another parameter of the construction) and re-routing them, where the operation of re-routing an edge is to re-assign one of the endpoints of the edge to a random vertex. 

% Because of our interest in studying the simplest models, with the fewest number of parameters, in which we can observe complex emergent behavior, we consider a simplified generative model in which the distribution of one-dimensional small-world graphs with parameter $q$ on $n$ vertices is simply the union of a cycle with $n$ vertices with an \Erdos\ random graph $\mathcal{G}_{n,q}$, in which edges are sampled independently, and each pair of nodes has probability $q$ of being an edge. 

Because of our interest in studying the most basic models, with the fewest number of parameters, in which we can observe complex emergent behavior, we consider the following  simplified generative model which was   introduced in \cite{MCN00} and often  adopted in different network applications   \cite{GGT03,MN00b,newman1999scaling}:  the distribution of \emph{one-dimensional small-world graphs} with parameter $q$ on $n$ vertices is just the union of a cycle with $n$ vertices with an \Erdos\ random graph $\mathcal{G}_{n,q}$, in which edges are sampled independently and each pair of nodes has probability $q$ of being an edge.

We will focus on the sparse case in which $q=c/n$, with $c$ constant, so that the overall graph has average degree $c+2$ and maximum degree that is, with high probability, $O(\log n / \log \log n)$. As we will see, we are able to determine, for every value of $c$, an exact threshold for the critical probability of transmission and to establish that, above the threshold, the epidemic spreads with a realistic pattern of a number of localized outbreaks that progressively become more numerous.

We are also interested in modeling, again with the simplest possible model and with the fewest parameters, the phenomenon of {\em superspreading}, encountered both in practice and in simulations of more complex models. This is the phenomenon by which the spread of an epidemic is disproportionately affected by rare events in which an infectious person contacts a large number of susceptible ones. To this end, we also consider a generative model of small-world 1-dimensional graphs obtained as the union of a cycle with a random perfect matching. This generative model has several statistical properties in common with the $c=1$ instantiation of the above generative model: the marginal distribution of each edge is the same, and edges are independent in one case and have low correlation in the random matching model. The only difference is the degree distribution, which is somewhat irregular (but with a rapidly decreasing exponential tail) in one case and essentially 3-regular in the second case. As we will see, we are able to determine an exact threshold for this latter model as well, and it notably differs from the previous model.

Before proceeding with a statement of our results, we highlight for future reference the definitions of our generative models.
   
\begin{definition}[$1$-Dimensional Small-World Graphs - \SWG ]\label{def:small-world}
For every  $n \geq 3$ and $0 \leq q \leq 1$, the distribution \SWG\ is sampled
by generating a one-dimensional small-world graph $G=(V,E)$, where $|V|=n$, $E
= E_1 \cup E_2$, $(V,E_1)$ is a cycle, and $E_2$ is the set of random edges, called \emph{bridges}, of
an \Erdos\ random graph ${\mathcal G}_{n,q}$.
\end{definition}

%The above small-world graph model generates sparse graphs with unbounded
%degree. The maximum  degree of a graph sampled from $\SWGm(n,1/n)$, like the
%maximum degree of a ${\mathcal G}_{n,1/n}$ random graph, is on average
%$\Theta(\log n/\log\log n)$.

\begin{definition}[$3$-regular $1$-Dimensional  Small-World Graphs - \SWGreg ]\label{def:small-world-3reg}
For every  even $n \geq 4$, the distribution  \SWGreg\ is sampled by generating
a one-dimensional small-world graph  $G=(V, E)$, where $|V|=n$, $E = E_1 \cup
E_2$, $(V,E_1)$ is a cycle, and $E_2$ is the set of edges, called \emph{bridges},  of a uniformly chosen
perfect matching on $V$.
\end{definition}

In the definition of \SWGreg, we allow edges of the perfect matching to belong
to $E_1$. If this happens, only edges in $E_2-E_1$ are called bridges. The
graphs sampled from \SWGreg have maximum degree 3, and every node has degree 3
or $2$. On average, only $\bigO(1)$ nodes have degree 2. This is why, with a
slight abuse of terminology, we refer to these graphs as being ``3-regular.''

\section{Our Contribution } \label{sec:ourresults}
%We study      Watts-Strogatz models $\SWGm$ where the graphs $G=(V, E_1 \cup E_2)$  where  $E_1$ is a cycle (representing the local connections) and $E_2$ is the subset of    random edges, which we call {\em bridges}.   As for the latter, we consider the    two alternatives:

%We observe that  each node in the above small-world graph $G$ has always degree bounded by $3$\footnote{If we let $G$ to be a multigraph, then the graph is 3-regular.}.

%\begin{itemize} 

%\item The subset $E_2$ of bridges   form a random perfect matching over the nodes of the
%cycle (this will assume that the number  $n$ of nodes is even), giving rise to
%a graph of maximum degree 3.
%We denote the resulting distribution of graphs as \SWGreg;

%\item  The  bridges in   $E_2$ are  selected as in an \Erdos\ random graph
%$\mathcal{G}_{n,q}$. We denote the resulting distribution of graphs as \SWG. We
%will consider the case $q=c/n$ for any fixed constant $c > 0$, noticing that
%the case $c=1$ yields a number of  bridge edges that is concentrated around
%$n/2$, similarly to the first case. 

%\end{itemize}

\subsection{Tight thresholds for bond percolation} \label{ssec::ourres}
Our main results are to establish sharp thresholds for the critical percolation probability $p$ in both
models. In particular, we are interested in fully rigorous analysis that hold in high concentration (i.e., with high probability), avoiding mean-field
approximations or approximations that treat certain  correlated events as independent, which are common in the analysis of complex networks in
the physics literature. While such approximations are necessary when dealing
with otherwise intractable problems, they can fail to capture subtle
differences between models. For example, for $q = 1/n$, the marginal distributions of 
bridge edges are the same in the two models above, while correlations between edges are non-existing in the \SWG\ model
and very small in the \SWGreg model. Yet, though the two models have similar
expected behaviors and are good approximations of each other, our rigorous
analysis shows that the two models exhibit notably different thresholds.

%Our main results are to establish sharp thresholds for the critical percolation probability $p$ in both

 As for the the \SWG\ model, we show the following threshold behaviour of the bond-percolation process.

\begin{theorem}[Percolation on the \SWG model] 
\label{thm:intro-erdos-2}
Let $V$ be a set of $n$ vertices and $p>0$ be a bond percolation  probability. For any constant
$c>0$,  sample a graph $G=(V, E_1 \cup E_2)$ from the $\SWGm(n,c/n)$ distribution, and consider the percolation graph $G_p$. For any constant $\varepsilon>0$:
\begin{enumerate}
\item If $p>\frac{\sqrt{c^{2} + 6c +1}-c-1}{2c} + \varepsilon$,  w.h.p.\footnote{As usual, we say that an event
$\mathcal{E}_n$ occurs \emph{with high probability} if $\Prc{ \mathcal{E}_n}
\geq 1-(1/n)^{\Omega(1)}$.}
      a subset of nodes of size $\Omega_\varepsilon(n)$ exists 
      that induces   a   subgraph  of $G_p$ having  diameter $\bigO_{\varepsilon}(\log n)$;
\item If  $p<\frac{\sqrt{c^{2} + 6c +1}-c-1}{2c} - \varepsilon$,  w.h.p.  all the connected components of  $G_p$ have size $\bigO_{\varepsilon}(\log n)$.
\end{enumerate}
\end{theorem}

Some remarks are in order. In the theorem above, probabilities are taken
both over the randomness in the generation of the graph $G$ and over the
randomness of the percolation process. 
%We discuss interesting  ``information-spreading'' consequences of  the  theorem above   in %Subsection \ref{ssec:consepid}, while
We highlight  the sharp result
  on the $\SWGm(n,c/n)$ model for the case $c = 1$:
similarly to the regular \SWGreg model, each node here has one bridge
edge in average, and the obtained critical value for the percolation probability
$p$ turns out to be $\sqrt{2}-1$. An analysis of the critical value for
the \SWGreg model is given by the next two results, while a detailed comparison of the two
models is provided in Subsection \ref{ssec:consepid}, after Theorem \ref{thm:main_matching}.
%%%% after Theorem~\ref{th:reg.upper} below.

\iffalse 
The lower bound in Claim 1 of
the theorem are proved by studying a BFS-like visit of the percolation graph.
Starting from an arbitrary node, we first
establish   that once the visit has
at least $\beta_\varepsilon \log n$ nodes in the queue, it has a high probability
of finding $\Omega_\varepsilon (n)$ nodes within distance $\bigO_\varepsilon (\log
n)$ from the nodes visited up to that point.  To prove the first part of   Claim 1, we then show that there is at least constant probability that, in the first
$\bigO_\varepsilon (\log n)$ steps of the visit, the queue grows to a size at least
$\beta_\varepsilon \log n$. 
\fi

\begin{theorem}[Percolation on the  \SWGreg model]\label{thm:main_matching}
Let $V$ be a set of $n$ vertices  and $p>0$ be a bond percolation  probability. Sample a graph $G=(V, E_1 \cup E_2)$ from the \SWGreg distribution,  and consider the percolation graph $G_p$. For any constant $\varepsilon>0$:

\begin{enumerate}
\item If $p>1/2+\varepsilon$, w.h.p.  a subset of nodes of size $\Omega_\varepsilon(n)$ exists   that induces  a  connected  subgraph (i.e. a \emph{giant} connected component)  of $G_p$;

%\item Let $p>1/2+\varepsilon$ and  $|I_0|\geq\beta_\varepsilon\log n$ for a sufficiently large constant $\beta_\varepsilon$ (that depends on $\varepsilon$ but not on $n$). Then w.h.p., there are  $\Omega_\varepsilon(n)$ connected to $I_0$ in $G_p$;

\item If $p<1/2-\varepsilon$, w.h.p. all the connected components of  $G_p$ have size $\bigO_{\varepsilon}(\log n)$.
\end{enumerate}
\end{theorem}

%In the above theorem, the first two  statements hold with constant probability over
%the randomness of the generation of $G$ and the randomness of the contagion
%process. 

Also in the above theorem, the probabilities are taken over the randomness of $G$ and  over the randomness of the percolation process process. The second claim is a special case of the following more general result of ours.

\begin{theorem}[Percolation on bounded-degree graphs] 
\label{th:reg.upper}
Let $G=(V,E)$ be a graph of maximum degree $d$, $\varepsilon >0$ be an
arbitrary positive number, $p < (1-\varepsilon)/(d-1)$  be a bond percolation  probability, and $I_0$ be a subset of $V$.  
  Consider the percolation graph $G_p$.  Then,
w.h.p.,  all the connected components of $G_p$ have size $\bigO_{\varepsilon}(\log n)$.
\end{theorem}

An overall view of our analysis, leading to all  the theorems above,  is provided in Section~\ref{sec:Proof Ideas}, while in the next subsection, we describe the main consequences of our analysis  for  the Independent-Cascade  protocol on the considered small-world   models.

\subsection{Applications to   epidemic  processes} \label{ssec:consepid}

As remarked in Section \ref{sec:newintro}, bond percolation with percolation probability   $p$ is equivalent to the   Reed-Frost   process (for short, RF process) with transmission probability $p$. Informally speaking, the nodes at hop-distance $t$ in the percolation graph $G_p$, from any fixed  source subset, are distributed exactly as those that will be informed (and activated) at  time $t$, according to the  RF process\footnote{We   remind that a  detailed description of   this equivalence is given in Appendix \ref{ssec:equi}.}.

In this setting, our analysis and results, we described in Subsection \ref{ssec::ourres}, have the following important consequences.

\begin{theorem}[The RF process  on the \SWG model] \label{thm:gen_overthershold}
Let $V$ be a set of $n$ vertices, $I_0 \subseteq V$ be a set of source nodes, and $p>0$ a constant probability. For any constant
$c>0$,  sample a graph $G=(V, E_1 \cup E_2)$ from the $\SWGm(n,c/n)$ distribution, and run the RF process with transmission probability $p$ over $G$ from $I_0$. For every $\varepsilon>0$, we have the following:

\begin{enumerate}
\item  If $p > \frac{\sqrt{c^{2} + 6c +1}-c-1}{2c} + \varepsilon$, with
probability $\Omega_\varepsilon(1)$ a subset of  $\Omega_\varepsilon(n)$ nodes  will be infectious within time $\bigO_\varepsilon(\log n)$, even if $|I_0| = 1$. Moreover, if  
$|I_0|\geq\beta_\varepsilon\log n$ for a sufficiently large constant
$\beta_\varepsilon$ (that depends only on $\varepsilon$), 
then the above event occurs  \emph{w.h.p.};
    
\item If $p<\frac{\sqrt{c^{2} + 6c +1}-c-1}{2c} - \varepsilon$,  w.h.p. the process will stop within $\bigO_\varepsilon(\log n)$ time steps, and the number of recovered nodes at the end of the process will be
$\bigO_\varepsilon(|I_0|\log n)$.
\end{enumerate}
\end{theorem}

As for the \SWGreg model, we get the following results for the Reed-Frost process.

\begin{theorem}[The RF process on the   \SWGreg model]\label{thm:main_matching_rf}
Let $V$ be a set of $n$ vertices, $I_0 \subseteq V$ be a set of source nodes, and $p>0$ be a bond percolation  probability. Sample a graph $G=(V, E_1 \cup E_2)$ from the \SWGreg distribution,  and run the RF protocol with transmission-probability $p$ over $G$ from $I_0$. For every $\varepsilon>0$, we have the following:

\begin{enumerate}
\item If $p>1/2+\varepsilon$, with probability $\Omega_{\varepsilon}(1)$, a subset of  $\Omega_\varepsilon(n)$ nodes  will be infectious within time $\bigO_\varepsilon( n)$, even if $|I_0| = 1$. Moreover, if  
$|I_0|\geq\beta_\varepsilon\log n$ for a sufficiently large constant
$\beta_\varepsilon$ (that depends on $\varepsilon$ but not on $n$), 
then the above event occurs  \emph{w.h.p.};

%\item Let $p>1/2+\varepsilon$ and  $|I_0|\geq\beta_\varepsilon\log n$ for a sufficiently large constant $\beta_\varepsilon$ (that depends on $\varepsilon$ but not on $n$). Then w.h.p., there are  $\Omega_\varepsilon(n)$ connected to $I_0$ in $G_p$;

\item If $p<\frac{\sqrt{c^{2} + 6c +1}-c-1}{2c} - \varepsilon$, then, w.h.p.,
  the process will stop within $O_\varepsilon(\log n)$ time steps, and the number of recovered nodes at the end of the process will be  $O_\varepsilon(|I_0|\log n)$.
\end{enumerate}
\end{theorem}

We notice that the  first claim of each of the above two theorems, concerning the multi-source case, i.e. the case $|I_0| \geq \beta \log n$), are not  direct consequences of  (the corresponding first claims of)  Theorems \ref{thm:intro-erdos-2} and \ref{thm:main_matching}:  although each element of $I_0$ has constant probability of belonging to the ``giant component'' of the graph $G_p$, these events are not independent, and so it is not immediate that, when $|I_0|$ is of the order of $\log n$, at least an element of $I_0$ belongs to the giant component with high probability. Such claims    instead are non-trivial consequences of  our technical analysis.

 On the other hand, the second claims of the above two theorems are   simple consequences of   the corresponding  claims  of Theorems \ref{thm:intro-erdos-2} and \ref{thm:main_matching}.
As for general bounded-degree graphs, from Theorem \ref{th:reg.upper}, we can recover an  upper bound on the critical value of $p$ for the RF process
equivalent to that of Claim 2 of  Theorem \ref{th:reg.upper}  (we omit here the formal statement).

%%%%%%%%

From a topological point of view, because of a mix of local and random edges, epidemic spreading in the above
models proceeds as a sequence of outbreaks, a process that is made explicit in
our rigorous analysis, where we see the emergence of two qualitative
phenomena that are present in real-world epidemic spreading. 

One is that the presence of long-distance random connections has a stronger
effect on epidemic spreading than local connections, that, in   epidemic scenarios, might  motivate lockdown
measures that shut down long-distance connections.  This can be seen,
quantitatively, in the fact that the critical probability in a cycle is $p =
1$, corresponding to a critical \emph{basic reproduction number}\footnote{The quantity $R_0$ in a SIR process is the expected number of people that an infectious person transmits the infection to, if all the contacts of that person are susceptible. In the percolation view of the process, it is the average degree of the percolation graph $G_p$.} $R_0$ equal to 2.
On the other hand, the
presence of random matching edges or random $\mathcal{G}_{n,c/n}$ edges in the
setting $c = 1$ defines networks in which the critical $R_0$ is, respectively,
$1.5$ and $3\cdot (\sqrt 2 -1) \approx 1.24$, meaning that notably fewer local
infections can lead to large-scale contagion on a global scale.

The other phenomenon is that  the irregular networks of the $\SWGm(n,c/n)$
model in the case $c=1$ show a significantly lower critical probability, i.e.
$\sqrt 2 -1 \approx .41$, than the  critical value $.5$ of the nearly regular networks of
the \SWGreg model, though they have the same number of edges (up to lower order
terms) and very similar distributions. As a further evidence of this
phenomenon, we remark the scenario yielded by the random irregular networks
sampled from the $\SWGm(n,c/n)$ distribution with $c$ even smaller than $1$:
for instance, the setting $c = .7$, though yielding a much sparser topology
than the \SWGreg networks, has a critical probability which is still smaller
than $.5$. Moreover, this significant difference between the $\SWGm(n,c/n)$
model and the regular \SWGreg one holds even for more dense regimes. In detail,
Theorem~\ref{th:reg.upper} implies that the  almost-regular version of $\SWGm$
in which $c$ independent random matchings are added to the ring of $n$ nodes
has a critical probability at least $1/(c+1)$. Then, simple calculus shows that
the critical probability given by Theorem~\ref{thm:gen_overthershold} for the
$\SWGm(n,c/n)$ model is smaller than $1/(c+1)$, for any choice  of the
density parameter $c$.

The most significant difference between the two distributions above  is the
presence of a small number of high-degree vertices in $\SWGm(n,c/n)$,
suggesting that even a small number of ``super-spreader'' nodes can have major
global consequences.

\subsection{Extensions of our results for epidemic models} \label{ssec:extenandgen}

\noindent
\textbf{Non-homogenous transmission probability.}
While keeping our focus on the rigorous analysis of simplified models that still capture important emergent phenomena, we remark that our techniques  allow extentions of our results  to a natural non-homogenous bond-percolation process on small-world graphs, in which local edges percolate with probability $p_1$, while bridges percolates with probability $p_2$: our analysis in fact  
keeps the role of the two type of connections above  well separated from each other. 
We are inspired, for instance,  by epidemic scenarios in which the chances for any node to get infected/informed  by a local tie are significantly higher than those from sporadic, long ties. 

In this non-homogeneous setting, for the  \SWG model with $q= c/n$ for some absolute constant $c>0$, we can   prove  that, w.h.p.,  the Independent-Cascade protocol reaches $\Omega(n)$ nodes within $\bigO(\log n)$ time\footnote{The formal statement is similar to that for the homogeneous case in Theorem \ref{thm:gen_overthershold} and is given in  Appendix \ref{sec:further}.}   iff the following condition on the three   parameters of the process is satisfied
\begin{equation*}
  p_1 + c \cdot p_1 p_2 + c \cdot p_2 > 1 \, .
%\label{eq:threshold_two_prob}
\end{equation*}

Some remarks are in order. 
In the case  $c=1$, the  formula above shows a perfect symmetry in the role of the two bond probabilities $p_1$ and $p_2$. In a graph sampled from $\SWGm(n,1/n)$,  however,  the overall number of local ties (i.e. ring edges) is $n$, while the number of bridges is highly concentrated  on   $n/2$ (it is w.h.p. $\leqslant n/2 + \sqrt{n\log n}$). This means that a public-health intervention aimed at reducing transmission has to suppress twice as much local transmissions in order to obtain the same effect of reducing by a certain amount the number of long-range transmissions.
If we consider the case $c=2$,  in which the number of bridges is about equal to the number of local edges, we see that the impact of a change in $p_2$ weighs roughly twice as much as a corresponding change $p_1$. 

So, even in the fairly unrealistic one-dimensional small-world model, it is possible to recover analytical evidences for the effectiveness of public-health measures that block or limit long-range mobility and super-events (such as football matches, international concerts, etc.).
The generalization to non-homogenous tramsmission probabilities is  provided   in Appendix \ref{sec:further}.

\smallskip
\noindent 
\textbf{Longer node activity and incubation.}
Natural generalizations of the setting considered in this work include models in which i) the interval of time
during which a node is active (i.e., the \emph{activity period}) follows some (possibly node-dependent) distribution and/or ii) once infected, a node only becomes active after an \emph{incubation} period, whose duration again follows some distribution.
While the introduction of activity periods following general distributions may considerably complicate the analysis, our approach rather straightforwardly extends to
two interesting cases, in which the incubation period of each node is a random variable (as long as incubation periods are independent) and/or the activity period of a node consists of $k$
consecutive units of time, with $k$ a fixed constant. 
This generalized model with random, node-dependent incubation periods corresponds to a discrete,
synchronous version of the $\mathrm{SEIR}$ model,\footnote{With respect to \SIR, for each node we have a fourth, \emph{Exposed} state, corresponding to the incubation period of a node.} which 
was recently considered as a model of the COVID-19 outbreak in Wuhan \cite{lin2020conceptual}.
These extensions are formalized and discussed in Appendix \ref{sec:further}.

\subsubsection*{Roadmap}
%In Section \ref{sec:ourresults}, we formalize our   
%main results and   discuss their  
%consequences for  epidemic IC protocols. 

 Section \ref{sec:Proof Ideas}    gives  an overall description of the main ideas and technical results behind our analysis of bond-percolation in one-dimensional small-world graphs. 
 While  the most-related, important  previous   contributions   have been already mentioned in the previous sections, further   related work is summarized in Section \ref{sec:related_new} which concludes the body of the paper.

The appendix of the paper is organized as follows.  Appendix \ref{sec:prely} introduces all  preliminaries we use in the full proofs of  our results. In Appendix \ref{sec:wup}, we consider the \SWG model when the percolation probability   $p$ is over the critical value and  give the full proofs of the first claims of Theorems \ref{thm:intro-erdos-2} and \ref{thm:gen_overthershold}. The case under the probability threshold for the \SWG model is analyzed in Appendix \ref{sec:belo}, where the second claims   of Theorems \ref{thm:intro-erdos-2} and \ref{thm:gen_overthershold} are proved.
The analysis proving  the first claims of Theorems \ref{thm:main_matching} and \ref{thm:main_matching_rf} for the \SWGreg model is provided in Appendix \ref{sec:regularcase}, while Appendix \ref{sec:reg-gen} is devoted to the proof of Theorem \ref{th:reg.upper} that easily implies the second claims of Theorems \ref{thm:main_matching} and \ref{thm:main_matching_rf}. Finally, Appendix  \ref{sec:further} describes the generalizations of our analysis to the setting where a different transmission probability  can be assigned to the two types of edges (i.e. ring edges and bridges) and the case of longer node activity and incubation.

\section{Overview of Our Analysis} \label{sec:Proof Ideas}

A standard technique in bond percolation, applied for example to percolation in
infinite trees and in random graphs, is to analyze the process of running a BFS
in the percolation graph, delaying decisions about the percolation of edges
from a node $w$ to unvisited vertices until the time $w$ is taken out of the
BFS queue. In random graphs and infinite trees, the distribution of unvisited
neighbors of $w$ in the percolation graph remains simple, even conditioned on
previous history, and one can model the size of the BFS queue as a
Galton-Watson process (see Definition~\ref{def:GW}), thus reducing the
percolation analysis to standard results about branching processes. Basically,
if the number of vertices that we add at each step to the queue is less than
one on average, the visit will reach on average a constant number of vertices
and if it is more than one and the graph is infinite the visit will reach on
average an infinite number of vertices. 

\subsection{Analysis of bond percolation in the \texorpdfstring{\SWG} \, model} 
\label{sec:overv-analysis}

In this section, we describe the key ingredients of our analysis of the \SWG
model proving Theorems~\ref{thm:intro-erdos-2} and~\ref{thm:gen_overthershold},
whose detailed and rigorous proofs can be found in Appendix~\ref{sec:wup}
(for the case in which $p$ is above the critical threshold) and in
Appendix~\ref{sec:belo} (for the case in which $p$ is below the critical
threshold).

It would be very difficult to analyze a BFS exploration of the percolation
graph to study percolation in the small-world model \SWG, since the
distribution of unvisited neighbors of a vertex $w$ in the percolation graph
is highly dependent on the previous history of the BFS (in particular, it
matters whether none, one, or both of the neighbors of $w$ along the cycle are
already visited).

Instead, and this is one of the technical innovations of our work, we define a
modified BFS visit whose process is more tractable to analyze.

The main idea of our modified BFS is that in one step we do the following:
after we pull a node $w$ from the queue, we first look at the neighbors $x$ of
$w$ that are reachable through bridge edges in the percolation graphs; then,
for each ``bridge neighbor'' $x$ of $w$, we visit the ``local cluster'' of $x$,
that is, we explore the vertices reachable from $x$ along paths that only
consist of edges of the cycle that are in the percolation graph (we indicate
the local cluster of $x$ with $\LC(x)$); finally, we add to the queue all
non-visited vertices in the local clusters of the bridge neighbors of $w$.
These steps are exemplified in Fig.~\ref{fig:local_cluster}.

The point of doing things this way is that if we delay decisions about the
random choice of the bridge edges and the random choices of the percolation,
then we have a good understanding of the following two key random variables:

\begin{figure}
\centering
\includegraphics[width=.9\textwidth]{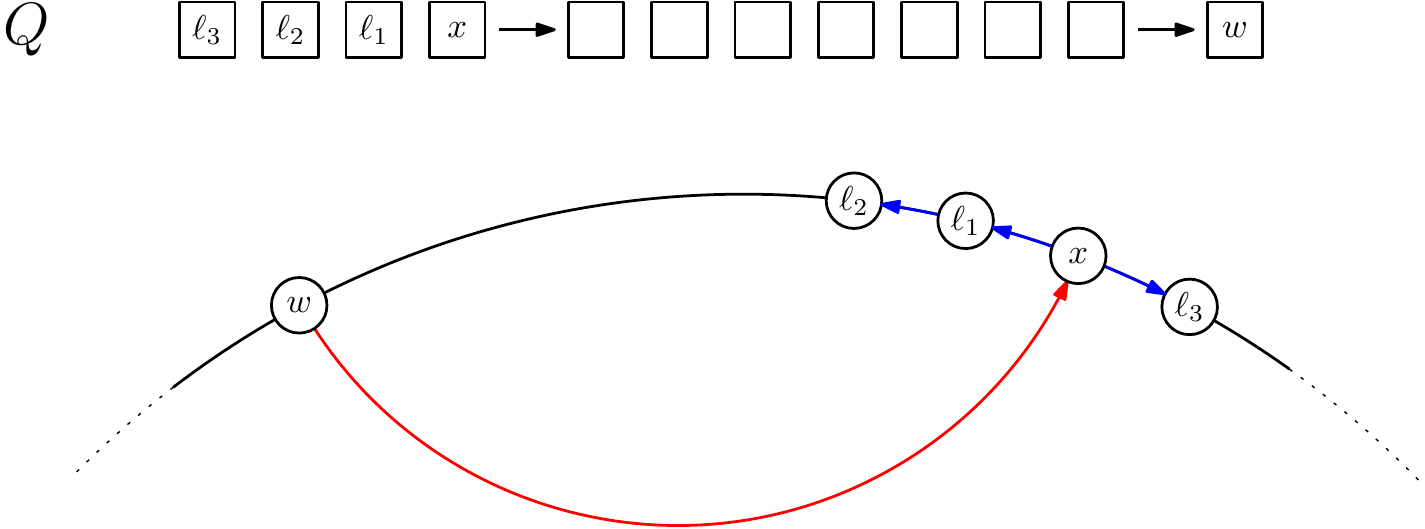}
\caption{The figure shows an example in which the visit first proceeds from a
node $w$ extracted from the queue to a new node $x$ over a bridge of the percolation graph and then reaches further nodes, starting from $x$ and proceeding along ring edges of the percolation graph. In terms of the RF protocol, this corresponds to the information (virus) first being transmitted from an infectious node $w$ to a susceptible one $x$ over a bridge edge and then propagating locally using ring edges. Note that in this case, i) we have a single bridge leaving $w$ (in general, there might be multiple ones) and ii) the information freely propagates locally from node $x$, informing susceptible nodes $\ell_1,\ell_2,\ell_3$. The edges are directed in the direction of information spread.}
\label{fig:local_cluster}
\end{figure}

\begin{itemize}

\item the number of bridge neighbors $x$ of $w$ along percolated bridge edges, which are, on average $pqn'$ if the graph comes from \SWG, $p$ is the percolation probability, and $n'$ is the number of unvisited vertices at that point in time;

\item the size of the ``local cluster'' of each such vertex $x$, that is of the vertices reachable from $x$ along percolated cycle edges, which has expectation 

\begin{equation}\label{eq:lc_intro}\Expcc{\LC(x)} = \frac{1+p}{1-p}\, .
\end{equation}

\end{itemize}

Intuitively, we would hope to argue that in our modified visit of a graph
sampled from \SWG to which we apply percolation with probability
$p$, the following happens in one step: we remove one node from the queue, and we add on average 
\begin{equation}\label{eq:nn_intro-new}
    N = pqn' \cdot \frac{1+p}{1-p}
\end{equation}
new nodes. As long as $n' = n-o(n)$ we can approximate $n'$ with $n$, and when the number $n-n'$ of visited vertices is $\Omega(n)$. This way, we would have modeled the size of the queue with a Galton-Watson
process and we would be done. The threshold behavior would occur at a $p$ such that $pqn \cdot (1+p)/(1-p) = 1$. A smaller value of $p$ would imply that we remove one node at every step and, on average, add less than one node to the queue, leading the process to die out quickly. A larger value of $p$ would imply that we remove one node at every step and, on average, add more than one node to the queue, leading the process to blow up until we reach $\Omega(n)$ vertices.

We are indeed able to prove this threshold behavior, at least for $q = c/n$ for constant $c$. However, we encounter significant difficulty in making this idea rigorous:  if we simply proceeded as described above, we would be double-counting vertices, because in general, the ``local cluster'' of a node added to the queue at a certain point may collide with the local cluster of another node added at a later point. This may be fine as long as we are trying to upper bound the number of reachable vertices, but it is definitely a problem if we are trying to establish a lower bound.

To remedy this difficulty, we truncate the exploration of each local cluster at a properly chosen constant size $L$ (we denote as $\LC ^L(x)$ the truncated local cluster of a node $x$). In our visit, we consider only unvisited neighbors $x$ of $w$ that are sufficiently far along the cycle from all previously visited vertices so that there is always ``enough space'' to grow a truncated local cluster around $x$ without hitting already visited vertices.
In more detail, we introduce the notion of ``\good node'' used in the algorithm and its analysis.
 
\begin{definition}[\good node]\label{def:free_node}
Let $G_{SW} = (V, E_1 \cup E_2)$ be a small-world graph and let $L \in
\mathbb{N}$. We say that a node $x \in V$ is  \good for a subset of nodes $X
\subseteq V$ if $x$ is at distance at least $L+1$ from any node in $X$ in the
subgraph $(V, E_1)$ induced by the edges of the ring.
\end{definition}
Thanks to the above definition, we can now formalize our modified BFS. 
%We remark that the value of $D_0$ can seems useless, but you can see its utility in the detailed proof in the appendix. TODO

\begin{algorithm}[H]
\caption{\textsc{Sequential} \SIRimmtrunc}
\small{
\textbf{Input}: A small-world graph $G_{\text{SW}} = (V, E_{\text{SW}})$; a subgraph $H$ of $G_{\text{SW}}$; a set of initiators $I_0 \subseteq V$; a set of deleted nodes $D_0 \subseteq V$.

    \begin{algorithmic}[1]
	\State $Q = I_0$
	\State $R = \emptyset$
	\State $D = D_0$
	\While{$Q \neq \emptyset $}\label{line:alg:L-visit:while}
	    \State $w = \dequeue(Q)$
	    \State $R = R \cup \{w\}$ 
 	            \For {each bridge neighbor $x$ of $w$ in $H$}\label{line:alg:L-visit:for}
 	              \If {$x$ is \good for $D \cup R \cup Q$ in $G_{SW}$}\label{line:alg:L-visit:if}
 	                  \For{each node $y$ in the $L$-truncated local cluster $\LC ^L (x)$} 
 	                       \State $\enqueue(y,Q)$
 	                  \EndFor
 	                
 	              \EndIf
 	            \EndFor
 	 \EndWhile\label{line:alg:L-visit:endwhile}
	\end{algorithmic}}
\label{alg:L-visit}
\end{algorithm}

To sum up, the $L$-truncation negligibly affects the average size of local clusters, the restriction to a subset of unvisited vertices negligibly affects the distribution of unvisited neighbors, and the analysis carries through with the same parameters and without the ``collision of local clusters'' problem.
% As an overview  of our general approach, we consider the percolation graph
% $G_p$ of a graph sampled from the $\SWGm(n,q)$ model with parameters $p$ and
% $q =
%c/n$, where $c$ is any positive constant such that $p \cdot q = pc/n <
%1$\footnote{Notice that if $p \cdot q \geq 1$, the analysis is trivial and all
%nodes will be infected, w.h.p.}.

In more detail, thanks to the arguments we described above, from~\eqref{eq:lc_intro} and~\eqref{eq:nn_intro-new}, we can prove that, 
%Claim 1 of Theorem \ref{thm:gen_overthershold}
if $p$ is above the critical threshold
\begin{equation*}
    \frac{\sqrt{c^{2} + 6c +1}-c-1}{2c} \, ,
\end{equation*}
then, with probability $\Omega(1)$, the connected components of $G_p$
containing the initiator subset have overall size $\Omega(n)$. In terms of our
BFS visit in Algorithm~\ref{alg:L-visit}, we in fact derive the following
result\footnote{We state the result for the case $|I_0|=1$.} (see
Subsection~\ref{ssec:bootstrap1} in the Appendix for its full proof).

\begin{lemma}\label{le:L-visit}
Let $V$ be a set of $n$ nodes, $s \in V$ an \textit{initiator} node and $D_0
\subseteq V \setminus \{s\}$ a set of deleted nodes such that $|D_0 | \leq
\log^4 n$. For every $\varepsilon >0$ and $c>0$, and for every probability $p$
such that
\[
\frac{\sqrt{c^{2} + 6c +1}-c-1}{2c} + \varepsilon \leq p \leq  1 \, , 
\]
there are positive parameters $L,k,t_0,\varepsilon'$, and $\gamma$, that depend
only on $c$ and $\varepsilon$, such that the following holds.  Sample a graph
$G = (V,E)$ according to the $\SWGm(n, c/n)$ distribution and let $G_p$ be the
percolation graph of $G$ with percolation probability $p$. Run the
\textsc{Sequential $L$-visit} procedure in Algorithm~\ref{alg:L-visit} on input
$(G,G_p,s,D_0)$: if $n$ is sufficiently large, for every $t$ larger than $t_0$,
at the end of the $t$-th iteration of the while loop we have 
\[
\Prc{|R \cup Q| \geq n/k \mbox{ \emph{OR} } |Q| \geq \varepsilon' t} \geq
\gamma \, , 
\]
where the probability is over both the randomness of the choice of $G$ from
$\SWGm(n, c/n)$ and over the choice of the percolation graph $G_p$.
\end{lemma}

The truncation is such that our modified BFS does not discover all vertices reachable from $I_0$ in the percolation graph, but only a subset. However, this is sufficient to prove lower bounds to the number of reachable vertices when $p$ is above the threshold. Proving upper bounds, when $p$ is under the threshold (i.e., the second claims of Theorems~\ref{thm:intro-erdos-2} and~\ref{thm:gen_overthershold}) is easier because, as mentioned, we can allow double-counting of reachable vertices.

The above line of reasoning is our key idea, when $p$ is above the threshold, to get $\Omega_\varepsilon(1)$ confidence probability for: (i) the existence of a linear-size, induced connected subgraph in $G_p$ (i.e., a ``weaker'' version of Claim~1 of Theorem~\ref{thm:intro-erdos-2}), and (ii) the existence of a large epidemic outbreak, starting from an arbitrary source subset $I_0$ (i.e., Claim~1 of Theorem~\ref{thm:gen_overthershold}). A full description of this analysis is provided in Appendix~\ref{sec:wup}, where we also describe the further technical steps to achieve \emph{high-probability} for event (i), and  also event (ii) when the size of the source subset is $|I_0| = \Omega(\log n)$ (see Subsection~\ref{ssec:proof:thm:gen_overthershold}).

\medskip
\noindent
\textbf{Bounding the number of hops: parallelization of the BFS visit.}
To get bounds on the number of the BFS levels, we study the BFS-visit in Algorithm~\ref{alg:L-visit}  only up to the point where there are $\Omega(\log n)$ nodes in the queue (this first phase is not needed if $I_0$ already has size $\Omega(\log n)$), and then we study a ``\emph{parallel}'' visit in which we add at once all nodes reachable through an $L$-truncated local cluster and through the bridges from the nodes currently in the queue, skipping those that would create problems with our invariants: to this aim, we need a stronger version of the notion of \good node (see Definition~\ref{def:good_nodes_par} in Subsection~\ref{ssec:parallelizatiobfs}).

Here we can argue that, as long as the number of visited vertices is $o(n)$, the number of nodes in the queue grows by a constant factor in each iteration, and so we reach $\Omega(n)$ nodes in $\bigO(\log n)$ number of iterations that corresponds to $\bigO(\log n)$ distance from the source subset in the percolation graph $G_p$.

A technical issue that we need to address in the analysis of our parallel visit is that the random variables that count the contribution of each $L$-truncated local cluster,  added during one iteration of the visit, are not mutually independent. To prove concentration results for this exponential growth, we thus need to show that such a mutual correlation satisfies a certain \emph{local} property and then apply suitable bounds for partly-dependent random variables~\cite{Svante04} (see Theorem~\ref{thm:svante} in  Subsection \ref{app:maths} in the Appendix). All details of this part can be found in Subsection~\ref{ssec:parallelizatiobfs} in the Appendix.

\subsection{Further challenges in regular small-world graphs} 
In this section, we describe the main differences of the analysis of the \SWGreg model with respect to the analysis of the \SWG one. The following arguments give a high-level overview of the proofs of Theorems~\ref{thm:main_matching} and~\ref{thm:main_matching_rf}. The detailed proofs can be found in Appendix~\ref{sec:regularcase} (for the case $p$ above the threshold) and in Appendix~\ref{sec:reg-gen} (for the case $p$ below the threshold). 

The \SWGreg\ model, in which bridge edges form a random matching, introduces additional dependencies on the past history, compared to the analysis of the \SWG model. We deal with this difficulty by disallowing additional unvisited vertices to be reached in the visit. When we take a node $w$ out of the queue in the \SWGreg model there can be at most one unvisited neighbor $x$ of $w$ reachable through a bridge edge in the percolation graph. If such a neighbor $x$ exists, and it is not one of the disallowed unvisited vertices, we find, as before, the truncated local cluster of $x$ and add the nodes of the local cluster of $x$ to the queue, {\em except for $x$ itself}. The reason for discarding $x$ is that we have already observed the unique bridge neighbor of $x$ (namely, $w$) so, if we added $x$ to the queue, there would be no randomness left to apply the deferred decision principle when we later remove $x$ from the queue.

This means that, while in one step of our visit on \SWG\ with activation probability $p$ we take out one node from the queue and add in expectation a number of nodes described by~\eqref{eq:nn_intro-new}, in \SWGreg\ we take out one node and add in expectation
\begin{equation*}%\label{eq:3reg_intro}
    N' = p \cdot \left(\frac{1+p}{1-p} - 1\right)
\end{equation*}
nodes. This is the reason why the \SWG\ model with $q = 1/n$ and the \SWGreg\ model have notably different thresholds, even though they are superficially very similar.

The above argument allows us to prove Claims~1 and~2 of Theorem~\ref{thm:main_matching}. Currently, we are not able to analyze the parallel visit in the \SWGreg\ model, because of the correlations between the edges, although we are able to analyse the sequential visit up to $\Omega(n)$ nodes. This is why in our theorems we do not have an exponential growth of the BFS levels for the \SWGreg model.

As for Claim~3 of Theorem~\ref{thm:main_matching}, as remarked in
Section~\ref{sec:ourresults}, it is a direct consequence of the more general
bounds given by Theorem~\ref{th:reg.upper}: Given any graph $G = (V,E)$ of
maximum degree $d$ and a percolation probability $p < (1-\varepsilon)/(d-1)$,
the number of nodes connected to any given source subset $I_0$ in $G_p$ is
$\bigO_\varepsilon (|I_0| \log n)$, w.h.p. 

The proof of the above result (see Appendix~\ref{sec:reg-gen} for the details)
again relies on a suitable BFS visit of the percolation graph $G_p$ which is
similar to that in Figure~\ref{fig:local_cluster} in the previous subsection:
We start with a queue $Q$ containing only the source node and at each iteration
of a while loop (that terminates when the queue is empty) we extract a node
from the queue and we add to the queue all its neighbors in the percolation
graph. Informally speaking, every time the BFS adds a node to the queue, it is
observing a Bernoulli random variable with parameter $p <
(1-\varepsilon)/(d-1)$ (the percolation probability of each visited edge).
Since the input graph has maximum degree $d$, if the procedure runs for $t$
iterations of the while loop then $t$ nodes are extracted from the queue and in
expectation $p \left(t \cdot (d-1)+1\right)$ are added to the queue.  Chernoff's bound then implies that the probability that a queue starting from
a single source node is not yet empty after $t$ iterations of the while loop is
$\exp(-\Theta(\varepsilon^2 t))$.  Hence, the size of the connected component
containing the source node is $\bigO(\log n) $, w.h.p.  Finally, the fact that
all components are of size $\bigO(\log n)$ follows from a union bound.

\section{Related Work} \label{sec:related_new}

The \emph{fully-mixed}  SIR model \cite{VespietAl15} is the simplest SIR epidemiological model, and it treats the number of people in each of the three possible states as continuous quantities that evolve in time in accordance with certain differential equations. In this setup, the evolution of the process is governed by the expected number $R_0$ of people that each infectious person would infect, if all the contacts of that person were susceptible. If $R_0 <1$, the process quickly ends, reaching a state with zero infectious people and a small number of recovered ones. If $R_0 > 1$, the process goes through an initial phase in which the number of infectious people grows exponentially with time, until the number of recovered people becomes a $1- 1/R_0$ fraction of the population (the {\em herd immunity threshold}); the number of infectious people decreases after that, and eventually the process ends with a constant fraction of the population in the recovered state.

If we consider the Reed-Frost process on a graph $G$ that is a clique on $n$ vertices, then the percolation graph $G_p$ is an \Erdos\ random graph with edge probability sampled from $\mathcal{G}_{n,p}$. Classical results from the analysis of random graphs give us that  
if $pn < 1 - \epsilon$ then, with high probability, all the connected components of the graph have size $O_\epsilon (\log n)$, and so the set of vertices that is reachable from $I_0$ has cardinality at most $O_\epsilon (|I_0| \cdot \log n)$ and if $pn > 1 + \epsilon$ then there is a connected component of cardinality $\Omega_\epsilon (n)$, and, except with probability exponentially small in $I_0$, at least one vertex of $I_0$ belongs to the giant component and is able to reach $\Omega_\epsilon (n)$ vertices. The parameter $R_0$  of the fully mixed continuous model corresponds to the average degree of $G_p$, which is $pn$ if $G_p$ is distributed as $\mathcal{G}_{n,p}$, so we see that the fully mixed continuous model agrees with the Reed-Frost process on a clique.

A number of techniques have been developed to study percolation in graphs other than the clique, and there is a vast body of work devoted to the study of models of bond percolation and epidemic spreading, as surveyed in~\cite{VespietAl15, Wang_2017}.
Below, we review analytical studies of such processes on finite graphs.
As far as we know, our results are the first rigorous ones to establish threshold phenomena in small-world graphs for the bond-percolation process (and, thus, for the Reed-Frost   process). 

There has been some previous work on studying sufficient conditions for the RF  process to reach a sublinear number of vertices.

In~\cite{DGM06}, for a symmetric, connected graph $G = (V, E)$, Draief et al. prove a general lower bound on the critical point for the IC process
in terms of spectral properties. Further versions of such bounds for
special cases have been subsequently derived in~\cite{LKAK16,LSV14}.
Specifically, if one lets $P$ be the matrix such that $P(u,v) = p(u,v)$ is the percolation probability of the edge $\{u,v\}$, and $P(u,v) = 0$ if $\{u,v\}\not\in E$, and if one call $\lambda$ the largest eigenvalue of $P$, then $\lambda < 1-\epsilon$ implies that for a random start vertex $s$ we have that the expected number of vertices to which $s$ spreads the infection is $o_\epsilon(n)$.

In the RF process,  in which all probabilities are the same, $P = p \cdot A$, where $A$ is the adjacency matrix of $G$, and so the condition is asking for $p < (1-\epsilon)/ \lambda_{\max{}} (A)$.

This condition is typically not tight, and it is never tight in the
``small-worlds'' graphs   we consider:

\begin{itemize}

\item In the \SWGreg model, the largest eigenvalue of the adjacency matrix is $3-o(1)$, but the critical probability is $1/2$ and not $1/3$;

\item In the $\SWGm(n,1/n)$ model of a cycle plus \Erdos\ edges, the largest
eigenvalue of the adjacency matrix is typically $\Omega(\sqrt{ \log n /
\log\log n})$ because we expect to see vertices of degree $\Omega(\log n /\log\log n)$ and the largest eigenvalue of the adjacency matrix of a graph is at least the square root of its maximal degree. The spectral bound would only tell us that the infection dies out if $p = \bigO(\sqrt{\log\log n / \log n})$, which goes to zero with $n$.  A better way to use the spectral approach is to model the randomness of the small-world graph and the randomness of the percolation together; in this case, we have matrix $P(u,v)$ such that $P(u,v)=p$ for edges of the cycle and $P(u,v) = p/n$ for the other edges. This matrix has the largest eigenvalue $3p-o(1)$, so the spectral method would give a probability of $1/3$, while we can locate the threshold at $\sqrt 2-1\approx .41$.
\end{itemize}

In any family of $d$-regular graphs, the largest eigenvalue of the adjacency matrix is $d$, and so the spectral bound gives that the critical threshold is at least $1/d$; our Theorem~\ref{th:reg.upper} shows the stronger bound that the critical threshold is at least $1/(d-1)$.

We are not aware of previous rigorous results that provide sufficient
conditions for the IC process to reach $\Omega(n)$ nodes
(either on average or with high probability) in general graphs, or for the
equivalent question of proving that the percolation graph of a given graph has a connected component with $\Omega(n)$ vertices.

As discussed in the previous section, our analysis proceeds by analyzing a BFS-like visit of the percolation graph. This is also how large components in the percolation of infinite trees and random graphs have been studied before. However, this idea requires considerable elaboration to work in our setting, given the mix of fixed edges and random edges in the small-world model and the complicated dependencies on the past history that one has to control in the analysis of the visit.

A fundamental and rigorous study of bond percolation in random graphs has been proposed by Bollob\'as et al. in~\cite{Bollo07}.  They establish a coupling between the bond percolation process and a suitably defined branching process. In the   general class of inhomogenous \Erdos\ random graphs,
%, calling $Q = \{q_{u,v}\}$   the edge-probability matrix,
they derived the critical point (threshold) of the phase transition and  the size of the giant component above the transition. The  class  of inhomogeneous random graphs to which their analysis applies includes  generative models that have been studied in the complex network literature. For instance, a version of the Dubin's model~\cite{Durrett1990ATT} can be expressed in this way, and so can the \emph{mean-field scale-free model}~\cite{Bollobs2004TheDO}, which is, in turn, related to the   Barabási–Albert model~\cite{Barab509}, having the same individual edge probabilities, but with edges present independently. Finally, we observe that the popular CHKNS model introduced by  Callaway et al.~\cite{CDHKNS01} can be analyzed using an edge-independent version of this model. Indeed, they consider a random graph-formation process where, after adding each node,   a Poisson  number of edges is added to the graph, again choosing the endpoints of these edges uniformly at random.  For all such important classes of random graph models, they show tight bounds for the critical points and the relative size of the giant component beyond the phase transition.

In our setting, if we sample a graph from $\SWGm(n,q)$ and then consider the percolation graph $G_p$, the distribution of $G_p$ is that of an inhomogenous \Erdos\ graph in which the cycle edges have probability $p$ and the remaining edges have probability $pq$ (the \SWGreg model, however, cannot be expressed as an inhomogenous \Erdos\ graph).

Unfortunately, if we try to apply the results of~\cite{Bollo07} to the
inhomogeneous random graph equivalent to percolation with parameter $p$ in the $\SWGm(n,q)$ model, we do not obtain tractable conditions on the critical value $p$ for which the corresponding graph has a large connected component of small diameter, which is the kind of result that we are interested in proving.

Bond percolation and the  IC process on the class of $1$-dimensional small-world networks (that is, graphs obtained as the union of a cycle and of randomly chosen edges) have been studied in~\cite{MCN00}: using numerical approximations on the moment generating function, non-rigorous bounds on the critical threshold have been derived while analytical results are given neither for the expected size of the number of informed nodes above the transition phase of the process nor for its completion time.  Further non-rigorous results on the critical points of several classes of complex networks have been derived in~\cite{LRSV14,LKAK16} (for good surveys see~\cite{VespietAl15, Wang_2017}).

In \cite{BB01,Bis04,Bis11}, different versions of the bond percolation process has been studied in small-world structures formed by a $d$-dimensional grid augmented by random edges that follow a power-law distribution: a bridge between points $x$ and $y$ is selected with probability $\sim 1/\dist(x,y)^{\alpha}$, where $\dist(x,y)$ is the grid distance between $x$ and $y$ and $\alpha$ is a fixed  power-law parameter. Besides other aspects, each version is characterized by: (1)  whether the grid is infinite or finite, and  (2) whether the grid edges (local ties) do percolate with probability $p$  or not. 
Research in this setting has focused on the emergence of a large connected component and on its diameter as functions of the parameters $d$ and $\alpha$, while, to the best of our knowledge,  no rigorous threshold bounds are known for the bond percolation probability $p$. 

\iffalse 
Interestingly enough, in both power-law  infinite and finite structures (i.e. graphs), such studies globally show three different regimes of behaviour depending on the range the parameter $\alpha$ belongs to w.r.t. the dimension $d$. 
\fi 

In the computer science community, to the best of our knowledge, Kempe et al. ~\cite{KKT15} were the first to investigate the IC process from an optimization perspective, in the context of viral marketing and opinion diffusion. In particular, they introduced the \emph{Influence Maximization} problem, where the goal is to find a source subset of $k$ nodes of an underlying graph to inform at time $t = 0$, so as to maximize the expected number of informed nodes at the end of the IC process. They prove this is an $NP$-hard problem and show a polynomial time algorithm achieving constant approximation. Further approximation results on a version of \emph{Influence Maximization} in which the completion time of the process is considered can be found in~\cite{CWZ12,LGD12}.

\bibliography{njl}

\begin{thebibliography}{10}

\bibitem{ABS04}
Noga Alon, Benjamini Itai, and Stacey Alan.
\newblock Percolation on finite graphs and isoperimetric inequalities.
\newblock {\em Annals of Probability}, 32:1727--1745, 2004.

\bibitem{AS16}
Noga Alon and Joel~H. Spencer.
\newblock {\em The Probabilistic Method}.
\newblock Wiley Publishing, 2nd edition, 2000.

\bibitem{Barab509}
Albert-L{\'a}szl{\'o} Barab{\'a}si and R{\'e}ka Albert.
\newblock Emergence of scaling in random networks.
\newblock {\em Science}, 286(5439):509--512, 1999.
\newblock \href {https://doi.org/10.1126/science.286.5439.509}
  {\path{doi:10.1126/science.286.5439.509}}.

\bibitem{Bollobs2004TheDO}
Bollob{\'a}s B\'ela and Riordan Oliver.
\newblock The diameter of a scale-free random graph.
\newblock {\em Combinatorica}, 24:5--34, 2004.

\bibitem{Bis04}
Marek Biskup.
\newblock {On the scaling of the chemical distance in long-range percolation
  models}.
\newblock {\em The Annals of Probability}, 32(4):2938 -- 2977, 2004.
\newblock \href {https://doi.org/10.1214/009117904000000577}
  {\path{doi:10.1214/009117904000000577}}.

\bibitem{Bollo07}
B\'ela Bollob\'as, Svante Janson, and Oliver Riordan.
\newblock The phase transition in inhomogeneous random graphs.
\newblock {\em Random Structures \& Algorithms}, 31(1):3--122, 2007.
\newblock URL: \url{https://onlinelibrary.wiley.com/doi/abs/10.1002/rsa.20168},
  \href
  {http://arxiv.org/abs/https://onlinelibrary.wiley.com/doi/pdf/10.1002/rsa.20168}
  {\path{arXiv:https://onlinelibrary.wiley.com/doi/pdf/10.1002/rsa.20168}},
  \href {https://doi.org/https://doi.org/10.1002/rsa.20168}
  {\path{doi:https://doi.org/10.1002/rsa.20168}}.

\bibitem{CDHKNS01}
Duncan~S. Callaway, John~E. Hopcroft, Jon~M. Kleinberg, M.~E.~J. Newman, and
  Steven~H. Strogatz.
\newblock Are randomly grown graphs really random?
\newblock {\em Phys. Rev. E}, 64:041902, Sep 2001.
\newblock \href {https://doi.org/10.1103/PhysRevE.64.041902}
  {\path{doi:10.1103/PhysRevE.64.041902}}.

\bibitem{CWLC13}
Wei Chen, Laks Lakshmanan, and Carlos Castillo.
\newblock Information and influence propagation in social networks.
\newblock {\em Synthesis Lectures on Data Management}, 5:1--177, 10 2013.
\newblock \href {https://doi.org/10.2200/S00527ED1V01Y201308DTM037}
  {\path{doi:10.2200/S00527ED1V01Y201308DTM037}}.

\bibitem{CWZ12}
Wei Chen, Wei Lu, and Ning Zhang.
\newblock Time-critical influence maximization in social networks with
  time-delayed diffusion process.
\newblock In {\em Proceedings of the Twenty-Sixth AAAI Conference on Artificial
  Intelligence}, AAAI'12, page 592–598. AAAI Press, 2012.

\bibitem{CPGE19}
Hyeongrak Choi, Mihir Pant, Saikat Guha, and Dirk Englund.
\newblock Percolation-based architecture for cluster state creation using
  photon-mediated entanglement between atomic memories.
\newblock {\em npj Quantum Information}, 5(1):104, 2019.

\bibitem{DGM06}
Moez Draief, Ayalvadi Ganesh, and Laurent Massouli\'{e}.
\newblock Thresholds for virus spread on networks.
\newblock In {\em Proceedings of the 1st International Conference on
  Performance Evaluation Methodolgies and Tools}, valuetools '06, page 51–es,
  New York, NY, USA, 2006. Association for Computing Machinery.
\newblock \href {https://doi.org/10.1145/1190095.1190160}
  {\path{doi:10.1145/1190095.1190160}}.

\bibitem{Durrett1990ATT}
Rick Durrett and Harry Kesten.
\newblock The critical parameter for connectedness of some random graphs.
\newblock {\em A Tribute to P. Erdos}, pages 161--176, 1990.

\bibitem{EK10}
David Easley and Jon Kleinberg.
\newblock {\em Networks, Crowds, and Markets: Reasoning About a Highly
  Connected World}.
\newblock Cambridge University Press, USA, 2010.

\bibitem{LKAK16}
{Eun Jee Lee}, Sudeep {Kamath}, Emmanuel {Abbe}, and Sanjeev~R. {Kulkarni}.
\newblock Spectral bounds for independent cascade model with sensitive edges.
\newblock In {\em 2016 Annual Conference on Information Science and Systems
  (CISS)}, pages 649--653, 2016.
\newblock \href {https://doi.org/10.1109/CISS.2016.7460579}
  {\path{doi:10.1109/CISS.2016.7460579}}.

\bibitem{GGT03}
Michele Garetto, Weibo Gong, and Don Towsley.
\newblock Modeling malware spreading dynamics.
\newblock In {\em IEEE INFOCOM 2003. Twenty-second Annual Joint Conference of
  the IEEE Computer and Communications Societies (IEEE Cat. No.03CH37428)},
  volume~3, pages 1869--1879 vol.3, 2003.
\newblock \href {https://doi.org/10.1109/INFCOM.2003.1209209}
  {\path{doi:10.1109/INFCOM.2003.1209209}}.

\bibitem{BB01}
Benjamini Itai and Berger Noam.
\newblock The diameter of long-range percolation clusters on finite cycles.
\newblock {\em Random Struct. Algorithms}, 19:102--111, 2001.

\bibitem{Svante04}
Svante Janson.
\newblock Large deviations for sums of partly dependent random variables.
\newblock {\em Random Structures and Algorithms}, 24, 05 2004.
\newblock \href {https://doi.org/10.1002/rsa.20008}
  {\path{doi:10.1002/rsa.20008}}.

\bibitem{KNT94}
Anna~R. Karlin, Greg Nelson, and Hisao Tamaki.
\newblock On the fault tolerance of the butterfly.
\newblock In {\em Proceedings of the twenty-sixth annual ACM symposium on
  Theory of Computing}, pages 125--133, 1994.

\bibitem{KKT15}
David Kempe, Jon Kleinberg, and \'{E}va Tardos.
\newblock Maximizing the spread of influence through a social network.
\newblock {\em Theory of Computing}, 11(4):105--147, 2015.
\newblock (An extended abstract appeared in Proc. of 9th ACM KDD '03).
\newblock \href {https://doi.org/10.4086/toc.2015.v011a004}
  {\path{doi:10.4086/toc.2015.v011a004}}.

\bibitem{KetAl80}
Harry Kesten.
\newblock The critical probability of bond percolation on the square lattice
  equals 1/2.
\newblock {\em Communications in mathematical physics}, 74(1):41--59, 1980.

\bibitem{KL19}
Alexander Kott and Igor Linkov.
\newblock {\em Cyber resilience of systems and networks}.
\newblock Springer, 2019.

\bibitem{LSV14}
Remi Lemonnier, Kevin Scaman, and Nicolas Vayatis.
\newblock Tight bounds for influence in diffusion networks and application to
  bond percolation and epidemiology.
\newblock In Z.~Ghahramani, M.~Welling, C.~Cortes, N.~Lawrence, and K.~Q.
  Weinberger, editors, {\em Advances in Neural Information Processing Systems},
  volume~27, pages 846--854. Curran Associates, Inc., 2014.
\newblock URL:
  \url{https://proceedings.neurips.cc/paper/2014/file/1bb91f73e9d31ea2830a5e73ce3ed328-Paper.pdf}.

\bibitem{LRSV14}
R\'{e}mi Lemonnier, Kevin Seaman, and Nicolas Vayatis.
\newblock Tight bounds for influence in diffusion networks and application to
  bond percolation and epidemiology.
\newblock In {\em Proceedings of the 27th International Conference on Neural
  Information Processing Systems - Volume 1}, NIPS'14, page 846–854,
  Cambridge, MA, USA, 2014. MIT Press.

\bibitem{lin2020conceptual}
Qianying Lin, Shi Zhao, Daozhou Gao, Yijun Lou, Shu Yang, Salihu~S Musa,
  Maggie~H Wang, Yongli Cai, Weiming Wang, Lin Yang, et~al.
\newblock {A conceptual model for the coronavirus disease 2019 (COVID-19)
  outbreak in Wuhan, China with individual reaction and governmental action}.
\newblock {\em International journal of infectious diseases}, 93:211--216,
  2020.

\bibitem{LGD12}
Bo~Liu, Gao Cong, Dong Xu, and Yifeng Zeng.
\newblock Time constrained influence maximization in social networks.
\newblock In {\em Proceedings of the IEEE International Conference on Data
  Mining, ICDM}, pages 439--448, 12 2012.
\newblock \href {https://doi.org/10.1109/ICDM.2012.158}
  {\path{doi:10.1109/ICDM.2012.158}}.

\bibitem{Bis11}
Biskup Marek.
\newblock Graph diameter in long-range percolation.
\newblock {\em Random Struct. Algorithms}, 39(2):210–227, sep 2011.
\newblock \href {https://doi.org/10.1002/rsa.20349}
  {\path{doi:10.1002/rsa.20349}}.

\bibitem{MCN00}
Christopher Moore and Mark~E.J. Newman.
\newblock Exact solution of site and bond percolation on small-world networks.
\newblock {\em Phys. Rev. E}, 62:7059--7064, Nov 2000.
\newblock URL: \url{https://link.aps.org/doi/10.1103/PhysRevE.62.7059}, \href
  {https://doi.org/10.1103/PhysRevE.62.7059}
  {\path{doi:10.1103/PhysRevE.62.7059}}.

\bibitem{MN00b}
Cristopher Moore and Mark~E.J. Newman.
\newblock Epidemics and percolation in small-world networks.
\newblock {\em Phys. Rev. E}, 61:5678--5682, May 2000.
\newblock URL: \url{https://link.aps.org/doi/10.1103/PhysRevE.61.5678}, \href
  {https://doi.org/10.1103/PhysRevE.61.5678}
  {\path{doi:10.1103/PhysRevE.61.5678}}.

\bibitem{morris1963note}
K.W. Morris.
\newblock A note on direct and inverse binomial sampling.
\newblock {\em Biometrika}, 50(3-4):544--545, 1963.

\bibitem{NewmanS86}
Charles~M. Newman and Lawrence~S. Schulman.
\newblock One dimensional $1/| j- i| s$ percolation models: The existence of a
  transition for $s \leq 2$.
\newblock {\em Communications in Mathematical Physics}, 104(4):547--571, 1986.

\bibitem{newman1999scaling}
Mark~E.J. Newman and Duncan~J. Watts.
\newblock Scaling and percolation in the small-world network model.
\newblock {\em Physical review E}, 60(6):7332, 1999.

\bibitem{VespietAl15}
Romualdo Pastor-Satorras, Claudio Castellano, Piet Van~Mieghem, and Alessandro
  Vespignani.
\newblock Epidemic processes in complex networks.
\newblock {\em Rev. Mod. Phys.}, 87:925--979, Aug 2015.
\newblock \href {https://doi.org/10.1103/RevModPhys.87.925}
  {\path{doi:10.1103/RevModPhys.87.925}}.

\bibitem{VK71}
Vinod~K.S. Shante and Scott Kirkpatrick.
\newblock An introduction to percolation theory.
\newblock {\em Advances in Physics}, 20(85):325--357, 1971.
\newblock \href
  {http://arxiv.org/abs/https://doi.org/10.1080/00018737100101261}
  {\path{arXiv:https://doi.org/10.1080/00018737100101261}}, \href
  {https://doi.org/10.1080/00018737100101261}
  {\path{doi:10.1080/00018737100101261}}.

\bibitem{wald1944}
Abraham Wald.
\newblock On cumulative sums of random variables.
\newblock {\em Ann. Math. Statist.}, 15(3):283--296, 09 1944.
\newblock \href {https://doi.org/10.1214/aoms/1177731235}
  {\path{doi:10.1214/aoms/1177731235}}.

\bibitem{watts1998collective}
Duncan~J. Watts and Steven~H. Strogatz.
\newblock Collective dynamics of ‘small-world’ networks.
\newblock {\em nature}, 393(6684):440--442, 1998.

\bibitem{Wang_2017}
Wang Wei, Tang Ming, Stanley Eugene, and Lidia~A. Braunstein.
\newblock Unification of theoretical approaches for epidemic spreading on
  complex networks.
\newblock {\em Reports on Progress in Physics}, 80(3):036603, feb 2017.
\newblock \href {https://doi.org/10.1088/1361-6633/aa5398}
  {\path{doi:10.1088/1361-6633/aa5398}}.

\end{thebibliography}

\appendix

\section{Preliminaries} \label{sec:prely}

\subsection {Formal definitions}\label{ssec:formal_definitions}

In this subsection of the Appendix, we give the rigorous definition of the bond percolation process.

Given a symmetric graph $G = (V,E)$, we define $|V| = n$ and, for
any node $v \in V$, we denote $N(v)$ as its neighborhood in $G$ and $d(v) =
|N(v)| $ as  its degree.  The distance $d_{G}(S,v)$ from the set $S$ to the
node $v$ is the length of the shortest path among all paths from any node in
$S$ to $v$ in $G$, if no such path exists $d_{G}(S,v) = + \infty$, if $v \in
S$, $d_{G}(S,v) = 0$. Since we will only consider symmetric graph, the term
``symmetric'' will be omitted.  
Given a  graph $G = (V, E)$, for any subset $S \subseteq V$ and node $v \in V$,
and for any integer $i \leq n-1$, we  let $N^{(i)}_{G}(S)$ be the subset of
nodes that are at distance $i$ from $S$, i.e. $N^{(i)}_{G}(S) = \{v \in V \mid
d_{G}(S,v) = i\}$. Moreover, the set of nodes that are  within finite distance
from $S$, i.e. that are reachable from $S$, will be denoted as $N^*_{G}(S)$.

%Introduci qui il Bond-Percolation process

We consider the following bond percolation process on any fixed graph $G$ (this process is also known in network theory as \emph{Live-arc
graph model with Independent arc selection} - see
also~\cite{CWLC13}).

\begin{definition}[The Bond Percolation process]
\label{def:live-arc-model}
Given a graph $G = (V,E)$, and given, for every edge $e \in E$, a
\emph{percolation probability} $p(e) \in [0,1]$, the \emph{bond percolation}  process consists to
\emph{remove} each edge $e \in E$, independently, with probability $1-p(e)$.
The random subgraph, called the \emph{percolation graph} $G_p=(V,E_p)$, is
defined by the edges that are not removed (i.e. they are activated), i.e., $E_p = \{ e \in E \, : \,
\mbox{edge $e$ is not removed} \}$. Given an initial subset $A_0 \subseteq V$ of
\emph{active} nodes, for every integer $t \leq n-1$, we define the random
subset $A_{t}$ of  $t$-active nodes as the subset of nodes that are at distance
$t$ from $A_{0}$ in the percolation graph $G_p=(V,E_p)$, i.e., $A_{t} =
N_{G_{p}}^{t}(A_{0})$. Finally, the subset of all active nodes from $A_{0}$ is
the subset $N^*_{G_p}(A_0)$.
\end{definition}

\subsection{The equivalence between the \texorpdfstring{\IC} \, model and the Live-Arc model}
\label{ssec:equi}

% Introduci qui l'IC ed il RF protocols

In this paper, we consider the 
following synchronous, discrete-time epidemic protocol working over any   graph $G$
(see~\cite{EK10,VespietAl15}).

\begin{definition}[\IC and RF protocols] \label{def:independent_cascade}
Given a graph $G = (V,E)$, an assignment of \emph{transmission probabilities} $\{
p(e) \}_{e\in E}$ to the edges of $G$, and a non-empty set $I_0 \subseteq V$ of
initially infectious\footnote{we use here the term \emph{infectious} for two reasons: to emphasize that the node is both informed and active and, moreover, to  be consistent with the literature in mathematical epidemiology.} nodes (that will also be called {\em initiators} or
{\em sources} since they have the information/virus since the very beginning), the {\em Independent Cascade} (for short, \IC) protocol defines the 
stochastic process $\{S_t,I_t,R_t\}_{t\geq 0}$ on $G$, where $S_t,I_t,R_t$ are three
sets of vertices, respectively called \emph{susceptible}, \emph{infectious}, and \emph{recovered},
which form a partition of $V$ and that are defined as follows. 

\begin{itemize}
     \item At time $t=0$ we have $R_0 = \emptyset$ and $S_0 = V-I_0$.
     \item At time $t\geq 1$:
     \begin{itemize}
         \item $R_t = R_{t-1} \cup I_{t-1}$, that is, the nodes that were
         infected at the previous step become recovered.
         
         \item Independently for each edge $e = \{ u,v\}$ such that $u \in
         I_{t-1}$ and $v\in S_{t-1}$, with probability $p(e)$ the event that
         ``$u$ transmits the infection (i.e. a copy of the source message) to $v$ at time $t$'' takes place. The
         set $I_t$ is the set of all vertices $v\in S_{t-1}$ such that for at
         least one neighbor $u \in I_{t-1}$ the event that $u$ transmits the
         infection to $v$ takes place as described above.
         
         \item $S_t = S_{t-1} - I_t$
     \end{itemize}
 \end{itemize}
The process stabilizes when $I_t = \emptyset$.\\ The Reed-Frost \SIR protocol (for
short, \emph{RF Protocol}) is the special case of the \IC protocol in which all transmission probabilities
are the same. 
\end{definition}

By the above definition, since each node can be in the infective state only for
one step, we observe that the stopping time $\tau = \min \{t>0:I_t=\emptyset
\}$ is upper bounded with probability $1$ by the diameter of $G$.

In~\cite{KKT15}, given any fixed graph $G = (V,E)$, the \IC protocol is shown to
be \emph{equivalent} to the bond percolation   process.

If we consider the set $I_t$ of nodes that are infectious at time $t$ in a graph
$G=(V,E)$ according to the IC protocol with transmission
probabilities $\{ p(e) \}_{e\in E} $ and with initiator set $I_0$, we see that
such a set has precisely the same distribution as the set of nodes at distance
$t$ from $I_0$ in the percolation graph $G_p$ generated by the bond percolation process with  probabilities $\{ p(e)
\}_{e\in E}$ (see Definition \ref{def:live-arc-model}). Furthermore, the set of recovered nodes $R_t$ is distributed
precisely like the set of nodes at distance $< t$ from $I_0$ in $G_p$.

We formalize this equivalence by quoting a theorem from~\cite{CWLC13}.

\begin{theorem}[Bond percolation and IC processes are equivalent, \cite{KKT15}]
\label{thm:equivalence}
Consider the bond percolation process and the   \IC protocol   on the same graph $G=(V,E)$ and
let $I_0 = A_0 = V_0$, where $V_0$ is any fixed subset of $V$, and with
transmission probabilities and percolation probabilities equal to $\{ p(e) \}_{e\in
E}$. 
 
Then,  for every integer $t \geq 1$ and for every subsets $V_1, \dots, V_{t-1}
\subseteq V$, the events $\{I_0 = V_0, \dots, I_{t-1} = V_{t-1}\}$ and $\{A_0 =
V_0, \dots, A_{t-1} = V_{t-1}\}$ have either both zero probability   or
non-zero probability, and, in the latter case, the  distribution of the
infectious set $I_{t}$,  conditional to the event $\{I_0 = V_0, \dots, I_{t-1}
= V_{t-1}\}$, is   the same  to that of  the $t$-active set $A_t$, conditional
to the event  $\{A_0 = V_0, \dots, A_{t-1} = V_{t-1}\}$. 
\end{theorem}

The strong equivalence shown in the previous theorem is obtained by applying
the principle of deferred decision on the percolation/infection events that take
place on every edge since they are mutually independent. This result can be
exploited to analyze different aspects and issues of the  
IC (and, thus, the RF) protocol. We here summarize such aspects in an informal way, while, in the
next sections, we show  rigorous  claims along our analysis.

As a first immediate consequence of Theorem~\ref{thm:equivalence}, we have
that, starting from any source subset $I_0$, to bound the size of the final set
$R_{\tau}$ of the nodes informed by $I_0$, we can look at the size of the union
of the connected components in $G_p$ that include all nodes of $I_0$, i.e., we
can bound the size of $N^*_{G_p}(I_0)$.

A further remark is that in the bond percolation process there is no {\em time}, and we can
analyze the connected component of the percolation graph in any order and
according to any visit process. Furthermore, if we want a lower bound to the
number of nodes reachable from $I_0$ in the percolation graph, we can choose to
focus only on vertices reachable through a subset of all possible paths, and,
in particular, we can restrict ourselves to paths that are easier to analyze.
In our analysis we will only consider paths that alternate between using a
bounded number of local edges and one bridge edge.

\subsection{Local clusters on the ring} \label{ssec:localclusters}

Given  a one-dimensional small world graph $G = (V,E= E_1 \cup E_2)$ where
$(V,E_1)$ is a cycle, a probability $p$, and  a vertex $v\in V$, we call the
\emph{local cluster} $\LC(v)$ the set of nodes that are reachable from $v$
using only local edges (that is, edges of $E_1$) that are  in the percolation
graph $G_p$ of $G$.

\begin{fact}\label{fa:exp_loc_ub}
\label{lem:local_cluster_complete}
If $G=(V,E)$ is a one-dimensional small-world graph and $p$ is a percolation
probability, for every $w\in V$, $ \Expcc{|\LC(w)|} \leq \frac{1+p}{1-p}$, and
this  bound becomes tight as the ring size tends to $\infty$.
\end{fact}
 
For technical reasons that will become clear later, when we explore the
percolation graph $G_p$ to estimate the size of its connected components, we do
not want to follow too many consecutive local edges. To analyze the effect of
this choice, it will be useful to have a notion of \emph{$L$-truncated local
clusters},  that we formalize below.

\begin{definition}[$L$-truncated local cluster] \label{def:trunc-local} 
Let $G =(V,E = E_1 \cup E_2)$ be a one-dimensional small-world graph, where
$(V,E_1)$ is a cycle, and the edges of $E_1$ are called ``local edges''. Let
$L$ a positive integer distance parameter, and $p$ be a percolation
probability.

The $L$-\emph{truncated local cluster} of $v \in V$ is the set of vertices
reachable from $v$ in the percolation graph $G_p$ using at most $L$  activated
local edges.
\end{definition}

The next fact provides the expected size of an $L$-truncated local cluster.

\begin{fact} 
If $G=(V,E)$ is a one-dimensional small-world graph and $p$ is a percolation
probability, for each node $v \in V$, the size $ \AN^L(v)$ of its $L$-truncated
local cluster $\LC ^L (v)$ satisfies the following
\begin{equation} \label{eq:truncated_LC}
    \Expcc{| \AN^L(v)|} =\frac{1+p}{1-p}-\frac{2p^{L+1}}{1-p}\,.
\end{equation}
\label{lemma:local_cluster}
\end{fact}
\begin{proof}
For any positive integer $L$ and  any  node $v \in V$, we define the random
variable $\RN^L(v)$ as the subset of nodes such that: they  are located at the
right of  $v$ at ring distance less than $L$; they will be infected by $v$
according to the \SIR process considering  only the ring edges (here, we
exclude $v$ from this set).  So, we have  

\begin{equation*}
\label{eq:pr_R_v=i}
    \Prc{|\RN^L(v)|=i}=
    \begin{cases} p^i(1-p) \hbox{ if $i < L $} \\ 
        p^L \hbox{ if $i=L$} \\ 
        0 \hbox{ otherwise.}
    \end{cases}
\end{equation*}

We observe that $|\RN^L(v)|$ is a well-known geometric random
variable\footnote{in our setting, the variable may assume value $0$.} with a
``cutoff'' and  it easily holds that 

\begin{equation*}
    \sum_{i=0}^{L}\Prc{ |\RN^L(v)|=i}=(1-p)\sum_{i=0}^{L-1}p^{i}+p^{L}=(1-p)\frac{1-p^L}{1-p}+p^{L}=1\,.
\end{equation*}
For any positive integer $L$ and  any  node $v \in V$, we also define the
``left-side'' random variable $\LN^L(v)$  indicating the  nodes in the ring
that are located at the left of  $v$  that are infected by $v$ according to the
local cluster with cut-off process. Clearly, $|\LN^L(v)|$ has the same
distribution of $|\RN^L(v)|$.  So, we can define  $\AN^L(v)$ as the overall set
of the local cluster of a node $v$ with cutoff $L$ including the node $v$
itself, i.e., 
\begin{equation*}
   | \AN^L(v)|  \, = \, |\RN^L(v)|  + |\LN^L(v)|+1\,. \label{eq:T_v>L_v+R_v}
\end{equation*}
So, since $\LN^L(v)$ and $ \RN^L(v)$ have the same distribution of probability,
\begin{align*}
    \Expcc{ \AN^L(v)} & =\Expcc{ \LN^L(v)+ \RN^L(v)+1}=2\Expcc{ \RN^L(v)}+1=2(1-p)\sum_{i=1}^{L-1}ip^i+2Lp^L+1 \\ &=2\frac{(L-1)p^{L+1}-Lp^{L}+p}{(1-p)}+2Lp^L+1 \\
    &=\frac{p+1}{1-p}+2\frac{(L-1)p^{L+1}-Lp^{L+1}}{1-p}=\frac{p+1}{1-p}-\frac{2p^{L+1}}{1-p}\,.
\end{align*}
\end{proof}

\subsection{Galton-Watson branching processes}
\label{sec:GW}

Our analyses of the bond percolation process will make use of a reduction to the analyses
of appropriately defined branching processes.

\begin{definition}[Galton-Watson Branching Process]
\label{def:GW}
Let $W$ be a non-negative integer random variable, and let $\{ W_{t,i}\}_{t\geq
1, i\geq 1}$ be an infinite sequence of independent identically distributed
copies of $W$. The \emph{Galton-Watson branching process} generated by the
random variable $W$ is the process $\{ X_t \}_{t\geq 0}$ defined by $X_0=1$ and
by the recursion
\[
	X_t = \sum_{i=1}^{X_{t-1}} W_{t,i}
\]
All properties of the process $\{ X_t \}_{t\geq 0}$ are captured by the process
$\{B_t\}_{t \geq 0}$  defined by the recursion
\[
	B_t = \left\{
		\begin{array}{ll}
			1, & t = 0;\\
			B_{t-1} + W_t - 1, & t > 0\ \text{and}\ 
			B_{t-1} > 0;\\
			0, & t > 0\ \text{and}\ B_{t-1} = 0.
		\end{array}\right.
\]
where $W_1,\ldots,W_t,\ldots$ are an infinite sequence of independent and
identically distributed copies of $W$. In the following, when we refer to the
Galton-Watson process generated by $W$ we will always refer to $\{B_t\}_{t \geq
0}$.

We define $\sigma=\min\{t>0: B_t=0\}$ (if no such $t$ exists we set
$\sigma=+\infty$) and notice that, for $T< \sigma$, we have $B_T = \sum_{t=1}^T
W_t -T$. 
\label{def:branchingprocess}
\end{definition}

Galton-Watson processes are characterized by the following important threshold
behavior.

\begin{theorem}[\cite{AS16}, Section 10.4]
\label{thm:branchingprocess}

Let $\{B_t\}_{t \geq 0}$ be a  Galton-Watson process with  integer random variable $W$. Then:
\begin{enumerate}
\item For every   constant $\varepsilon>0$, if  $\Expcc{W}<1-\varepsilon$,
the process dies out $(\sigma < +\infty)$ with probability $1$; 
\item For every   constant $\varepsilon>0$, if    $\Expcc{W}>1+\varepsilon$,
the process diverges, i.e.,    a constant $c>0$ exists  such that
$\Prc{\sigma=+\infty}\geq c$.
\end{enumerate}
\end{theorem}

When the expectation of $W$ is over the threshold, the above theorem implies
that, with probability $c>0$, for every time $t$ we have $B_t >0$. The next
lemma shows that, if the variance of $W$ is bounded then, with constant
positive  probability, the value of $B_t$ is not only positive, but it is at
least $\Omega(t)$.

\begin{lemma}
\label{lm:bp.lowerbound}
Let $\varepsilon$ be any positive constant, and consider a Galton-Watson
process $\{B_t\}_{t \geq 0}$ with a non-negative integer random variable $W$
with $\Expcc {W} \geq 1 + \varepsilon$ and with finite variance, i.e.,
${\mathbf Var} (W) \leq U$ for some positive constant $U$. Then there is a
constant $c'$ that depends only on $\varepsilon$ and a constant $t_0$ that
depends only on $\varepsilon$ and $U$ such that, for every $t \geq t_0$, $\Prc{
B_t \geq (\varepsilon t)/{2} } \geq c'$.
\end{lemma}
\begin{proof}
By definition of  Galton-Watson process, if $W_1,\ldots,W_t$ are mutually
independent copies of $W$, \[ \Prc { B_t < \frac {\varepsilon t}{2} } = \Prc {
B_t = 0 \  \vee \  \sum_{i=1}^t W_i < t + \frac {\varepsilon t}2 } \leq \Prc {
B_t = 0 } + \Prc { \sum_{i=1}^t W_i < t + \frac {\varepsilon t}2 } \, ,  \]
where the second inequality follows by a simple   union bound. From Theorem
\ref{thm:branchingprocess}, there is a constant $c = c(\varepsilon)$ such that
\[
\Prc { B_t = 0 } \leq 1 - c \,.
\]
From Chebyshev's inequality (Theorem~\ref{thm:chebyshev}), \[ \Prc { \sum_{i=1}^t W_i < t
+ \frac{t\varepsilon }{2}  } \leq \frac {4U}{\varepsilon^2 t} \leq \frac c2 \,
, \] where the second inequality holds if $t \geq \frac {8U}{\varepsilon^2 c}$.
The lemma then follows  setting $c' = c/2$ and $t_0 =  {8U}/{\varepsilon^2 c}$.
 
\end{proof}

\begin{lemma}
Let $\varepsilon$ be any positive constant, and consider a Galton-Watson process $\{B_t\}_{t \geq 0}$ with a non-negative integer random variable $W$ such that $0 \leq W \leq M$ and $\Expcc{W}\geq 1+\varepsilon$. Then, for any $\gamma\geq 4M^2/\varepsilon^2$ and for any $t_0 \geq 1$, we have
\begin{equation*}
    \Prc{B_{n+t_0}>0 \mid B_{t_0}\geq \gamma \log n} \geq 1-\frac{1}{n}.
\end{equation*}
\label{lem:bp:growingfromlognton}
\end{lemma}

\begin{proof} Let $\varepsilon'=\varepsilon/2$. For any $\ell >0$ and any $i \geq 1$, consider the event
\begin{equation*}
    A = \{B_{i+t_0} \geq (1+\varepsilon')\ell \mid B_{i-1+t_0} \geq \ell \}.
\end{equation*}
If we consider $\ell $ generic i.i.d. copies of the random variable $W$, $W_1,\dots,W_\ell$ we have that
\begin{equation}
    \Prc{A} \geq \Prc{\sum_{i=1}^\ell W_i \geq (1+\varepsilon')\ell} \geq 1-e^{-\frac{\varepsilon^2}{2M^2}},
    \label{eq:sum_M_copiesofW}
\end{equation}
where the last inequality follows from the Hoeffding bound. 

We notice that, if we define the events
\begin{equation*}
    A_{i} = \{B_{i+t_0} \geq (1+\varepsilon')^i \gamma \log n \mid B_{i-1+t_0} \geq (1+\varepsilon')^{i-1}\gamma \log n\},
\end{equation*}
then, for the chain rule, we will have
\begin{equation}
    \Prc{B_{n+t_0} >0 \mid B_{t_0}\geq \gamma \log n} \geq \prod_{i=1}^{n}\Prc{A_i}.
    \label{eq:prob_queuenotempty}
\end{equation}
For \eqref{eq:sum_M_copiesofW}, we have
\begin{equation}
    \Prc{A_i} \geq 1-e^{-\frac{\varepsilon^2}{2M^2}\gamma\log n} \geq 1-\frac{1}{n^2},
    \label{eq:prob_Ai}
\end{equation}
where the last inequality follows since $\gamma \geq 4M^2/\varepsilon^2$. So, for \eqref{eq:prob_Ai} and \eqref{eq:prob_queuenotempty},
\begin{equation*}
    \Prc{B_{n+t_0}>0 \mid B_{t_0} \geq \gamma \log n} \geq \left(1-\frac{1}{n^2}\right)^{n} \geq 1-\frac{1}{n}.
\end{equation*}
\end{proof}

\subsection{Further mathematical tools} \label{app:maths}
\begin{definition}[Stochastic dominance]
\label{def:stochastic_domination}
Let $X$, $Y$ be two real-valued random variables. Then, $Y$ is said to
\emph{stochastically dominates} $X$ ($X \preccurlyeq Y$) if, for every $x \in
\mathbb{R}$, $\Prc{X>x} \leq \Prc{Y > x}$.
\end{definition}

\begin{definition}[Coupling] \label{def:coupli}
Let $X_1$ and  $X_2$ be two random variables that are defined on the
probability spaces $(\Omega_1,F_1, P_1)$ and $(\Omega_2,F_2, P_2)$,
respectively. Then a \emph{coupling} between $X_1$ and  $X_2$ is formed by: i)
a  probability space $(\Omega, F, P)$, and ii)  a vector random variable $W =
(Y_1, Y_2)$ defined over this space such that: the marginal distribution of
$Y_1$ equals the distribution of  $X_1$, while the marginal distribution of
$Y_2$ equals that of $X_2$.
\end{definition}

Devising a coupling is often an effective way to show stochastic dominance, as
formally stated below.

\begin{lemma}
\label{lem:stochastic_domination_coupling}
A random variable $X_1$ is   dominated by a random variable $X_2$ if and only
if there exists a coupling $(Y_1, Y_2)$ between $X_1$ and  $X_2$ such that
$\Prc{Y_1 \leq Y_2}=1$. 
\end{lemma}

\begin{lemma}[Wald's equation, \cite{wald1944}]
\label{lem:Wald-equation}
Let $\{X_n\}_{n \in \mathbb{N}}$ be an infinite  sequence of real-valued,
mutually independent, and identically distributed random variables.   Let $N$
be a non-negative integer-value random variable that is independent of the
sequence $\{X_n\}_{n \in \mathbb{N}}$. Suppose that $N$ and $X_n$ have finite
expectations. Then,
\begin{equation*}
    \Expcc{X_1+\dots+X_N} = \Expcc{N} \cdot \Expcc{X_1}\,.
\end{equation*}
\end{lemma}

\begin{theorem}[Hoeffding's Inequality]
\label{thm:hoeff}
Let $X_1, \dots, X_n$ be independent random variables with $X_i$ strictly
bounded in $[a_i, b_i]$ for every $i \in \{1, \dots, n\}$, where $- \infty <
a_i \leq b_i < + \infty $. Let $S = \sum_{i=1}^{n} X_i$. Then,
\begin{equation*}
    \Prc{|S - \Expcc{S}| \geq t} \leq 2 \exp{\left(\frac{-2t^{2}}{\sum_{i=1}^{n}(b_i - a_i)^{2}}\right)}\,.
\end{equation*}
\end{theorem}

\begin{definition}[Dependency graph]
\label{def:dependency_graph}
Let $\{Y_{\alpha}\}_{\alpha \in \mathcal{A}}$ be a sequence of random
variables. A \emph{dependency graph} for $\{Y_\alpha\}_{\alpha \in
\mathcal{A}}$ is a graph $\Gamma$ with vertex set $\mathcal{A}$ such that if
$\mathcal{B}\subseteq \mathcal{A}$ and $\alpha \in \mathcal{A}$ is not
connected by an edge to any vertex in $\mathcal{B}$, then $Y_\alpha$ is
independent of $\{Y_\beta\}_{\beta \in \mathcal{B}}$.
\end{definition}

The sum of a set of random variables, with mutual correlations that can be
described by a dependency graph, enjoys of the following concentration result.

\begin{theorem}[\cite{Svante04}]
\label{thm:svante}
Suppose that $X$ is a random variable such that $X = \sum_{\alpha \in
\mathcal{A}}Y_\alpha$, where, for every $\alpha \in \mathcal{A}$, $Y_\alpha
\sim \text{Be}(p)$, for some fixed $p \in (0,1)$. Let $N = |\mathcal{A}|$.
Then, for every $t \geq 0$, 
\begin{equation*}
    \Prc{X \leq \Expcc{X}-t}\leq \exp{-\frac{8t^2}{25\Delta_1(\Gamma) Np}},
\end{equation*}
where $\Gamma$ is the dependency graph of $\{Y_\alpha\}_{\alpha \in
\mathcal{A}}$, $\Delta(\Gamma)$ is the maximum degree of $\Gamma$, and
$\Delta_1(\Gamma)=\Delta(\Gamma)+1$.
\end{theorem}

\begin{theorem}[Chebyshev's inequality]
\label{thm:chebyshev}
Let $X$ be a real-valued random variable with bounded expectation and variance.
Then, for every real $a>0$,
\begin{equation*}
    \Prc{|X-\Expcc{X}|\geq a}\leq \frac{\Varcc{X}}{a^2}\,.
\end{equation*}
\end{theorem}

\section{The \texorpdfstring{$\SWGm(n,q)$} \, Model above the Threshold} \label{sec:wup}

%\andy{Qui bisogna dire che: in questa sezione proveremo il claim   I dei thm ... e .... Per fare questo, procederemo in questo modo. In Subsection, ... (dare quindi una roadmap chiara di tutta la Section...}

%\andy{Ricordiamoci che da questa section in poi, utilizziamo definizioni e risultati tecnici che sono solo nell'appendice A e non sono necessariamente puntati dal body del paper...  quindi vanno puntati  ogni volta che li usiamo...}

In this section, we will prove Claim 1 of Theorem \ref{thm:gen_overthershold} and Claim 1 of Theorem \ref{thm:intro-erdos-2}, respectively in Subsection \ref{ssec:proof:thm:gen_overthershold} and \ref{ssec:proof:thm:intro-erdos-2}. 
Before proceeding with the proofs of the theorems, we introduce two preliminary lemmas. In particular, in the Subsection \ref{ssec:bootstrap1} we present the proof of the Lemma \ref{le:L-visit} (already introduced in Section \ref{sec:overv-analysis}) while in Subsection \ref{sec:overv-analysis} we state and prove a further preliminary lemma.

In all this section, we will indicate with $V$ a set of $n$ nodes, with $G=(V,E)$ a graph sampled according the $\mathcal{SWG}(n,c/n)$ distribution, and with $G_p$ the percolation graph of $G$ with percolation probability $p$.
\subsection{Proof of Lemma \ref{le:L-visit}}
\label{ssec:bootstrap1}

This section  provides the full proof of  Lemma \ref{le:L-visit} we state   in Section \ref{sec:overv-analysis} to sketch our general analysis.
%The main result of this section can be formulated  as follows.

%\begin{lemma}\label{le:L-visit}
%Let $V$ be a set of $n$ nodes and let $s \in V$ be an \textit{initiator} node.
%For every $\varepsilon >0$ and $c>0$, and for every contagion probability   $p$
%such that
%\[
%\frac{\sqrt{c^{2} + 6c +1}-c-1}{2c} + \varepsilon \leq p \leq  1 \, , 
%\]
%there are positive parameters $L,k,\varepsilon'$, and $\gamma$, that depend only on $c$ and $\varepsilon$, such that the following holds.  Sample a graph $G = (V,E)$ according to the $\SWGm(n, c/n)$ distribution and let $G_p$ be the percolation graph of $G$ with parameter $p$.  Run the $L$-visit procedure in Algorithm~\ref{alg:L-visit} on input $(G,G_p,s)$: if $n$ is sufficiently large, for every sufficiently large $t$, at the end of the $t$-th iteration of the while loop it holds that \[ \Prc{|R \cup Q| \geq n/k \mbox{ \emph{OR} } |Q| \geq \varepsilon' t} \geq \gamma \, ,  \]
%where the probability is over both the randomness of the choice of $G$ from $\SWGm(n, c/n)$ and over the choice of the percolation graph $G_p$. \end{lemma}

%\begin{proof} [Proof of Lemma \ref{le:L-visit}]
For $t = 1, 2, \dots,$ let $Q_t$ be the set of nodes in the queue $Q$ at the
end of the $t$-th iteration of the while loop in Algorithm~\ref{alg:L-visit}
and let $Z_t$ be the number of nodes added to the queue $Q$ during the $t$-th
iteration. Notice that $|Q_0| = 1$ and 
\[
|Q_{t}| = 
\left\{
\begin{array}{cl}
0 & \mbox{ if } |Q_{t-1}| = 0 \\
|Q_{t-1}| + Z_t - 1 & \mbox{ otherwise}
\end{array}
\right.
\]
We next show that, as long as the overall number of visited nodes is below a
suitable constant fraction of $n$, the sequence $\{|Q_t|\}_t$ stochastically
dominates a diverging Galton-Watson branching process
(Definition~\ref{def:GW}). 

Let $k > 1$ be a constant and let $\tau = \inf\{t \in\mathbb{N} \,:\, |Q_t| + t
> n/k\}$ be the random variable indicating the first time the size of the queue
plus the number nodes in $R$ exceeds $n/k$. Consider any iteration $t < \tau$
of the while loop with $Q \neq \emptyset$, let $|Q \cup R| \leqslant n/k$ be
the  number of nodes in the queue or in the set $R$ at the beginning of the
while loop, and let $A$ be the set of nodes at distance larger than $L$ from
any node in $D \cup Q \cup R$ in the ring $(V, E_1)$, i.e.,  \[A = \{ v \in V \, | \,
d_{(V,E_1)}(D \cup Q \cup R,v) \geq L+1  \}\,.\] Observe that  there are at most $2
L(n/k+\log^4 n)\leq 4L(n/k)$ nodes at distance smaller than or equal to $L$ from a node in $D \cup Q \cup
R$ in $(V, E_1)$, so $|A| \geqslant n(1 - 4L/k)$.

Let $w$ be the node dequeued at the $t$-th iteration of the while loop and let
$x_1, \dots, x_{|A|}$ be the nodes in $A$. For every $i = 1, \dots, |A|$, let
$X_i$ be the random variable counting the number on nodes added to the queue
``through'' node $x_i$ during the current iteration of the while loop at line \ref{line:alg:L-visit:while}
of Algorithm~\ref{alg:L-visit}. Observe that $X_i$ is either zero (if $x_i$ is
not a bridge neighbor of $w$ in the percolation graph, or if $x_i$ is a bridge
neighbor of $w$ but it is not \good at its turn in line \ref{line:alg:L-visit:if}) or it is equal to
the size of the truncated local cluster centered at $x_i$. Moreover, $Z_t
\geqslant \sum_{i=1}^{|A|} X_i$.

Now observe that the edge $\{w,x_i\}$ exists in the percolation graph $G_p$
with probability $pc/n$, independently of the other edges: we can use the
principle of deferred decisions here, since the existence or not of each such
edge was never observed before $w$ was extracted from the queue. Moreover,
since each node in $A$ has at most $4L$ other nodes of $A$ at ring distance
less than $2L$, the probability that $x_i$ is a bridge neighbor of $w$ in $E_p$
\textit{and} it is \good for the subset $Q \cup R$ in $G_{\text{SW}}$ at its
iteration in the for loop at line \ref{line:alg:L-visit:for} of Algorithm~\ref{alg:L-visit} is at least
$pc/n (1-pc/n)^{4L}$, i.e.  the probability that $x_i$ is a bridge neighbor for
$w$ in $E_p$ and all the nodes in $A$ at ring distance at most $2L$ from $x_i$
are not.  From~\eqref{eq:truncated_LC} it follows that 
\[
\Expcc{X_i} \geqslant pc/n (1-pc/n)^{4L} \Expcc{|LC^L(x_i)|}
= pc/n (1-pc/n)^{4L} \left(\frac{1 + p - 2p^{L+1}}{1-p}\right)\,.
\]
Thus, the expected number of new nodes added to the queue in an iteration of the
while loop is 
\begin{align*}
\Expcc{Z_t \;|\; |Q_{t-1}| > 0, \, \tau > t} 
& \geqslant n (1 - 4L/k)(pc/n)(1-pc/n)^{4L} \left(\frac{1 + p - 2p^{L+1}}{1-p}\right) \\
& \geqslant pc (1 - 4L/k) (1-4Lpc/n) \frac{1 + p - 2p^{L+1}}{1-p} \\
& = \frac{pc(1+p)}{(1-p)} \left( 1 - \bigO(p^L) - \bigO(L/k) - \bigO(L/n) \right) \, .
\end{align*}

The critical value $pc(1+p)/(1-p) = 1$ is achieved for $p = \frac{\sqrt{c^{2} +
6c +1}-c-1}{2c}$. So, for every choice of $\varepsilon \in (0, 1 -
\frac{\sqrt{c^{2} + 6c +1}-c-1}{2c})$, if $p = \frac{\sqrt{c^{2} + 6c
+1}-c-1}{2c} + \varepsilon$ we can choose sufficiently large constants $L$ and
$k$ such that, whenever $n$ is large enough, $\Expcc{Z_t \;|\; |Q_{t-1}| > 0,
\, t < \tau } \geqslant 1 + \varepsilon'$, with $\varepsilon'> 0$.  

At each while iteration of Algorithm~\ref{alg:L-visit}, the node $w$ extracted
from $Q$ has at most $\Bin(n,pc/n)$ bridge neighbors in $E_p$. Since each \good
node is also a bridge neighbor in $E_p$ for the node $w$ extracted from the
queue, we further have

\begin{equation*}
\Varcc{Z_t \mid |Q_{t-1}|>0, \tau>t} \leq \Expcc{Z_t^2 \mid |Q_{t-1}|>0, \tau>t}
\leq (2L)^2 \Expcc{\Bin(n,pc/n)^2}\leq 4L^2.
\end{equation*}
Hence, we can define a Galton-Watson branching process $\{B_t\}_t$ according to
Definition~\ref{def:GW} as follows: 

\noindent
- If at the beginning of the $t$-th iteration of the while loop it holds that
  $|Q \cup R| \leqslant n/k$ then we consider an arbitrary set $\hat{A}$ such
  that each node in $\hat{A}$ is at distance larger than $L$ from any node in
  $Q \cup R$, and the size of $\hat{A}$ is exactly $\lceil n (1 - 4L/k)
  \rceil$. In this setting, we define $W_t$ as the number of new nodes in
  $\hat{A}$ added to the queue $Q$ during the $t$-th iteration of the while
  loop. 

\noindent
- Otherwise (i.e., if at the $t$-th iteration the size $|Q \cup R| > n/k$) then
  consider two arbitrary disjoint sets of nodes $\hat{Q}$ and $\hat{R}$ with
  $|\hat{Q} \cup \hat{R}| \leqslant n/k$ and an arbitrary set $\hat{A}$ such
  that each node in $\hat{A}$ is at distance larger than $L$ from any node in
  $\hat{Q} \cup \hat{R}$ and the size of $\hat{A}$ is exactly $\lceil n (1 -
  4L/k) \rceil$. In this setting, we define   $W_t$ as the number of new nodes
  in $\hat{A}$ that would be added to the queue if at the beginning of the
  $t$-th iteration of the while loop it was $Q = \hat{Q}$ and $R = \hat{R}$.

Notice that $\{W_t\}_t$ is a sequence of i.i.d. random variables with
$\Expcc{W_t} > 1$ (thus, according to Theorem~\ref{thm:branchingprocess},
$\{B_t\}_t$ is a diverging branching process) and finite variance. Observe also
that the pair $(B_t,Q_t)$ is a coupling between the two considered processes
(see Definition~\ref{def:coupli} in Subsection \ref{app:maths}) such that, with probability
$1$, at each round $t$ either $|Q_t \cup R_t| > n/k$ or it holds that $Z_t
\geqslant W_t$. Thanks to Lemma~\ref{lem:stochastic_domination_coupling},  we
thus get that, at each round $t$,  
\[ 
\Prc{|R_t \cup Q_t| \geq n/k \mbox{ \emph{OR} } |Q_t| \geq \varepsilon' t}
\geqslant \Prob{}{B_t \geqslant \varepsilon' t} . 
\]
The lemma then follows by applying Lemma~\ref{lm:bp.lowerbound}
in Subsection \ref{sec:GW}.
%\end{proof}

\smallskip \noindent 
\textbf{Remark.}
The lemma above implies that
the nodes visited by the end of the sequential \SIRimmtrunc in
Algorithm~\ref{alg:L-visit} reaches size at least $n/k$, with probability at least $\gamma$. This result thus shows a linear lower bound on the  size of the connected component of the source $s$ in $G_p$. 

%\section{The \SWG Model over the Threshold, Revisited} \label{sec:boot}

\subsection{Parallelization of the sequential BFS visit}
\label{ssec:parallelizatiobfs}

In this section, we strenghten the analysis of the visit in the graph $G_p$, when the percolation probability $p$ is over the
threshold. 

Our goal here  is to  prove that, if we explore the connected components of $\log n$ nodes taken arbitrarily in the graph, then this process leads us, \emph{w.h.p.}, to the visit of a linear fraction of the nodes in the percolated graph, within $\Theta(\log n)$ number of hops.

We follow an approach that proceeds along the general lines of Subsection~\ref{ssec:bootstrap1}, albeit with important differences and some technical challenges. We begin by introducing Algorithm~\ref{alg:par_L-visit} below, which is partly ``parallel'' extension of the sequential BFS visit described by Algorithm~\ref{alg:L-visit}. We assume Algorithm~\ref{alg:par_L-visit} is run on an input $(G, G_p, I_0,D_0)$, where $I_0$ is an arbitrary subset of initiators and $D_0 \subseteq V \setminus I_0$ is a set of deleted nodes. 

\begin{algorithm}
\caption{\SIRimmtruncpar}
\small{
\textbf{Input}: A small-world graph $G_{\text{SW}} = (V, E_{\text{SW}})$ and a
subgraph $H$ of $G_{\text{SW}}$; a set of initiators $I_0
\subseteq V$; a set of deleted nodes $D_0 \subseteq V \setminus I_0$.
\begin{algorithmic}[1]
	\State $Q = I_0$
	\State $R = \emptyset$
	\State $D = D_0$
%	\While{$Q \neq \emptyset$ and $Q < \beta \log n$}
%	    \State $w = \dequeue(Q)$
%	    \State $R = R \cup w$ 
 %	            \For {each bridge-neighbor $x$ of $w$ in $H$}
 %	              \If {$x$ is \good for $R \cup Q$}		\Comment{We are 
% 	              using Definition \ref{def:free_node}}
% 	                  \For{each node $y$ in the $L$-truncated local cluster $\LC ^L (x)$} 
% 	                       \State $\enqueue(y,Q)$
% 	                  \EndFor
% 	              \EndIf
% 	            \EndFor
% 	 \EndWhile
     \While{$Q \neq \emptyset$}\label{line:parvisit:while}
	   \State $A = R\cup Q \cup D$	\Comment{This is the overall set of nodes 
	   visited so far}
	   \State $X = \text{bridge-neighors}(Q)$ \Comment{Set of bridge-neighbors in $H$
	   of nodes in $Q$}
       \State $Q' = Q$ 
       \State $Q = \emptyset$
	   \While{$Q' \neq \emptyset$ }
	    \State $w = \dequeue(Q')$
	    \State $R = R \cup w$ 
 	            \For {each $x\in X$}
 	              \If {$x$ is \good for $(X, A)$}  \Comment{We are 
 	              using Definition \ref{def:good_nodes_par}}
 	           \For{each node $y$ in the $L$-truncated local cluster $\LC ^L (x)$} 
 	                       \State $\enqueue(y,Q)$
 	                  \EndFor
 	              \EndIf
 	            \EndFor
 	 \EndWhile
     \EndWhile\label{line:parvisit:endwhile}
\end{algorithmic}}
\label{alg:par_L-visit}
\end{algorithm}

In the remainder of this section, $Q_t$ and $R_t$ respectively denote the
subsets $Q$ and $R$ at the end of the $t$-iteration of the while loop in
line~$10$. Consistently with the notation used in Section~\ref{sec:wup}, we
also let $S_t = V \setminus ( R_t\cup Q_t)$.

\begin{lemma}\label{thm:exponential_growth}
Let $V$ be a set of $n$ nodes, $I_0 \subseteq V$ a set of initiators and $D_0 \subseteq V \setminus I_0$ a set of deleted nodes such that $|D_0| \leq \log^4 n$. For every
$\varepsilon>0$, $c>0$ and for every contagion probability $p$ such that
\[
\frac{\sqrt{c^{2} + 6c +1}-c-1}{2c} + \varepsilon \leq p \leq 1,
\] 
there are positive parameters $L,k,\beta,\delta$ that depend only on $c$ and
$\varepsilon$ such that the following holds. Sample a graph $G=(V,E)$ according
to the $\mathcal{SWG}(n,c/n)$ distribution, and let $G_p$ be the percolation
graph of $G$ with parameter $p$. Run the \SIRimmtruncpar in
Algorithm~\ref{alg:par_L-visit} on input $(G,G_p,I_0,D_0)$: in every iteration $t
\geq 1$ of the while loop at line \ref{line:parvisit:while} in Algorithm~\ref{alg:par_L-visit}, for
every integer $i \geq \beta \log n$ and $r \geq 0$ such that $i+r\leq n/k$:
\begin{equation}
    \Prc{|{Q}_{t}|\geq (1+\delta)i \mid \condpar}\geq 1-\frac{1}{n^2} \,     .
    \label{eq:lem:exp_growth}
\end{equation}
\end{lemma}

In what follows, we introduce some definitions and lemmas preliminary to the proof of the above lemma.
  
We first need to slightly revisit the notion of free node given by
Definition~\ref{def:free_node} for the Sequential \SIRimmtrunc, adapting it to
the second phase of Algorithm~\ref{alg:par_L-visit} .

\begin{definition}[\good nodes]\label{def:good_nodes_par} 
Consider $X, A\subset V$. A node $x\in X$ is \emph{free} for the pair 
$(X, A)$ if the following holds:
\begin{enumerate}
	\item $x$ is at distance on the ring $E_1$ at least $L+1$ from every 
	node in $A$; 
	\item $x$ is at distance on the ring $E_1$ at least $2L+1$ from every 
	other node in $X$.
\end{enumerate}
If $Q_{t-1}$ is the queue $Q$ at the end of the $(t-1)$-th iteration of 
the while loop in line 10 of Algorithm \ref{alg:par_L-visit}, we denote 
by $X_t$ the set of bridge-neighbors (w.r.t. $E_p$) of nodes in $Q_{t-1}$, while 
$Y_t\subseteq X_t$ is the subset of free nodes for the pair $(X_t, 
R_{t-1}\cup Q_{t-1})$.
\end{definition}
Definition~\ref{def:good_nodes_par} implies the following properties for the
generic, $t$-th iteration of the while loop at line \ref{line:parvisit:while} of
Algorithm~\ref{alg:par_L-visit}. At the beginning of the iteration, we
initialize set $Q_t=\emptyset$ and we consider the set $Y_t$ of \good nodes for
$(X_t, R_{t-1}\bigcup Q_{t-1})$. For each node $x \in Y_t$, we add to the
\emph{queue} $Q_t$ the set $|\LC^L(x)|$ of the $L$-truncated local cluster of
node $x$ (Definition~\ref{def:good_nodes_par}): the set $Q_t$ can thus be seen
as the nodes that $Q_{t-1}$ infects via its free bridge-neighbors in round $t$.
The process stops in the first round $\tau$, for which queue $Q_{\tau}$ is
empty, i.e. $\tau=\{t>0:Q_t=\emptyset \}$. Hence, for each $t \leq \tau$
\begin{equation*}
    Q_t=\bigcup_{x\in Y_t}\LC^L(x),
\end{equation*}
and, if we label  the nodes in $Y_t$ as $1, 2, \ldots, |Y_t|$,  we get  
\begin{equation}
    \label{eq:Z_t_as_sum}
    |Q_t|=|\LC^L(1)|+\dots+|\LC^L(|Y_t|)|
\end{equation}
\label{def:Z_t_par}
since the subsets  $\LC^L(j)$'s are mutually disjoint from Definition 
\ref{def:good_nodes_par}.

In the remainder, we denote by $A \subseteq S_{t-1}$ the subset of nodes that
are at  ring distance at least $L+1$ from each node in $Q_{t-1}\bigcup R_{t-1}\cup D$
at the end of of the $(t-1)$-th iteration of the while loop at line 10 of
Algorithm~\ref{alg:par_L-visit}.  Under the hypotheses of
Lemma~\ref{thm:exponential_growth} on $|Q_{t-1}| + |R_{t-1}|$, we have:
\begin{equation}
    n\left(1-\frac{4L}{k}\right) \leq |A|\leq n-(|R_{t-1}|+|Q_{t-1}|+|D|)\leq n\,.
    \label{eq:A}
\end{equation}
Moreover, we can write
\begin{equation}
\label{eq:good_nodes_as_sum}
   |Y_t| = \sum_{x \in A}Y(x),
\end{equation}
where each $Y(x)$ is a Bernoulli random variable, whether  node $x \in A$ is
\good. By a standard argument, we can bound the conditional expectation of
$|Y_t|$ as follows.

\begin{fact}
\label{fact:exp_value_Y}
Under  the hypotheses of Lemma \ref{thm:exponential_growth} we have
\begin{equation}
    i\cdot p \cdot c \cdot \left(1-\frac{8L}{k}\right) \leq \Expcc{|Y_{t}| \mid
    \condpar} \leq i\cdot p \cdot c\, .
    \label{eq:fact_bound}
\end{equation}
\end{fact}

\begin{proof}
We know a  node $x$ is \good if it is connected via  a bridge in $\percgraph$
with at least one node in $Q_{t-1}$ \textit{and} no node, within ring distance
$2L$ from $x$, is connected via a bridge in $\percgraph$ with a node in
$Q_{t-1}$. Therefore,

\begin{equation} \label{eq:lowbounY}
    \left(1-\left(1-p  q\right)^{i}\right) \left(1- p q\right)^{4Li} \leq \Prc{Y(w')=1} \leq  i  p  q\, .
\end{equation}
Using the assumptions $i + r \leq n/k$, $p \cdot q \cdot n < 1$ and the
inequalities 
\[
(1+y)^n \leq \frac{1}{1 -n y}, \ \ n \in \mathbf{N}, \ y \in [-1,0]  \ \mbox{ and }  \ (1+y)^r \geq 1 + y r\, , \   y \geq -1, \ r \in \mathbf{R} \setminus (0,1) \, ,
\]  
we get that the LHS of \eqref{eq:lowbounY} is not smaller than 

\begin{align*}
   & (1-p q)^{i} (1- p  q)^{4Li} \geq \left(1-\frac{1}{1+ i p q}\right)\cdot \left(1-4L p q\right)\geq 
   \geq i p q \left(1 - \frac{i p  q}{1 + i p q}\right) \left(1-4L i  pq\right) \geq \\
   & \geq i pq (1 - i pq)  (1 - 4Li  p  q) \geq i  p q  (1 -5L p q) \geq i  p q \left(1 - \frac{5L p}{k}\right)\, .
\end{align*}

Consequently, from~\eqref{eq:A}, \eqref{eq:good_nodes_as_sum},
and~\eqref{eq:lowbounY} we get~\eqref{eq:fact_bound}.
\end{proof}

Our next step is to prove that w.h.p., $|Y_{t}|$ does not deviate much from its
expectation. As we noted earlier, $|Y_{t}|$ can be expressed as the  sum of
Bernoulli random variables $Y(x)$ with $x \in A$.  Unfortunately, these
variables are not    mutually independent: for instance, $Y(x)= 1$ implies
$Y(x')=0$ for every other  $x'\in A$ that lies  within ring distance  $2L$ in
$\percgraph$ from $x$.  However, we are able to prove the following
concentration bound, by leveraging the key fact that the variables above only
have \emph{local}, mutual correlations. 

\begin{lemma}
\label{lemma:Y_t__not_too_small}
Under the hypotheses of Lemma~\ref{thm:exponential_growth}, if $\beta$ and $k$
are sufficiently large, we have
\begin{equation*}
		\Prc{|Y_t| \geq i\cdot p \cdot c \cdot 
		\left(1-\frac{9L}{k}\right)\mid \condpar} \geq 1 - 
		\frac{1}{n^3} \, .
\end{equation*}
\end{lemma}
\begin{proof}
Recall that for any $x \in A$, $Y(x)$ is the Bernoulli random variable that
indicates whether $x$ is \good. Since $x$ is free only if it is connected via a
bridge in $G_p$ with at least one node in $Q_{t-1}$ we have, for every $x\in
A$:
\begin{equation}
\label{eq:bound_f}
    f = \Prc{Y(x)=1}\leq \frac{i \cdot p \cdot c}{n}
\end{equation}

Now, for any $x\in A$, denote by $\mathcal{N}_{E_1}^{2L}(x)$ the set of nodes
that are within ring distance $2L$ from $x$. For any other $x' \in A$,
Definition \ref{def:good_nodes_par} implies that $Y(x)$ and $Y(x')$ are
mutually dependent if and only if $\mathcal{N}_{E_1}^{2L}(x) \cap
\mathcal{N}_{E_1}^{2L}(x') \ne \emptyset$. Hence, we can bound the maximum
number of random variables $Y(x')$ that are correlated with $Y(x)$ as follows.
Consider $\mathcal{N}_{E_1}^{2L}(x) = \{x - 2L, \dots, x, \dots, x + 2L\}$. If
$\mathcal{N}_{E_1}^{2L}(x) \cap \mathcal{N}_{E_1}^{2L}(x) \neq \emptyset$ for
some other $x' \in A$, it must be the case that either $x < x' \land x' - 2L
\leq x + 2L$, or $x' < x \land x - 2L \leq x' + 2L$. The former happens for
every $x'$ such that $x < x' \leq x + 4L$ (notice that exactly $4L$ nodes can
meet this condition), while the latter happens for every $x'$ such that $x - 4L
\leq x' < x$ (again, exactly $4L$ nodes can meet this condition).  It thus
follows that, for a fixed $Y(x)$, at most $8L$ other random variables can be
correlated with  $Y(x)$. This property can be described by the dependency graph
$\Gamma$ on the subset $\{ Y(x) \}_{x \in A}$ (see
Definition~\ref{def:dependency_graph} in Appendix~\ref{app:maths}). In our
case, the maximum degree $\Delta$ of the dependency graph  is $8L$, whence we
have $\Delta_1(\Gamma) =  8L+1$ in Definition~\ref{def:dependency_graph}.
 
We can thus apply Corollary $2.4$ in~\cite{Svante04} (see
Theorem~\ref{thm:svante} in Appendix~\ref{app:maths})). In more detail, we use
Fact~\ref{fact:exp_value_Y}, \eqref{eq:A} and~\eqref{eq:bound_f}, the
assumption $p \cdot c < 1$, and  we apply Theorem~\ref{thm:svante} to complete
the proof:
\begin{align}
    & \Prc{|Y_t|\leq i p  c \left(1-\frac{9L}{k}\right)\mid \condpar} \leq \notag \\
    & \leq \Prc{|Y_t|\leq i p  c \left(1-\frac{8L}{k}\right) \left(1-\frac{L}{k}\right)\mid \condpar} \leq  \notag \\
    & \leq \Prc{|Y_t|\leq \Expcc{|Y_{t}| \mid \condpar} \left(1 - \frac{L}{k}\right)\mid \condpar} \leq \notag \\
    & \leq \exp{\left(\frac{-8(\Expcc{|Y_{t}|\mid \condpar})^2}{25(L^2/k^2)\cdot \Delta_1(\Gamma) \cdot |A| \cdot f}\right)}\leq \exp{\left(\frac{-8  i^2  p^2 c^2 \left(1 - \frac{8L}{k}\right)^2}{25(L^2/k^2)(8L+1)i p  c}\right)} \leq \notag \\
    & \leq \exp{\left(\frac{-8 i p c  \left(1 - \frac{8L}{k}\right)^2}{25(L^2/k^2)(8L+1)}\right)} \leq \exp{\left(\frac{-8 \beta \log n p  c \left(1 - \frac{8L}{k}\right)^2}{25(L^2/k^2)(8L+1)}\right)}\leq  n^{-3} \,,
    \label{eq:thm_applicato}
\end{align}
where the last equation holds whenever $k$ and $\beta$ are sufficiently large
constants (depending on $\varepsilon$).
\end{proof}

Now we are ready to conclude the proof of Lemma~\ref{thm:exponential_growth}.

\begin{proof}[Proof of Lemma \ref{thm:exponential_growth}]
Essentially, Lemma~\ref{lemma:Y_t__not_too_small} implies that, w.h.p.,
$|Y_{t}|$ is at least $i\cdot p\cdot q \cdot n$ up to a constant factor that
can be made arbitrarily close to $1$, provided constants $k$ and $\beta$ are
sufficiently large. Next, using~\eqref{eq:Z_t_as_sum}
and~\eqref{eq:truncated_LC} in Lemma~\ref{lemma:local_cluster} and applying
Wald's equation (see Lemma~\ref{lem:Wald-equation} in Appendix~\ref{app:maths})
we have:
\begin{align*}
    &\Expcc{|Q_{t}| \mid \condpar} = \Expcc{|\LC^L(1)|}\Expcc{|Y_t| \mid \condpar}= \\ 
    &= \left(\frac{1+p}{1-p} - \frac{2p^{L+1}}{1-p}\right) \cdot \Expcc{|Y_{t} | \condpar} \,.
\end{align*}
Omitting the conditioning on the event $\{\condpar\}$ for the sake of brevity
in the remainder of this proof, Lemma~\ref{lemma:Y_t__not_too_small} implies $
|Y_t| \geq p\cdot c \cdot i \cdot (1-9L/k)$ with probability at least $1 -
n^{-3}$. Moreover, by definition of $Q_t$, 
\begin{equation*}
  \Prc{  |Q_t| \geq \sum_{x=1}^{p\cdot c \cdot i (1-9L/k)} 
  |\LC^{L}(x)|\mid |Y_t |\geq p\cdot c \cdot i \cdot (1-9L/k)}=1.
\end{equation*}
If we set $Z=\sum_{x=1}^{p\cdot c \cdot i(1-9L/k)}|\LC^L(x)|$ the above
inequality implies
\begin{equation}
	\Prc{|Q_t| \leq z \mid \{|Y_t| \geq p\cdot c \cdot i \cdot 
	(1-9L/k)\}} \leq \Prc{Z \leq z} \, . \label{eq:stoc_domination_Q_X} 
\end{equation}
Again from Wald's equation,  
\begin{equation*}
	\Expcc{Z}=\frac{p(1+p)}{1-p}\cdot c \cdot i\cdot 
	\left(1-\frac{9L}{k}\right)\left(1-\frac{2p^{L+1}}{(1-p)(1+p)}\right)=\mu.
\end{equation*}
Hence, for sufficiently large $n$, from the law of total probability, from
Lemma~\ref{lemma:Y_t__not_too_small} and from~\eqref{eq:stoc_domination_Q_X},
we have:
\begin{align}
	& \Prc{|Q_{t}| \leq \mu \left(1-\frac{L}{k}\right)}  \leq 
	\Prc{|Q_{t}| \leq \left(1 - \frac{L}{k}\right)\mu \mid 
	\{|Y_t| \geq p\cdot c\cdot i \cdot 
	(1-9L/k)\}} + \frac{1}{n^3}  \notag \\
	& \leq \Prc{Z \leq \left(1 - \frac{L}{k}\right)\mu} + \frac{1}{n^3}  
	\leq_{(*)}  2\exp{\left(\frac{-2\mu^2}{(k^2/L^2) \cdot p\cdot i\cdot q \cdot (1-9L/k) 
	\cdot 4L^{2}}\right)} + \frac{1}{n^3}\notag \\ & \leq _{(**)} 
	2\exp{\left(\frac{- c \cdot \beta \log n }{4k^2}\right)} + 
	\frac{1}{n^3} \leq 3n^{-3} \leq n^{-2} \, , \notag
\end{align}
where in (*) we used the Hoeffding inequality (see Theorem~\ref{thm:hoeff} in
the Appendix), by leveraging the fact that the random variables counting the
number  of infectious nodes in each local cluster are mutually independent and,
moreover, they range between $1$ and $2L+1$. Moreover, (**) holds if we take
$\beta$ and $k$ sufficiently large. Recalling that for simplicity we omitted
the conditioning on $\condpar$, the above derivations imply
\begin{align}
	& \Prc{|Q_t| \geq \frac{pc(1+p)}{(1-p)} \cdot i\cdot 
	\left(1 - \frac{10L}{k}\right)\left(1-\frac{2p^{L+1}}{(1-p)(1+p)}\right) \mid 
	\condpar}\notag \\ & \geq 1 - \frac{1}{n^2} \, . \label{eq:Q_t_grows_final}
\end{align}
The proof of Lemma~\ref{thm:exponential_growth} then follows by observing
that, since $p=\frac{\sqrt{c^{2} + 6c +1}-c-1}{2c} + \varepsilon$, we can fix
suitable values for constants $k$ and $L$, so that 
\begin{equation}
    c \cdot \frac{p(1+p)}{1-p}\left(1 -
    \frac{2p^{L+1}}{(1+p)(1-p)}\right)\left(1 -
    \frac{10L}{k}\right)=\left(1+\delta\right) \, , 
    \label{eq:1+delta}
\end{equation} 
for some constant $\delta>0$. Together, \eqref{eq:Q_t_grows_final} and
\eqref{eq:1+delta} imply \eqref{eq:lem:exp_growth} in
Lemma~\ref{thm:exponential_growth}.

\end{proof}

\subsection{Wrapping up: proof of Claim 1 of Theorem \ref{thm:gen_overthershold}}
\label{ssec:proof:thm:gen_overthershold}

We first prove Theorem \ref{thm:gen_overthershold}, since it is implicated almost directly by the lemmas proved in the previous subsections, namely Lemmas \ref{le:L-visit} and \ref{thm:exponential_growth}. To prove the theorem, we introduce the following algorithm, which is nothing more than a simple combination of  Algorithms \ref{alg:L-visit} and \ref{alg:par_L-visit}, with some simplifications.

\begin{algorithm}
\caption{\textsc{$L$-visit from a set of initiators}}
\small{
\textbf{Input}: A small-world graph $G_{\text{SW}} = (V, E_{\text{SW}})$; a
subgraph $H$ of $G_{\text{SW}}$; a set of initiators $I_0
\subseteq V$.
\begin{algorithmic}[1]
	\State $Q = I_0$
	\State $R = \emptyset$
	\State $D = \emptyset$
	\While{$Q \neq \emptyset$ and $Q < \beta \log n$}\label{line:totvisit:firstwhile}
	\State Perform lines \ref{line:alg:L-visit:while}-\ref{line:alg:L-visit:endwhile} of Algorithm \ref{alg:L-visit} \Comment{Sequential $L$-visit}
%	    \State $w = \dequeue(Q)$
%	    \State $R = R \cup w$ 
 %	            \For {each bridge-neighbor $x$ of $w$ in $H$}
 %	              \If {$x$ is \good for $R \cup Q$}		\Comment{We are 
% 	              using Definition \ref{def:free_node}}
% 	                  \For{each node $y$ in the $L$-truncated local cluster $\LC ^L (x)$} 
% 	                       \State $\enqueue(y,Q)$
% 	                  \EndFor
% 	              \EndIf
% 	            \EndFor
 	 \EndWhile
     \While{$Q \neq \emptyset$}\label{line:totvisit:secondwhile}
	 \State Perform lines \ref{line:parvisit:while}-\ref{line:parvisit:endwhile} of Algorithm \ref{alg:par_L-visit}\Comment{Parallel $L$-visit}
     \EndWhile
\end{algorithmic}}
\label{alg:union_L-visit}
\end{algorithm}

It should be noted that the  algorithm above is essentially a sequence of two
main while loops.  The first loop in line \ref{line:totvisit:firstwhile} corresponds to
Algorithm~\ref{alg:L-visit} and describes a ``bootstrap'' phase of the RF
process, while the second main loop in line \ref{line:totvisit:secondwhile}, describes a second phase
starting with a subset $Q$ of visited nodes of size $\Omega(\log n)$.

The following lemma states the main properties produced by our  analysis of Algorithm \ref{alg:union_L-visit} with  input $(G,G_p,I_0)$. In particular, it claims that, with probability $\Omega(1)$, the first while loop terminates after $\bigO(\log n)$ rounds. Moreover, once the second while loop starts, it is such that, after $\bigO(\log n)$ rounds, there will be at least $\Omega(n)$ visited nodes w.h.p. This result implies that, starting from a single source $s \in V$, the algorithm will visit $\Omega(n)$ nodes with constant probability. On the other hand, starting from a set of sources $I_0$ such that $|I_0|=\Omega(\log n)$, the algorithm will reach $\Omega(n)$ nodes, w.h.p. Hence, Claim 1 of Theorem \ref{thm:gen_overthershold} follows from the following lemma.

\begin{lemma}
\label{thm:bootstrap}
Let $V$ be a set of $n$ nodes and $I_0 \in V$ a set of \textit{initiators}.
For every $\varepsilon>0$, $c>0$ and for every 
probability $p$ such that 
\[
\frac{\sqrt{c^{2} + 6c +1}-c-1}{2c} + \varepsilon \, \leq p \,  \leq \, 1  ,
\] 
there are positive parameters $L,k,\gamma,\beta$ that depend only on $c$ and
$\varepsilon$ such that the following holds. Sample a graph $G = (V,E)$
according to the $\mathcal{SWG}(n,c/n)$ distribution, and let $G_p$ be the
percolation graph of $G$ with parameter $p$. Run    
Algorithm~\ref{alg:union_L-visit} on input $(G,G_p,I_0)$  for sufficiently large  $n$, then:
\begin{enumerate}
    \item The first while loop in line \ref{line:totvisit:firstwhile} terminates at
some round $\tau_1 =\Theta(\log n)$ in which \[\Prc{|R \cup Q| \geq n/k \mbox{
\emph{OR} } |Q| \geq \beta \log n} \geq \gamma;\]
\item Conditioning to the above event, the second while loop in line \ref{line:totvisit:secondwhile} terminates at some round $\tau_2 = \Theta(\log n)$ in which $|Q \cup R| \geq n/k$, , w.h.p.
\end{enumerate} 
\end{lemma}

\begin{proof}
The first claim  is a direct consequence of  Lemma~\ref{le:L-visit}   by setting  $t$ in the latter so that
$\varepsilon' t = \beta\log n$.\footnote{Note that
$t_0$ in the claim of Lemma~\ref{lm:bp.lowerbound} is a constant.}
Then, at the end of  the first while loop, two
cases may arise: i) $|Q \cup R | \geq n/k$ (where $Q$ and $R$ are the snapshot of the two sets at the end of the first while loop), so  $\Theta(n)$
nodes are visited within $\bigO(\log n)$ iterations and Claim $2$ immediately follows; ii) A round $\tau_1 = \Theta(\log n)$ exists in  which  the subset of infectious nodes gets size at least $|Q|
\geq \beta \log n$, where $Q$ is the queue's snapshot at the end of the first
while loop.

In order to complete the proof of Claim 2 of the lemma, it thus suffices to only address case ii) above, which corresponds to the setting in which, at the beginning of the second while loop, the queue $Q$ is initialized with a set of size $\geq \beta \log n$. 
%We remark that the second while loop  corresponds to a ``parallel''  counterpart of Algorithm~\ref{alg:L-visit},\footnote{It is essentially a restricted BFS-like visit of the percolation graph. It is restricted because nodes that are infectious over a  bridge edge can only propagate the infection within distance $L$ using active ring edges.} in which all nodes are removed from the queue at the beginning of each iteration of the second main while loop at line 10.
We can now observe that Claim $2$ is a
direct consequence of Lemma~\ref{thm:exponential_growth}.
\end{proof}

\subsection{Wrapping up: proof of Claim 1 of Theorem \ref{thm:intro-erdos-2}}
\label{ssec:proof:thm:intro-erdos-2}

In the previous subsection we heve essentially proved that, starting from a single source $s \in V$, if we explore the connected component of $s$ in $G_p$, then, with probability $\Omega(1)$, we will visit at least $\Omega(n)$ nodes and, moreover,  such  nodes induce a connected sub-graph of diameter $\bigO(\log n)$.

The goal of this subsection is  to prove  that, w.h.p.,   a set of $\Omega(n)$ nodes in $V$ exists that induces in $G_p$ a connected component with diameter $\bigO(\log n)$. To this aim, we introduce the following algorithm that can be informally seen as  several ``attempts'' of bootstraps, according to  Algorithm \ref{alg:L-visit}, each one performed   from a different source node $s$,   then followed by   a \textsc{parallel $L$-visit} according to  Algorithm \ref{alg:par_L-visit}.

\begin{algorithm}
\caption{\textsc{Search of the giant component}}
\small{
\textbf{Input}: A small-world graph $G_{\text{SW}} = (V, E_{\text{SW}})$; a
subgraph $H$ of $G_{\text{SW}}$; two integers $\beta$, $\beta'$.
\begin{algorithmic}[1]
\State $Q = \emptyset$
\State $D = \emptyset$
\While{$|Q| \leq \beta \log n$ or $|Q \cup R | \leq n/k$}\label{line:bootstrap:while} \Comment{\textsc{First phase}}
	\State Let $s \in V \setminus D$
	\State $Q = \{s\}$
	\State $R = \emptyset$
	\For{$i=1,\dots,\beta'\log n$}
	    \State Perform lines \ref{line:alg:L-visit:while}-\ref{line:alg:L-visit:endwhile} of Algorithm \ref{alg:L-visit}\Comment{Sequential $L$-visit}\label{line:bootstrap:endfor}
 	 \EndFor
 	 \State $D = D \cup R$
\EndWhile
\If{$|Q \cup R| \leq n/k$}\Comment{\textsc{Second phase}}
\State Perform Algorithm \ref{alg:par_L-visit} with input $G_{SW}$, $H$, $I_0=Q$, $D_0=R\cup D$ \label{line:bootstrap:parvisit} \Comment{Parallel $L$-visit}
\EndIf
\end{algorithmic}}
\label{alg:bootstrap}
\end{algorithm}

More in detail, the  algorithm above works in two   phases. The first phase starts in the while loop in line \ref{line:bootstrap:while} and  performs   different ``bootstraps''. In this phase, we essentially look for a source node $s \in V$ such that the queue $Q$ of the visit of the component of $s$ in $G_p$ gets $\Omega(\log n)$ nodes, after $\Theta(n)$ steps of the visit. In particular, at each iteration of the while loop, a node $s \in V$ is chosen arbitrarily and, starting from it, $\Theta(\log n)$ steps of the sequential $L$-visit of Algorithm \ref{alg:L-visit} are performed: if, at this point, the set $Q$ has $\Omega(\log n)$ nodes, the first phase is successfully completed, otherwise it starts again from another source node. The second phase of the algorithm starts in line \ref{line:bootstrap:parvisit} with a subset $Q$ of $\Omega(\log n)$ visited nodes, and consists in the parallel visit in Algorithm \ref{alg:par_L-visit} setting $Q$ as the   source subset.

The following lemma essentially states that, w.h.p.: i) within 
$\tau_1 = \bigO(\log n)$  bootstrap attempts, 
the first phase ends successfully,  and, then, ii) the second phase will discover $\Omega(n)$ nodes within further  $\tau_2=\bigO(\log n)$ iterations.

\iffalse 
(1) contains the analysis of Algorithm \ref{alg:bootstrap} with in input $(G,G_p,\beta,\beta')$, where $\beta$ and $\beta'$ are two positive constants. It claims that the while loop in line \ref{line:bootstrap:while} lasts for $\tau_1 = \bigO(\log n)$ rounds w.h.p., and the second phase for other $\tau_2=\bigO(\log n)$ rounds w.h.p. In detail, we have that, in the first phase, the algorithm, in $\bigO(\log n)$ iterations of the while loop, find a node $s \in V$ such that in the sequential $L$-visit starting from $s$, after $\Theta(\log n)$ iterations, $Q$ has size $\Omega(\log n)$ w.h.p. In the second phase, the parallel $L$-visit starting from such a $Q$ reach  w.h.p.

\fi

\begin{lemma}
Let $V$ be a set of $n$ nodes. For every $\varepsilon>0$, $c>0$ and $\beta>0$, and for every probability $p$ such that
\begin{equation*}
    \frac{\sqrt{c^2+6c+1}-c-1}{2c}+\varepsilon \leq p \leq 1,
\end{equation*}
there are positive parameters $L$, $k$ (that depends only on $c$ and $\varepsilon$) and a constant $\beta'$ (that depends on $c$, $\varepsilon$ and $\beta$) such that the following holds. Sample a graph $G=(V,E)$ according to the $\mathcal{SWG}(n,c/n)$ distribution and let $G_p$ be the percolation graph of $G$ with parameter $p$. Run  Algorithm \ref{alg:bootstrap} on input $(G,G_p,\beta, \beta')$ with any  sufficiently large $n$, then:   
\begin{enumerate}
    \item The first while loop in line \ref{line:bootstrap:while} terminates at some round $\tau_1= \bigO(\log n)$, w.h.p.;
    \item Conditioning to the above event, the second phase of the algorithm starting in line \ref{line:bootstrap:parvisit} terminates at some round $\tau_2=\Theta(\log n)$ in which $|Q \cup R| \geq n/k$, w.h.p.
\end{enumerate}
%where the probability is over both the randomness of the choice of $G$ from $\mathcal{SWG}(n,c/n)$ and over the choice of the percolation graph $G_p$. 
\label{lem:bootstrap}
\end{lemma}

\begin{proof}
As for the first claim, let $\gamma>0$ and $\varepsilon'$ be the constants  in Lemma \ref{le:L-visit}, and   define $\beta'=\beta/\varepsilon'$. Moreover, let $\tau= \log_{1-\gamma}(n)$. First, we notice that, at each iteration of the while  in line \ref{line:bootstrap:while}, the set $D$ grows w.h.p. of at most  $8L \beta'\log^2 n$ size: indeed,   each node in $G_p$ has degree of at most $4 \log n$ with probability at least $1-1/n^2$ (this fact follows from a standard application of Chernoff's bound)
and so, within $\beta' \log n$ iterations of the for loop, at most $8L \beta'\log^2 n$ nodes will be reached by $s$. This implies that, at each iteration $j \leq \tau$ of the while loop, $D$ has size  $|D| =\bigO(\log^3 n)$.

Then, we claim that, at each $j$-th iteration of the while loop in line \ref{line:bootstrap:while}, if $j \leq \tau$, there is  probability at least $ \gamma>0$ that the process terminates. Indeed, thanks to  Lemma \ref{le:L-visit}, there exists a constant $\gamma>0$ such that, at the end of the for loop in line \ref{line:bootstrap:endfor}, \[\Prc{|Q|\geq \varepsilon' \beta' \log n \text{ or } |R \cup Q| \geq n/k}\geq \gamma.\] Therefore, the probability that the process continues after $\tau$ iterations is at most 
\[(1-\gamma)^\tau \leq \frac{1}{n}.\]

 Claim $2$  is instead a
direct consequence of Lemma~\ref{thm:exponential_growth}.
%So, we prove the existence w.h.p. of a node $v$ such that the sequential-BFS starting from $v$, after $\bigO(\log n)$ steps, is in one of the two cases i) $|Q| \geq \beta \log n$ ii) $|Q \cup R | \geq n/k$.
%If ii) happens, we prove the lemma. Indeed, w.h.p. we have the existence of a node $v$ such that there is a set of $\Omega(n)$ nodes at distance at most $\bigO(\log n)$ from $v$.
%If i) happens, it suffices to perform a sequential-BFS (Algorithm \ref{alg:sequential-BFS-visit}) with in input $G_p$, $I_0 = Q$ and $R_0$ and apply Lemma \ref{lem:old_erdos_phase2} to claim that such BFS reaches at least $\Omega(n)$ nodes in $\bigO(\log n)$ steps. So, also in this case we prove that there is a node $v$ such that there are at least $\Omega(n)$ nodes at distance $\bigO(\log n)$ from it, w.h.p. 
\end{proof}

\section{The \texorpdfstring{$\SWGm(n,q)$} \,  Model below the Threshold}\label{sec:belo}
The goal of this section is to prove the second claims  of
Theorems \ref{thm:intro-erdos-2} and \ref{thm:gen_overthershold}. Informally, in the following we show that,
whenever $p< \frac{\sqrt{c^{2} + 6c +1}-c-1}{2c}$, w.h.p., the percolation graph $G_p$ of
a graph $G=(V,E)$ sampled from \SWG is such that every connected component has $\bigO(\log n)$ nodes.  We state here the claim which is proved in
Subsection~\ref{ssec:proofbelowth}.

\begin{lemma} \label{thm:below_th}
Let $V$ be a set of $n$ nodes. For every $\varepsilon>0$, $c>0$ and for every
contagion probability $p$ such that
\[
0 \, \leq \, p \, \leq \, \frac{\sqrt{c^{2} + 6c +1}-c-1}{2c} - \varepsilon
\]
there is a positive constant $\beta$ that depends only on $c$ and $\varepsilon$
such that the following holds. Sample a graph $G=(V,E)$ according to the
$\mathcal{SWG}(n,c/n)$ distribution, and let $G_p$ be the percolation graph of
$G$ with parameter $p$. If $n$ is sufficiently large, with probability at least
$1 - 1/n$ with respect to the randomness of $G$ and the randomness of $G_p$,
$G_p$ contains no connected component of size exceeding $\beta \log n$.
\end{lemma}

Interestingly enough, thanks to the equivalence between the SIR process and the percolation process, the above lemma also implies Claim 2 of Theorem \ref{thm:gen_overthershold}, 
since the \SIR process infects at least one new  node in each round unless it
has died out, the above result also implies that, w.h.p., the \SIR process dies
out within $\beta\log n$ rounds, infecting at most $|I_0| \cdot \beta\log n$
new nodes.

\subsection{Proof of Lemma~\ref{thm:below_th}} \label{ssec:proofbelowth}
In order to prove upper bounds on the number of nodes in a connected component, we might
proceed as with did to prove lower bounds. We proceed as follows: i) we run a BFS in $G_p$, i.e. we run
Algorithm~\ref{alg:upper_bound} with input $G_{SW} = G$, $H =G_p$ and an
arbitrary source $s \in V$,  and we define a (sequential) Galton-Watson
branching process that stochastically dominates\footnote{More precisely in
Lemma~\ref{le:ub_dom} we use Definition~\ref{def:stochastic_domination} in the
appendix and notice that a simple coupling argument (see
Definition~\ref{def:coupli} and
Lemma~\ref{lem:stochastic_domination_coupling}.) applies between the two
processes.} the BFS process with respect to the overall size of the set of
nodes visited upon termination; ii) thanks to step (i), we can prove that for
each source $s \in V$ and every $t > \beta \log n$, the BFS in
Algorithm~\ref{alg:upper_bound} with input $(G,G_p,s)$  terminates within $t$
iterations of the while loop after visiting less than $t$ nodes, w.h.p.

\begin{algorithm}[H]
\caption{\textsc{BFS visit}}
\small{
\textbf{Input}: A small-world graph $G_{\text{SW}} = (V, E_{\text{SW}})$ and a
subgraph $H$ of $G_{\text{SW}}$; a source $s \in V$.

    \begin{algorithmic}[1]
		\State $Q = \{s\}$
		\While{$Q \neq \emptyset $}
			\State $w = \dequeue(Q)$
			\State $\texttt{visited}(w) = \True$ 
					\For {each bridge neighbor $x$ of $w$ in $H$  such that $\texttt{visited}(x) = \False$}
					  \For {each $y$ in the local cluster  $\LC(x)$ such that $\texttt{visited}(y)=\False$}
					  \State $\enqueue(y,Q)$
					  \EndFor
					\EndFor
		 \EndWhile
	\end{algorithmic}}
	\label{alg:upper_bound}
\end{algorithm}
 
It should be noted that i) Algorithm~\ref{alg:upper_bound} visits the connected
component containing $s$; ii) the number of
nodes visited by the algorithm is exactly equal to the number of iterations of
the main while loop before $Q$ becomes empty.  To formalize our approach we
need to define the following random subsets of nodes.

\begin{definition}
For each $t \geq 1$, let $Q_t$ be the set $Q$ of the BFS in
Algorithm~\ref{alg:upper_bound} at the end of the $t$-th iteration of the
while loop.  Let $R_t=\cup_{i=1}^{t-1}Q_i$, and let $S_t=V \setminus (R_t \cup
Q_t)$. We also define $cc(s)=\max\{t: Q_t\ne\emptyset\}$ as the overall number
nodes visited at the end of   the  execution of Algorithm~\ref{alg:upper_bound}
with input $G_{SW} = G$, $H=G_p$, and  an arbitrary source $s \in V$.
\label{def:visitupperbound}
\end{definition}

We now consider the ``sequential'' Galton-Watson Branching process
$\{B_t\}_{t\geq 0}$ (see Definition~\ref{def:branchingprocess}) determined by
the  random variables $\{W_t\}_{t\geq 0}$,  where  the $W_t$'s are {\em
independent} copies of the following random variable $W$:
 
\begin{definition}\label{def:zeta}
$W$ is generated as follows: i) we first randomly sample an integer $Y$ from
the distribution $\Bin(n,pc/n)$; ii) $W = Y + \sum_{j=1}^{2Y}L_j$, where each
$L_j$ is a variable that counts the number of successes in a sequence of
independent Bernoulli trials with success probability $p$, till the first
failure. 
\end{definition}

It should be noted that $\sum_{j=1}^{2x}L_j$ is the overall number of successes
in a sequence of $2x$ Bernoulli trials with parameter $p$, until we observe
exactly $2x$ failures. As such, $\sum_{j=1}^{2x}L_j$ follows a {\em negative
binomial distribution}.  The next lemma shows that the above Galton-Watson
process $\{B_t\}_{t\geq 0}$ dominates the process $\{Q_t\}_t$.

\begin{fact}\label{le:ub_dom} 
For every $x\ge 0$, the following holds, for every $t$:
\[
	\Prc{B_t\ge x}\ge\Prc{|Q_t|\geq x},
\]
where the left side of the inequality is taken over the randomness of the
Galton-Watson process $\{B_t\}_{t\geq 0}$, while the right side is taken on the
outcome of Algorithm~\ref{alg:upper_bound} with input $G_{SW}=G$ and $H=G_p$,
with respect to the  randomness of the initial graph $G$ sampled from \SWG and
that of its percolation $G_p$.
\end{fact}

\begin{proof}
The argument is based on the following observations. Consider the process
generated by Algorithm~\ref{alg:upper_bound} in
Definition~\ref{def:visitupperbound} with input $G_{SW}=G$, $H=G_p$ and any
fixed source $s \in V$  and consider a given state at iteration $t$. Consider
the node $w$ dequeued from $Q$ at the beginning of iteration $t$ of
Algorithm~\ref{alg:upper_bound}, and consider the set of nodes it can possibly
add to the queue at iteration $t$.  In the best possible case, i) $w$ has some
number $\hat{Y}$ of bridge edges in $E_p$ (the bridge edges considered in line
6 of the algorithm), and ii) a local cluster consisting of some $L$ nodes will
be infected starting at each of them. $\hat{Y}$ is distributed as
$\text{Bin}(|S_t|,pc/n)$ and is thus dominated by any variable distributed as
$\text{Bin}(n,pc/n)$. As for $L$, it is certainly dominated by the sum of
$\hat{Y}$ variables, each counting the total number of activations in the local
cluster created by a single node of an infinite path topology (see
Fact~\ref{fa:exp_loc_ub}).  But the above considerations imply that the
variable counting the number of nodes that $w$ can possibly infect is dominated
by a variable distributed like the $W_t$'s.
\end{proof}

Thanks to the above fact, to get an upper bound on the overall number of nodes
the \SIR process can infect under the hypotheses of Lemma~\ref{thm:below_th},
we can analyse the Galton-Watson process specified in
Definition~\ref{def:zeta}.
 
\begin{lemma}\label{le:ch_geom}
Consider the   Galton-Watson process $\{B_t\}_{t\geq 0}$ with $W_1, W_2,\ldots$ 
as in Definition \ref{def:zeta}. For any $t>0$, 
\begin{align}
\label{eq:le:ch_geom}
	&\Prc{\sum_{i=1}^t W_i \ge (1 + 2\delta) \frac{p(1+p)}{1-p}ct}\le 2e^{-\frac{\delta^2 p^2t}{9}}.
\end{align}
\end{lemma}
\begin{proof}
We begin by observing that, the definition of the $W_i$'s and Definition
\ref{def:zeta} imply, for every $i \geq 0$, 
\begin{equation*}
	W_i = Y_i + \sum_{s=1}^{2Y_i}L_{i,s} \, .
\end{equation*}
Here, $Y_i$ is distributed as $\text{Bin}(n,pc/n)$, while each $L_{i,s}$ is an
independent copy of a variable that counts the number of successes until the
first failure in a sequence of independent Bernoulli trials with success
probability $p$.\footnote{It can be equivalently regarded as a geometric random
variable with success probability $1-p$.} Next, if we set $W=\sum_{i=1}^t W_i$,
we have
\begin{equation}\label{eq:zeta_1}
	W = \sum_{i=1}^t W_i = \sum_{i=1}^t(Y_i + \sum_{s=1}^{2Y_i}L_{i,s}) \, .
\end{equation}
An obvious remark is that all the $L_{i,s}$ are just independent copies of a
random variable with identical distribution. As a consequence, if we set $Y =
\sum_{i=1}^tY_i$, we can rewrite \eqref{eq:zeta_1} as 
\begin{equation}\label{eq:zeta_2}
	W = \sum_{i=1}^t W_i = Y + \sum_{j=1}^{2Y}L_j,
\end{equation}
where the $L_j$'s are independent random variables distributed like the
$L_{i,s}$'s. Hence:
\begin{equation}
	\Prc{\sum_{i=1}^t W_i\geq (1+2\delta)\frac{p(1+p)}{(1-p)}ct } 
    = \Prc{Y + \sum_{j=1}^{2Y}L_j\ge (1+2\delta)\frac{p(1+p)}{(1-p)}ct}.
	\label{eq:equivalence_W_YL}
	\end{equation}
In order to bound the right hand side of the equation above we proceed in two
steps. First, we prove a concentration result on $Y$. This is easy, since it is
$Y = \sum_{i=1}^t Y_i$, with each $Y_i$ being an (independent) binomial
variable with distribution $\text{Bin}(n,pc/n)$. Each $Y_i$ is in turn the sum
of $n$ independent Bernoulli variables, each with parameter $pc/n$. Overall,
$Y$ is just the sum of $nt$ independent Bernoulli variables with parameter
$pc/n$. Hence, $\Expcc{Y} = pct$. Moreover, a straightforward application of
Chernoff bound yields, for every $0 < \delta < 1$,
\begin{equation}\label{eq:ch_y}
\Prc{Y \geq (1+ \delta) pct}\le e^{-\frac{\delta^2}{3}pct} \, .
\end{equation}
We next argue about $\sum_{j=1}^{2Y}L_j$. We further have, for $0 < \delta <
1$, 
\begin{align}
	& \Prc{\sum_{j=1}^{2Y}L_j > (1 + 2\delta)\frac{2p^2}{1-p}ct \mid Y \leq (1+\delta)pct} \leq  \Prc{\sum_{j=1}^{2(1+\delta)pct}L_j > (1 + 2\delta)\frac{2p^2}{1-p}ct }\, ,
	\label{eq:sum_1}
\end{align}
where we dropped the conditioning on $\{Y \leq (1+\delta)pct\}$, since this
only implies that we are summing at most $(1+\delta)pct$ independent, geometric
random variables with parameter $p$. We next note that
\begin{equation}
	\Prc{\sum_{j=1}^{2(1+\delta)pct}L_j > (1 + 2\delta)\frac{2p^2}{1-p}ct} = 
	\Prc{\sum_{i=1}^{2(1+\delta)pct+(1 + 2\delta)\frac{2p^2}{1-p}ct}\check{B}_i < 2(1+\delta)pct},
	\label{eq:L_j_B_j}
\end{equation}
where the $\check{B}_i$'s are independent, Bernoulli variables with {\em
success} probability $1-p$. The equality above is true since
$\sum_{j=1}^{2y}L_j$ follows a negative binomial distribution, whose cumulative
distribution function is related to the one of the binomial
\cite{morris1963note}.\footnote{Intuitively, $\sum_{j=1}^{2y}L_j$ is the number
of successes in a sequence of $\sum_{j=1}^{2y}L_j + 2y$ Bernoulli trials with
success probability $p$, before exactly $2y$ failures are observed. As a
consequence, $\sum_{j=1}^{2y}L_j > x$ implies that $2y + x$ trials were not
sufficient to observe $2y$ failures.}  The expectation of the sum of the
$\check{B}_i$'s is    
\begin{equation}\label{eq:exp_sumB}
	\Expcc{\sum_{i=1}^{2(1+\delta)pct+(1 + 2\delta)\frac{2p^2}{1-p}ct}\check{B}_i} = 
	2(1+\delta)pct(1 - p) + 2p^2(1 + 2\delta)ct = 2(1+\delta)pct \left(1+ \frac{\delta p}{1+\delta}\right) \, . 
\end{equation}
The above derivations imply
\[
	\frac{2(1+\delta)pct}{\Expcc{\sum_{i}\check{B}_i}}\le \frac{1}{1 + \frac{\delta 
	p}{1+\delta}} < 1 - \frac{\delta p}{3},
\]
where the last inequality follows from simple manipulations and where we
eventually use $1 + \delta + \delta p < 3$. We denote $\mu=\Expcc{\sum_i
\check{B}_t}$ and from~\eqref{eq:exp_sumB} we have that $\mu\geq 2pct$.
From~\eqref{eq:sum_1} and~\eqref{eq:L_j_B_j}, this allows us to write
\begin{align}
	&\Prc{\sum_{j=1}^{2Y}L_j > (1 + 2\delta)\frac{2p^2}{1-p}ct \mid Y \leq(1+\delta)pct} \notag \\ & \le  \Prc{\sum_{i=1}^{2(1+\delta)pct+(1 +
	2\delta)\frac{2p^2}{1-p}ct}\check{B}_i < \left(1 - \frac{\delta 
	p}{3}\right)\mu}\le e^{-\frac{\delta^2 
	p^3ct}{9}} \,,
	\label{eq:L_j_cond}
\end{align}
where the first inequality follows from the inequality relating
$\Expcc{\sum_{i}\check{B}_i}$ to $2(1+\delta)pct$ written above, and the last
inequality is a simple Chernoff bound on the lower tail, considering that $\mu
\geq 2pct$.  

This allows us to conclude the proof. Indeed, we have that for every $t>0$ and
from the law of total probability
\begin{align*}
   &\Prc{ Y + \sum_{j=1}^{2Y}L_j\ge (1+2\delta)\frac{p(1+p)}{(1-p)}ct}\\ &\leq \Prc{\sum_{j=1}^{2Y}L_j\geq(1+2\delta)\frac{2p^2}{1-p}ct \mid Y \leq(1+\delta)pct}+\Prc{Y \geq (1+\delta)pct}  \\ &\leq e^{-\frac{\delta^2p^3ct}{9}}+e^{-\frac{\delta^2}{3}pct}\leq 2e^{-\frac{\delta^2p^3ct}{9}},
\end{align*}
where the last inequality follows from \eqref{eq:ch_y} and \eqref{eq:L_j_cond}.
The claim follows finally from \eqref{eq:equivalence_W_YL}.
\end{proof}
Next, assume that $p = \frac{\sqrt{c^{2} + 6c +1}-c-1}{2c} - \varepsilon$ for
constant $ \varepsilon \in (0, 1 - \frac{\sqrt{c^{2} + 6c +1}-c-1}{2c})$. This
yields that, for some $\varepsilon'>0$, we have $pc(1+p)/(1-p)=1-\varepsilon'$,
so
\begin{align*}
	&(1 + 2\delta)\frac{pc(1+p)}{1-p} = (1+2\delta)(1-\varepsilon')\leq \left(1-\frac{\varepsilon'}{2}\right)\, ,
\end{align*}
where the last inequality holds whenever
$\delta\le\frac{\varepsilon'}{4(1-\varepsilon')}$. For this choice of $\delta$,
we consider Lemma~\ref{le:ch_geom} setting $t=\beta \log n$ for a sufficiently
large $\beta$, in order to have the RHS in~\eqref{le:ch_geom} smaller than
$1/n^2$. We notice that the choice $\beta$ depends only on $\varepsilon$.  With
the above choices, Lemma~\ref{le:ch_geom} implies
\begin{equation}
\label{eq:Z_1+<logn}
    W_1+\dots+ W_{t}<\beta\log n \, , 
\end{equation}
with probability at least $1 - 1/n^2$. 

To complete our proof, consider again the Galton-Watson Branching process
$\{B_t\}_{t\geq 0}$.  The size of the overall population  up to iteration $t$
is $\sum_{i=1}^t W_i$ and we just proved that, for $t = \beta \log n$, 
\[ 
\sum_{i=1}^{t}W_i < \beta\log n \, , 
\] 
with probability at  least $1 - 1/n^2$, which implies that, with the same
probability,   if $B_{t}>0$ then  it would hold 
\[
B_{t} = \sum_{i=1}^{t} W_i - \beta\log n < 0 \, ,
\]
so a contradiction. Therefore,  w.h.p. there is a $\tau < \beta\log n$, such
that $B_{\tau} = 0$,  which in turn implies that, with probability at least $1
- 1/n^2$, the Galton-Watson Branching process $\{B_t\}_{t\geq 0}$ goes
extinct within $\beta \log n$ iterations, with a population size less than
$\beta \log n$.

We are now in a position to show that, with the same probability, $|Q_t| = 0$
for $t\le\beta \log n$, implying that Algorithm~\ref{alg:upper_bound} completes
within $\beta \log n$ rounds and hence, it visits at most $\beta \log n$ nodes.
In fact:
\[
    \Prc{|cc(s)|\ge\beta \log n} 
    = \Prc{|Q_{\beta \log n}| > 0}
    \le\Prc{\sum_{i=1}^{\beta\log n} W_i - \beta\log n > 0}
    \le\frac{1}{n^2} \, ,
\]
where the first equality follows from the definition of $R_\infty$, the second
inequality follows from Fact~\ref{le:ub_dom}, while the last inequality follows
from Lemma~\ref{le:ch_geom}.

We thus proved that, for a fixed source $s \in V$,
Algorithm~\ref{alg:upper_bound} with input $G$ and $G_p$ visits at most $\beta
\log n$ nodes in the graph $G_p$ with probability at least $1-1/n^2$. So,
Lemma~\ref{thm:below_th} follows from a union bound over all possible choices
of source $s \in V$.

\section{The \texorpdfstring{\SWGreg}\, Model above the Threshold}
\label{sec:regularcase}

In this section we prove the first claims  of Theorems~\ref{thm:main_matching} and \ref{thm:main_matching_rf}. We show
that, with probability $\Omega(1)$, the connected component of an initiator node
in the percolation graph $G_p$ of a \SWGreg\ graph contains $\Omega(n)$ nodes,
as soon as $p = 1/2 + \varepsilon$, where $\varepsilon
> 0$ is an arbitrarily-small constant. Moreover, over the same conditions on $p$, we show that $G_p$ has a giant component of $\Omega(n)$ nodes w.h.p.

\subsection{Sequential \texorpdfstring{$L$}\,-visit (proof of Claim 1 of Theorem \ref{thm:main_matching_rf})}

As in the proofs of Section \ref{sec:wup}, we analyze the number of nodes
reached by a BFS-like visit of the percolation graph $G_p$ starting at a set of nodes $I_0$. We consider a slightly modified version of Algorithm~\ref{alg:L-visit} in
which, once a bridge neighbor $x$ of a dequeued node $w$ is ``observed'', it
will no longer be considered in any of the subsequent iterations of the while
loop. This allows us to use the principle of deferred decisions on the
randomness used to determine the bridge neighbor $x$ of a dequeued node $w$ and
on the randomness used to determine whether a bridge edge exists in the
percolation graph.

 \begin{algorithm}
 \caption{\textsc{Sequential} \SIRimmtrunc}
 \small{
\textbf{Input}: A small-world graph $G_{\text{SW}} = (V, E_{\text{SW}})$; a
 subgraph $H = (V, E_H)$; a set of initiators $I_0 \subseteq V$ and a set of deleted nodes $D_0 \subseteq V \setminus I_0$.
 \begin{algorithmic}[1]
 \State $Q = I_0$
 \State $R = \emptyset$
 \State $D = D_0 \cup N(D_0)$
 \While{$Q \neq \emptyset $}\label{line:randommatching:while}
     \State $w = \dequeue(Q)$
 	\State $R = R \cup \{w\}$ 
  	\State $x = $ bridge neighbor of $w$ in $G_{SW}$
     \Case{$x$ is \good for $(D \cup R \cup Q)$ in $G_{\text{SW}}$ and $\{w, x\}
     \in E_H$} \ \ \ \ \Comment{We are using Definition~\ref{def:free_node} }
        \State $Q = Q \cup (\LC^L(x) \setminus \{x\})$
        \State $R = R \cup \{x\}$
     \EndCase
     \Case{$x$ is \good for $(D \cup R \cup Q)$ in $G_{\text{SW}}$ and $\{w, x\} \notin E_H$}
       \State $D = D \cup \{x\}$
     \EndCase
     \Case{$x$ is not \good for $(D \cup R \cup Q)$ in $G_{\text{SW}}$ and $x \in Q$}
         \State $Q = Q \setminus \{x\}$
         \State $R = R \cup \{x\}$
     \EndCase
     \Case{$x$ is not \good for $(D \cup R \cup Q)$ in $G_{\text{SW}}$ and $x \notin Q$}
         \State $D = D \cup \{x\}$
     \EndCase
 \EndWhile\label{line:randommatching:endwhile}
 \end{algorithmic}}
 \label{alg:L-visit-var-matching}
 \end{algorithm}

The following lemma contains the analysis of the Algorithm \ref{alg:L-visit-var-matching}. In detail, it shows that, if we start the visit from a single source node $s$, after $\Theta(\log n)$ steps of the visit we will have $\Omega(\log n)$ nodes in the queue, with constant probability. Moreover, the lemma claims also that, if we start the visit from $\Omega(\log n)$ nodes, the visited nodes will be $\Omega(n)$ w.h.p.

We notice also that the first claim of Theorem \ref{thm:main_matching_rf} is directly implied by the following lemma. 
\begin{lemma}
 Let $V$ be a set of $n$ nodes, $I_0 \subseteq V$ a set of initiators and $D_0 \subseteq V \setminus I_0$ a set of deleted nodes such that $|D_0| \leq \log^4 n$. For every $\varepsilon>0$ and for every probability $p$ such that $1/2+\varepsilon\leq p \leq 1$, there are positive parameters $L$, $k$, $\beta$ and $\gamma$ that depend on $\varepsilon$ such that the following holds.
Sample a graph $G=(V,E)$ according to the \SWGreg distribution and let $G_p$ be the percolation graph of $G$ with percolation probability $p$. Run the \textsc{Sequential $L$-visit} procedure in Algorithm \ref{alg:L-visit-var-matching} on input $(G,G_p,I_0,D_0)$: if $n$ is sufficiently large,
\begin{enumerate}
    \item if $|I_0|=1$, after $\tau_1=\bigO(\log n)$ iterations of the while loop we have that
    \[\Prc{|Q |\geq \beta \log n \text{ or } |Q \cup R \cup D|\geq n/k} \geq \gamma;\]
    \item if $|I_0|\geq \beta \log n$, after $\tau_2=\bigO(n)$ iterations of the while loop we have that $|Q \cup R| \geq n/(4k)$ w.h.p.
\end{enumerate}
\label{lem:randommatching:analysis}
 \end{lemma}
 
 \begin{proof}
 We first make some observation that will be useful for the proof of both the claims of the lemma. We begin with  noticing that
Algorithm~\ref{alg:L-visit-var-matching} preserves the following invariant: at
the beginning of each iteration of the while loop, the bridge neighbors of all
nodes in $Q$ have not been observed so far.  From the principle of deferred
decisions, it thus follows that, when a node is dequeued, its bridge neighbor
can be any of the nodes not in $R \cup D$, with uniform probability.

For $t = 1, 2, \dots,$ let $Q_t$, $R_t$, and $D_t$ be the sets of nodes in $Q$,
$R$, and $D$, respectively, at the end of the $t$-th iteration of the while loop,
and let $Z_t$ be the number of nodes added to the queue $Q$ during the $t$-th
iteration. Notice that $|Q_0| = 1$ and 
\begin{equation}\label{eq:queue_recursion_matching}
|Q_{t}| = 
\left\{
\begin{array}{cl}
0 & \mbox{ if } Q_{t-1} = \emptyset \\
|Q_{t-1}| + Z_t - 1 & \mbox{ otherwise}
\end{array}
\right.
\end{equation}
Observe that here $Z_t$ is an integer-valued random variable with $-1 \leqslant
Z_t \leqslant 2L$. As we did in the proof of Lemma~\ref{le:L-visit} we show
that, as long as the overall number of nodes in $Q \cup R \cup D$ is below a
constant fraction of $n$, the sequence $\{|Q_t|\}_t$ stochastically dominates a
diverging Galton-Watson branching process (Definition \ref{def:GW}).

Let $k \geq 1 $ be a sufficiently large constant that will be fixed later.
Consider the generic $t$-th iteration of the while loop with $Q \neq
\emptyset$, let $|Q \cup R \cup D| \leqslant n/k$ be the  number of nodes in
the queue or in the set $R \cup D$ at the beginning of the while loop, and let
$A$ be the set of nodes at distance larger than $L$ from any node in $Q \cup R
\cup D$ in the subgraph of $G_{\text{SW}}$ induced by the edges of the ring,
i.e.
\[
A=\{v \in V \mid d_{(V,E_1)}(Q \cup R \cup D,v) \geq L+1\}\,.
\]
Observe that there are at most $2L(n/k)$ nodes at distance smaller than or
equal to $L$ from a node in $Q \cup R \cup D$ in $(V,E_1)$, so $|A| \geq
n(1-2L/k)$.

Let $w$ be the node dequeued at the $t$-th iteration of the while loop and let
$x$ be its bridge neighbor. Since $|Q \cup R \cup D| \leqslant n/k$, from the
principle of deferred decision it follows that $x$ is already in the queue $Q$
with probability at most $1/k$ while $x$ is free for $Q \cup R \cup D$ in
$G_{\text{SW}}$ with probability at least $|A|/n = (1 - 2L/k)$ and the bridge
edge $\{w,x\}$ exists in the percolation graph with probability $p$.  Hence,
random variable $Z_t$ in \eqref{eq:queue_recursion_matching} takes vaues either
$-1$, with probability at most $1/k$, or the size of the local cluster centered
at $x$ excluding $x$ itself, with probability at least $p (1 - 2L/k)$.
From~\eqref{eq:truncated_LC} on the expected size of a local cluster it thus
follows that the expected number of new nodes added to the queue is 
\begin{align*}
\Expcc{Z_t \;|\; Q_{t-1} \neq \emptyset, \, |Q_t \cup R_t \cup D_t| \leqslant n/k} 
& \geqslant - \frac{1}{k} + p \left( 1 - \frac{2L}{k} \right) 
\left(\frac{1 + p - 2p^{L+1}}{1-p} - 1 \right) \\
& = - \frac{1}{k} + \frac{2p^2}{1-p} \left( 1 - \frac{2L}{k} \right) 
\left(1 - 2 p^L\right) \\
& = \frac{2p^2}{1-p} - \bigO\left(\frac{L}{k}\right) - \bigO\left({p^L}\right) \, .
\end{align*}
For every $\varepsilon > 0$, the above inequality allows us to fix suitable
constants $L$, $k$, and $\varepsilon' > 0$ such that, if $p = \frac{1}{2} +
\varepsilon$ then 
\[
\Expcc{Z_t \;|\; Q_{t-1} \neq \emptyset, \, |Q_t \cup R_t \cup D_t| \leqslant n/k} 
\geqslant 1 + \varepsilon'\,.
\]

As long as there are less than $n/k$ nodes in $(Q \cup R \cup D)$ and $Q$ is
not empty, the number of nodes in $Q$ thus satisfies $|Q_t| = |Q_{t-1}| - 1 +
Z_t = |Q_{t-1}| - 2 + (Z_t + 1)$, where $Z_t + 1$ is a non-negative
integer-valued random variable with expectation larger than $2$.
We now proceed as in the proof of Lemma \ref{le:L-visit}, and omit some details.

We define a Galton-Watson branching process $\{ B_t \}_t$ with the aim of
bounding $Q_t$ in terms of $B_{2t}$. The branching process $\{ B_t \}_t$ is
defined in terms of a random variable $W$ defined to be the worst-case
distribution of $(Z_t + 1)/2$ when $|Q_t \cup R_t \cup D_t | \leq n/k$. The
process is such that $\Expcc{ W} \geq 1 + \varepsilon'$ for an absolute
constant $\varepsilon' > 1$, and we can construct a coupling (see
Definition~\ref{def:coupli} and Lemma~\ref{lem:stochastic_domination_coupling}
in the Appendix) of $\{ B_t \}_t$ with the execution of the algorithm such
that, at every time step $t$, it holds with probability 1 that $|Q_t \cup R_t
\cup D_t | \geq n/k$ or that $|Q_t| \geq  B_{2t}$.

Now we proceed with the proofs of the two claims of the lemma. As for the first claim, it follows from Lemma \ref{lm:bp.lowerbound}. Indeed, since $W$ is a bounded random variable (i.e. $W \leq 2L+1$), it has finite variance. Therefore, for Lemma \ref{lm:bp.lowerbound}, we have that there exists positive constants $\gamma$, $\beta$ and $\beta'$ (depending on $\varepsilon$) such that, if we indicate with $t=\beta \log n$,
\begin{equation*}
    \Prc{|Q_{t}|\geq \beta \log n \text{ or } |Q_{t}\cup R_t \cup D_t|\geq n/k}\geq \Prc{B_{2t}\geq  \beta \log n}\geq \gamma,
\end{equation*}
and the first claim follows from the equation above.

As for the second claim, it follows by Lemma \ref{lem:bp:growingfromlognton}. Indeed, the random variables $W$ are finite, since $W \leq 2L+1$ and then, we have that there exists positive constants $c$ and $\beta$ (which depends on $\varepsilon$, we can take $\beta$ as the maximum with the previous constant) such that, if we indicate with $t=cn$
\begin{equation*}
    \Prc{|Q_{t}|>0 \text{ or } |Q_t \cup R_t \cup D_t| \geq n/k \mid |Q_{1}|\geq \beta \log n} \geq \Prc{B_{2t}>0 \mid B_1 \geq \beta \log n}\geq 1-\frac{1}{n}.
\end{equation*}
The second claim follows from the fact that, if we are in the case in which $Q_t>0$, we obviously have that $|R_t|\geq t$, since in each iteration of the while loop in the algorithm at least one node is added to $R_t$. Instead, in the case in which $|Q_t \cup R_t \cup D_t|\geq n/k$, we can claim that $|Q_t \cup R_t | \geq n/(4k)$. Indeed, while the visit is running, at least one node is added to $R_t$ at each step and at most one node is added to $D_t$, so $|R_t| \geq |D_t|-|D_0|$. So, since $|D_0| \leq 3\log^4 n$, for a sufficiently large $n$
\[|Q_t \cup R_t \cup (D_t \setminus D_0)| \geq \frac{n}{2k}\]
and this implies that
\[|Q_t \cup R_t| \geq \frac{n}{4k}.\]
 \end{proof}
 
 \subsection{Wrapping up: proof of Claim 1 of Theorem \ref{thm:main_matching}}
 
 The proof proceeds on the same lines as the proof in Subsection \ref{ssec:proof:thm:intro-erdos-2}. Indeed, we want to prove that w.h.p. in $G_p$ there exists a giant component with $\Omega(n)$ nodes. In order to do so, we introduce the following algorithm, which is divided into two phases.

\begin{algorithm}
\caption{\textsc{Search of the giant component}}
\small{
\textbf{Input}: A small-world graph $G_{\text{SW}} = (V, E_{\text{SW}})$; a
subgraph $H$ of $G_{\text{SW}}$; two integers $\beta$, $\beta'$.
\begin{algorithmic}[1]
\State $Q = \emptyset$
\State $D = \emptyset$
\State $R = \emptyset$
\While{$|Q| \leq \beta \log n$ or $|Q \cup R \cup D| \leq n/k$}\label{line:bootstrap:matching:while}\Comment{\textsc{First phase}}
	\State Let $s \in V \setminus D$
	\State $Q = \{s\}$
	\State $R = \emptyset$
	\For{$i=1,\dots,\beta'\log n$}
	    \State Perform lines \ref{line:randommatching:while}-\ref{line:randommatching:endwhile} of Algorithm \ref{alg:L-visit-var-matching}\Comment{Sequential $L$-visit}
 	 \EndFor
 	 \State $D = D \cup R$
\EndWhile
\If{$|Q \cup R \cup D| \leq n/k$}\label{line:bootstrap:matching:secondphase}\Comment{\textsc{Second phase}}
\State Perform Algorithm \ref{alg:L-visit-var-matching} with input $G_{SW}$, $H$, $I_0=Q$, $D_0=R\cup D$  \Comment{Sequential $L$-visit}
\EndIf
\end{algorithmic}}
\label{alg:bootstrap_randommatching}
\end{algorithm}
In the first phase is a ``bootstrap'' where we search for a node in the giant component: we look for a source node $s$ such that, after $\Theta(\log n)$ steps of the sequential $L$-visit, the queue $Q$ has $\Omega(\log n)$ nodes. The second phase consists of the sequential $L$-visit starting from the queue $Q$ with $\Omega(\log n)$ nodes. We note that in this case, we have not performed the analysis via the parallel visit: this is because the random variables describing the process, in this case, assume dependencies that are difficult to handle in the parallel case.
 
The following lemma contains the analysis of Algorithm \ref{alg:bootstrap_randommatching} with in input $(G,G_p,\beta,\beta')$, where $\beta$ and $\beta'$ are two positive constants. It claims that the first phase of the algorithm, that begins in the while loop in line \ref{line:bootstrap:matching:while}, ends after $\tau_1=\bigO(\log n)$ iterations of the while loop w.h.p. Moreover, the lemma claims that the in second phase of the algorithm (starting in line \ref{line:bootstrap:matching:secondphase}) the visit reach $\Omega(n)$ nodes.

\begin{lemma}
Let $V$ be a set of $n$ nodes. For every $\varepsilon>0$ and $\beta>0$, and for every probability $p$ such that $1/2+\varepsilon \leq p \leq 1$, there are positive parameters $L$, $k$ (that depends only on $\varepsilon$)
and a constant $\beta'$ (that depends on $\varepsilon$ and $\beta$) such that the following holds. Sample a graph $G=(V,E)$ according to the \SWGreg distribution and let $G_p$ be the percolation graph of $G$ with percolation probability $p$. Run the \textsc{Search of the giant component} in Algorithm \ref{alg:bootstrap_randommatching} on input $(G,G_p,\beta,\beta')$: if $n$ is sufficiently large,
\begin{enumerate}
    \item the first while loop in line \ref{line:bootstrap:matching:while} terminates at some round $\tau_1 = \bigO(\log n)$ w.h.p.;
    \item conditioning on the above event, the second phase of the algorithm starting in line \ref{line:bootstrap:matching:secondphase} terminates at some round $\tau_2 =\bigO(n)$ in which, w.h.p. $|Q \cup R| \geq n/(4k)$.
\end{enumerate}
\end{lemma}

\begin{proof}
The proof proceeds similarly to the proof of Lemma \ref{lem:bootstrap} in Section \ref{sec:wup}. In particular, it follows from Lemma \ref{lem:randommatching:analysis}. 
\end{proof}

We notice that Claim 1 of Theorem \ref{thm:main_matching} follows from the lemma above.
\section{Regular Graphs Below the Threshold} \label{sec:reg-gen}
In this section, we analyze the percolation process in regular graphs and prove Theorem~\ref{th:reg.upper}. As for the $3$-regular graphs generated by the \SWGreg\ model, we observe that Claim~$2$ of Theorem~\ref{thm:main_matching} is a direct consequence of the general result proved in this section.  

\begin{theorem}\label{thm:reg-infection-stops}
Let $G = (V,E)$ be a graph of maximum degree $d$ and let $s$ be a vertex in
$V$.  If $p = (1-\varepsilon)/(d-1)$ for some $\varepsilon$ such that $1 > \varepsilon > 0$, then the
probability that the connected component of $s$ in the percolation graph $G_p$
of $G$ has size $> t$ is at most $\exp(-\Omega(\varepsilon^2 t))$. Furthermore,
w.h.p., all connected components of $G_p$ have size $\bigO(\varepsilon^{-2}
\log n)$.
\end{theorem}

\begin{proof}
We consider an execution of the BFS algorithm below with the percolation graph $G_p = (V,E_p)$ of $G = (V,E)$ as input $H$ and any fixed source $s \in V$\footnote{We notice that, unlike the other visiting procedures we used in the previous sections, this algorithm does not require to have the graph $G$ as input, only its percolation.}.

\begin{algorithm}[H]
\caption{\textsc{BFS visit}}
\small{
\textbf{Input}: a graph $H=(V,E_H)$ and a  source $s \in V$. \\

    \begin{algorithmic}[1]
		\State $Q = \{s\}$
		\While{$Q \neq \emptyset $}\label{line:bfs:while}
			\State $w = \dequeue(Q)$
			\State $\texttt{visited}(w) = \True$ 
					\For {each neighbor $x$ of $w$ in $H$ such that $\texttt{visited}(x) = \False$}\label{line:bfs:for}
		
						 \State $\enqueue(y,Q)$
					\EndFor\label{line:bfs:endfor}
		 \EndWhile
	\end{algorithmic}}
\end{algorithm}

\noindent We first notice that if the above algorithm visits more than $t$
vertices, then it executes the while loop more than $t$ times. Consider what has happened after the $t$-th iteration of the while loop. Each iteration removes one node from the queue, the queue is not empty, and initially the queue held one node, so we have to conclude that we added at least $t$ nodes to the queue in the first $t$ iterations of the while loop. Consider how many times we run the for cycle in lines \ref{line:bfs:for}-\ref{line:bfs:endfor} in each while-loop iteration and assume by deferred decision that we make the choice about the edge $(w,x) \in E_p$
only when   line \ref{line:bfs:for} is executed.  That cycle  is  executed at most $d$ times at the first iteration, and at most $d-1$ times subsequently (because every vertex in the queue has at most $d-1$ non-visited neighbors in $G_p$) and so it is executed at most $t \cdot (d-1) + 1$ times. Each time it is executed implies that the event $(w,x) \in E_p$ holds and it has probability $p$, independent of everything else.

From the above argument, it follows that  we have observed at most $t\cdot (d-1)+1$ Bernoulli random variables with parameter $p = (1-\varepsilon)/(d-1)$ and we found that at least $t$ of them were $1$. By Chernoff bounds this happens with probability $\exp(-\varepsilon^2 t/3)$.

As for the ``Furthermore'' part, it suffices to choose a real $b$ large enough so that the probability that the connected component of $s$ has size more than $b$ is at most $1-1/n^2$, then take a union bound over all source  $s$.
\end{proof}

\section{Generalizations and Outlook} \label{sec:further}

As discussed in the introduction, our goal was to investigate the simplest models that can at least qualitatively capture essential aspects of epidemic processes observed in realistic scenarios. On the other hand, we believe some variants and generalizations  of the models  we considered in this paper deserve a rigorous study. Two natural directions concern extensions to the IC protocol we considered and more general models of the underlying network topology.

\subsection{Non-homogenous  bond percolation probabilities} \label{ssec:diff-probs}

A possible extension of the considered small-world models is to introduce two different  bond percolation (i.e., transmission) probabilities, each one assigned to each type of connection.
Formally, given a small-world graph $G = (V, E_1 \cup E_2)$, the percolation graph $G_{p_1, p_2}$ is the result of the following process: a bond percolation   with probability $p_1$ is applied on $G_1=(V, E_1)$; and  a bond percolation with probability $p_2$ is applied to    $G_2 = (V, E_2)$;   finally, we get the  union of the two resulting random subgraphs, denoted as  $G_{p_1, p_2}$.

With  $RF(p_1, p_2)$, we refer to the corresponding generalization of the RF protocol  considered in this paper.

Since  $p_1$ is the percolation probability of the local edges, it is immediate the following result.

\begin{fact} 
Let  $G=(V,E)$ be  a one-dimensional small-world graph and $p_1, p_2$ be, respectively, the percolation
probabilities on the local and bridge edges. Then,  for each node $v \in V$, the size $ \AN^L(v)$ of its $L$-truncated
local cluster $\LC ^L (v)$ satisfies the following
\begin{equation} \label{eq:truncated_LC_extension}
    \Expcc{| \AN^L(v)|} =\frac{1+p_1}{1-p_1}-\frac{2p_1^{L+1}}{1-p_1}\,.
\end{equation}
\label{lemma:local_cluster_two_prob}
\end{fact}

%where the probability is over both the randomness of the choice of $G$ from $\mathcal{SWG}(n,c/n)$ and over the choice of the percolation graph $G_p$.

Our analysis for the homogenous case easily extends to the above non-homogenous setting: the next theorem formalizes the main result in terms of epidemic protocols.

\begin{theorem}[The $RF(p_1, p_2)$ protocol  on the \SWG model] \label{thm:gen_overthershold_two_prob}
Let $V$ be a set of $n$ vertices, $I_0 \subseteq V$ be a set of source nodes, and $p_1, p_2 \geq 0$ two constant probabilities. For any constant
$c>0$,  sample a graph $G=(V, E_1 \cup E_2)$ from the $\SWGm(n,c/n)$ distribution, and run the $RF(p_1, p_2)$ protocol with transmission probabilities $p_1$ over $G_1 = (V, E_1)$ and $p_2$ over $G_2 = (V, E_2)$ from $I_0$. For every $\varepsilon>0$, we have the following:

\begin{enumerate}
\item If $p_1 + c \cdot p_1 p_2 + c \cdot p_2 \geq 1 + \varepsilon$, then, with
probability $\Omega_\varepsilon(1)$ a subset of  $\Omega_\varepsilon(n)$ nodes  will be informed within time $\bigO_\varepsilon(\log n)$, even if $|I_0| = 1$. Moreover, if  
$|I_0|\geq\beta_\varepsilon\log n$ for a sufficiently large constant
$\beta_\varepsilon$ (that depends only on $\varepsilon$), 
then the above event occurs  \emph{w.h.p.};
    
\item If $p_1 + c \cdot p_1 p_2 + c \cdot p_2 \leq 1 - \varepsilon$, then w.h.p.
  the total  number of informed nodes will be  $\bigO_\varepsilon(|I_0|\log n)$, and the protocol will stop within $\bigO_\varepsilon(\log n)$ time.
\end{enumerate}
\end{theorem}

As observed above, it is possible to easily recover the proof of the above theorem from the analysis of the homogeneous case we described in the previous sections. In the following two subsections, we thus only describe how the main technical statements changes in this non-homogeneous case.

\paragraph*{Proof of Claim I of Theorem \ref{thm:gen_overthershold_two_prob}}

We first consider  Algorithm \ref{alg:L-visit}, recalling  the notion of    \good node in    Definition \ref{def:free_node}, and    generalize   Lemma \ref{le:L-visit}.

\begin{lemma}\label{le:L-visit_two_prob}
Let $V$ be a set of $n$ nodes, $s \in V$ an \textit{initiator} node and $D_0 \subseteq V \setminus \{s\}$ a set of deleted nodes such that $|D_0 | \leq \log^4 n$.
For every $\varepsilon >0$ and $c>0$, and for every   probabilities   $p_1, p_2$
such that
 \[p_1 + c \cdot p_1 p_2 + c \cdot p_2 \geq 1 + \varepsilon,\]
there are positive parameters $L,k,t_0,\varepsilon'$, and $\gamma$, that depend
only on $c$ and $\varepsilon$, such that the following holds.  Sample a graph
$G = (V,E)$ according to the $\SWGm(n, c/n)$ distribution and let $G_{p_1, p_2}$ be the
percolation graph of $G$ with percolation probability $p_1, p_2$. Run
Algorithm~\ref{alg:L-visit} on input $(G,G_{p_1, p_2},s,D_0)$: if $n$ is sufficiently large,
for every $t$ larger than $t_0$, at the end of the $t$-th iteration of the
while loop it holds that 
\[
\Prc{|R \cup Q| \geq n/k \mbox{ \emph{OR} } |Q| \geq \varepsilon' t} \geq
\gamma \, .
\]
\end{lemma}
 
Lemma \ref{le:L-visit_two_prob} implies that the nodes visited by the end of Algorithm~\ref{alg:L-visit} reach a size at least $n/k$, with a probability of at least $\gamma$. The consequence is a linear lower bound on the size of the connected component of the source $s$ in $G_p$.

As we made for the homogenous case,
the next goal is to show that if we explore the connected components of $\log n$ nodes taken arbitrarily in the graph, then \emph{w.h.p.}, we visit a linear fraction of the nodes in the percolated graph, within $\Theta(\log n)$ number of hops. To do this, we analyze the execution of Algorithm~\ref{alg:par_L-visit} on input $(G, G_{p_1, p_2}, I_0,D_0)$, where $I_0$ is an arbitrary subset of initiators and $D_0 \subseteq V \setminus I_0$ is a set of deleted nodes. 
Recall that $S_t = V \setminus ( R_t\cup Q_t)$, where $Q_t$ and $R_t$ respectively are the subsets $Q$ and $R$ at the end of the $t$-iteration of the while loop in
line~$10$. We can thus state the new version  Lemma \ref{thm:exponential_growth}.

\begin{lemma}\label{thm:exponential_growth_two_prob}
Let $V$ be a set of $n$ nodes, $I_0 \subseteq V$ a set of initiators and $D_0 \subseteq V \setminus I_0$ a set of deleted nodes such that $|D_0| \leq \log^4 n$. For every
$\varepsilon>0$, $c>0$ and for every contagion probabilities $p_1, p_2$ such that
\[
 p_1 + c \cdot p_1 p_2 + c \cdot p_2 \geq 1 + \varepsilon
\]
there are positive parameters $L,k,\beta,\delta$ that depend only on $c$ and
$\varepsilon$ such that the following holds. Sample a graph $G=(V,E)$ according
to the \SWG distribution, and let $G_{p_1, p_2}$ be the percolation
graph of $G$ with parameters $p_1, p_2$. Run
Algorithm~\ref{alg:par_L-visit} on input $(G,G_{p_1, p_2},I_0,D_0)$: in every iteration $t
\geq 1$ of the while loop at line \ref{line:parvisit:while} in Algorithm~\ref{alg:par_L-visit}, for
every integer $i \geq \beta \log n$ and $r \geq 0$ such that $i+r\leq n/k$:

\begin{equation}
    \Prc{|{Q}_{t}|\geq (1+\delta)i \mid \condpar}\geq 1-\frac{1}{n^2} \, 
    .
    \label{eq:lem:exp_growth_two_prob}
\end{equation}
\end{lemma}

At this point, we run  Algorithm~\ref{alg:union_L-visit} in the non-homogeneous framework and get the new version of \ref{thm:bootstrap}.

\begin{lemma}
\label{thm:bootstrap_two_prob}
Let $V$ be a set of $n$ nodes and $I_0 \in V$ a set of \textit{initiators}.
For every $\varepsilon>0$, $c>0$ and for every contagion
probabilities $p_1, p_2$ such that 
\[
 p_1 + c \cdot p_1 p_2 + c \cdot p_2 \geq 1 + \varepsilon
\]
there are positive parameters $L,k,\gamma,\beta$ that depend only on $c$ and
$\varepsilon$ such that the following holds. Sample a graph $G = (V,E)$
according to the $\mathcal{SWG}(n,c/n)$ distribution, and let $G_{p_1, p_2}$ be the
percolation graph of $G$ with parameters $p_1, p_2$. Run
Algorithm~\ref{alg:union_L-visit} on input $(G,G_{p_1, p_2},I_0)$ and, if $n$ is sufficiently
large:
\begin{enumerate}
    \item The first while loop in line \ref{line:totvisit:firstwhile} terminates at
some round $\tau_1 =\Theta(\log n)$ in which \[\Prc{|R \cup Q| \geq n/k \mbox{
\emph{OR} } |Q| \geq \beta \log n} \geq \gamma;\]
\item Conditioning at the above event, the second while loop in line \ref{line:totvisit:secondwhile} terminates at some round $\tau_2 = \Theta(\log n)$ in which, w.h.p. $|Q \cup R| \geq n/k$.
\end{enumerate} 
\end{lemma}

As in the homogeneous case, this result implies that, starting from a single source $s \in V$, the algorithm will visit $\Omega(n)$ nodes with constant probability. On the other hand, starting from a set of sources $I_0$ such that $|I_0|=\Omega(\log n)$, the algorithm will reach $\Omega(n)$ nodes, \emph{w.h.p}. This concludes the proof of  Claim 1 of Theorem \ref{thm:gen_overthershold_two_prob}.

\paragraph*{Proof  of Claim II of Theorem \ref{thm:gen_overthershold_two_prob}}
To prove Claim 2 of Theorem \ref{thm:gen_overthershold_two_prob}, we  state Lemma \ref{thm:below_th} to the two-probabilities case.
\begin{lemma} \label{thm:below_th_two_prob}
Let $V$ be a set of $n$ nodes. For every $\varepsilon>0$, $c>0$ and for every
transmission probabilities $p_1, p_2 > 0$ such that
\[
 p_1 + c \cdot p_1 p_2 + c \cdot p_2 \leq 1 - \varepsilon
\]
there is a positive constant $\beta$ that depends only on $c$ and $\varepsilon$
such that the following holds. Sample a graph $G=(V,E)$ according to the
\SWG distribution, and let $G_{p_1, p_2}$ be the percolation graph of
$G$ with parameters $p_1, p_2$. If $n$ is sufficiently large, with probability at least
$1 - 1/n$ contains no connected component of size exceeding $\beta \log n$.
\end{lemma}
The proof is a simple generalization of the proof of Lemma \ref{thm:below_th}. In particular,  we need the following   version of  Lemma \ref{le:ch_geom} that considers a Galton-Watson process $\{B_t\}_{t\geq 0}$ with $W_1, W_2,\ldots$, defined by extending Definition \ref{def:zeta} to the two-probabilities case.

\begin{lemma}\label{le:ch_geom_two_prob}
Let $\{B_t\}_{t\geq 0}$ be the Galton-Watson process described above. For any $t>0$, 
\begin{align}
\label{eq:le:ch_geom_two_prob}
	&\Prc{\sum_{i=1}^t W_i \ge (1 + 2\delta) \frac{p_2(1+p_1)}{1-p_1}ct}\le 2e^{-\frac{\delta^2\min{\{p_1, p_2\}}^3ct}{9}}
\end{align}
\end{lemma}

The above bound is obtained as follows. We recall that:  $Y = \sum_{i=1}^t Y_i$, with each $Y_i$ being an (independent) Binomial
variable with distribution $\text{Bin}(n,p_2c/n)$; each $L_j$ is an (independent) variable that counts the number of successes until the first failure with success
probability $p_1$; and each $\check{B}_i$ is an (independent) Bernoulli random variable with \textit{success} probability $1-p_1$.
Proceeding as in the homogeneous case, we get

\[\Prc{Y \geq (1+ \delta) pct}\le e^{-\frac{\delta^2}{3}p_2ct} \]
and 
\begin{align*}
    & \Prc{\sum_{j=1}^{2(1+\delta)p_2ct}L_j > (1 + 2\delta)\frac{2p_1p_2}{1-p_1}ct} = \\
    &= \Prc{\sum_{i=1}^{2(1+\delta)p_2ct+(1 + 2\delta)\frac{2p_1p_2}{1-p_1}ct}\check{B}_i < 2(1+\delta)p_2ct} \leq e^{-\frac{\delta^2 
	p_1^3ct}{9}}
\end{align*} 

At this point, generalization follows easily.

As in the homogeneous case, we exploit the full equivalence between the bond percolation process and the IC-process: so, Lemma \ref{thm:below_th_two_prob} also implies Claim 2 of Theorem \ref{thm:gen_overthershold}. In fact, since the IC process infects at least one new node in each round unless it has died out, Lemma \ref{thm:below_th_two_prob} also implies that, \emph{w.h.p.}, the IC process dies out within $\beta\log n$ rounds, infecting at most $|I_0| \cdot \beta\log n$ new nodes.

\subsection{Non-unit activity periods.}
In the previous sections, we assumed that each infectious node has one single chance to infect its neighborhood in the step immediately following the one in which it became infected. Natural generalizations include models where the interval of time during which a node is infectious follows some distribution. While this can considerably complicate the analysis, our approach straightforwardly extends to a simple generalization, in which the activity period of a node consists of $k$ consecutive units of time, where $k$ is a fixed constant. In this case, the corresponding versions of the epidemic models we considered in this paper can be easily formalized as follows.

\begin{definition}[\IC and \SIR models with $k$ attempts] \label{def:independent_cascade_attempts}
Given a graph $G=(V,E)$, an assignment of \emph{contagion probabilities} $\{
p(e) \}_{e\in E}$ to the edges of $G$, and a non-empty set $I_0 \subseteq V$
of initially infectious vertices (that will also be called {\em initiators} or
{\em sources}), the {\em Independent Cascade} (for short, \IC) protocol \emph{with
$k$ attempts} is the  stochastic process $\{S_t,I_t,R_t\}_{t\geq 0}$, where
$S_t,I_t,R_t$ are three sets of vertices, respectively called susceptible,
infectious, and recovered, which form a partition of $V$ and that are defined
as follows. Let $\hat{I}_t\subseteq I_t$ be the subset of  those  nodes which
receive the infection for the first time at step $t$.

\begin{itemize}
     
     \item At time $t=0$ we have $R_0 = \emptyset$ and $S_0 = V-I_0$. We set
     $\hat{I}_{-k}=\dots=\hat{I}_0=\emptyset$.
     
     \item At time $t\geq 1$:
     \begin{itemize}
         
         \item $R_t = R_{t-1} \cup \hat{I}_{t-k}$, that is, the nodes that got
         infected $k$ steps before become recovered.
         
         \item Independently from the previous steps, for each edge $e=\{
         u,v\}$ such that $u\in I_{t-1}$ and $v\in S_{t-1}$, with probability
         $p(e)$ the event that ``$u$ transmits the infection to $v$ at time
         $t$'' takes place. The set $\hat{I}_t$ is the set of all vertices
         $v\in S_{t-1}$ such that, for at least one neighbor $u\in I_{t-1}$,
         the event   ``$u$ transmits the infection to $v$'' takes place. We set
         $I_t=(I_{t-1}\cup \hat{I}_t)\setminus \hat{I}_{t-k}$.
         
         \item $S_t = S_{t-1} - \hat{I}_t$
     \end{itemize}
 \end{itemize}
The process stabilizes when $I_t = \emptyset$.
\end{definition}

We recall that the RF (\SIR) protocol is the special case of the \IC protocol in which all probabilities are the same. To analyze the process described above, we use the following result, which is a direct consequence of Theorem \ref{thm:equivalence} that states the equivalence between the Independent Cascade process and the percolation process.

\begin{cor}[\cite{KKT15}] \label{cor:delays_k}
Let $R_{\infty}$ be the final set of nodes reached by the \IC process with $k$
attempts, according to above definition, on the graph $G = (V,E)$ and contagion
probabilities $\{p(e)\}_{e \in E}$. Let $\hat{R}_{\infty}$ be the final set of
nodes reached by the \IC process and on the same graph $G$,  with only one
activation and  with contagion probabilities $\{\hat{p}(e)\}_{e \in E}$, where
$\hat{p}(e)=1-(1-p(e))^k$. Then, $R_\infty$ and $\hat{R}_\infty$ have the same
distribution.
\end{cor}
Thanks to the above equivalence result, our results stated in Theorems~\ref{thm:gen_overthershold} and~\ref{thm:main_matching} (and the general result in Theorem~\ref{th:reg.upper}) can be easily generalized to the \SIR model with $k$ activations by setting the contagion probability to the value $\hat{p}=1-(1-p)^k$. As for the convergence time of the \SIR process, we observe that the  $k$ consecutive attempts of every infectious node clearly result in a slow-down of (at most) an extra multiplicative factor $k$ with respect to the obtained bounds.

\paragraph*{Random incubation periods.} 
Our results easily extend to a discrete, SEIR generalization of the epidemic model studied in this paper in which every node has an associated, random {\em incubation period}. In more detail, in our discrete-time setting, each node $v$ has an associated random variable  $h(v)$, which gives the number of incubation steps after which, once infected, node $v$ becomes infectious itself. This model can be reduced to a percolation problem in which the activation of edges is as before, each node $v$ is labeled by $h(v)$, the set of  nodes reached by the infection is the set of nodes reachable from $I_0$ in the percolation graph. In this case, the time of contagion of a node $v$ is the length of the shortest path from $v$ to $I_0$ in the percolation graph, where the ``length'' of a path $P$ is the number of edges plus the sum of the incubation times of the vertices along the path.

If the incubation times $h(v)$ are independent of the randomness of the activation of edges, then incubation does not affect the number of nodes eventually reached by the infection, it only affects the time of spreading. 

Moreover, if the incubation times are also mutually independent random variables with a nice (for example, subgaussian) tail, then it is also possible to get bounds in probability for the infection spreading time.

%The epidemic models studied in this paper can be further generalized by adding to the model a non-zero {\em incubation period} for infected nodes before they start to be contagious. In our discrete-time setting, we %can associate to each node $v$
%a random variable  $h(v)$ which gives the number of incubation steps of node $v$. This model can be reduced to a percolation problem in which the activation of edges is as before, each node $v$ is labeled by $h(v)$, %the set of  nodes reached by the infection is the set of nodes reachable from $I_0$ in the percolation graph, and the time of contagion of a node $v$ is the length of the shortest path from $v$ to $I_0$, where the %``length'' of a path is the number of edges plus the sum of the incubation times of the vertices along the path.

%If the incubation times $h(v)$ are independent of the randomness of the activation of edges, then incubation does not affect the number of nodes eventually reached by the infection, and it only affect the time of %spreading. 

%If the incubation times are also mutually independent random variables with a nice (for example, subgaussian) tail, then it is also possible to get bounds in probability for the infection spreading time.

\paragraph*{Other topologies.} A second natural direction is investigating more general topologies than those considered in this paper. In this respect, natural generalizations include families of graphs used to model short connections and the random networks used to model long-range ones. As for the former, a natural extension would be considering 2-dimensional grids. Already moving to this setting poses non-trivial challenges. For example, in this case, characterizing the spread over local clusters seems considerably harder, whereas this can be done exactly in rings. As for long-range connections, it would be interesting to investigate distributions in which the existence of an edge depends on the distance between the end-points in the underlying graph of local connections. While this is a natural generalization of the setting addressed in this paper, it might prove considerably more challenging.

%%
%% Bibliography
%%

%% Please use bibtex, 

% \input{./trunk-full/appendix.tex}

\end{document}